\documentclass[11pt, a4paper]{article}

\usepackage{amssymb,amsmath, wasysym, graphicx, epsfig, setspace, wrapfig, comment}
\usepackage{latexsym,dsfont,amsfonts,amsmath,amssymb,amsthm}
\usepackage{pstricks,epsfig,psfrag,color}
\usepackage{bbm}
\usepackage[show]{ed}
\usepackage{soul}
\usepackage{mathrsfs}   
\usepackage{tikz}
\usepackage{upgreek}
\usepackage{algpseudocode,algorithm} 
\algdef{SE}[DOWHILE]{Do}{doWhile}{\algorithmicdo}[1]{\algorithmicwhile\ #1}

\usepackage{mathtools}
\usepackage{graphicx}
\usepackage{bbm}
\usepackage{subfig} 
\usepackage{amsmath,amsfonts,amssymb}  

\numberwithin{equation}{section}

\usepackage{bm}

\usepackage{float}
\usepackage{mathtools}
\usepackage{hyperref}
\usepackage[capitalize]{cleveref}
\usepackage[maxbibnames=99]{biblatex}
\usepackage{color}
\usepackage{comment}
\usepackage{tikz-cd}
\usepackage[shortlabels]{enumitem}
\usepackage{nomencl}

\newcommand{\hdi}[3]{\boldsymbol{H}_{#1}(\di^{#2},{#3})}
\newcommand{\hcu}[3]{\boldsymbol{H}_{#1}(\cur^{#2},{#3})}

\newcommand{\hdinorm}[2]{\|{#1}\|_{\hdi{}{}{{#2}}}}
\newcommand{\norm}[2]{\|{#1}\|_{#2}}
\newcommand{\lnorm}[2]{\norm{#1}{L^2({#2})}}
\newcommand{\hnorm}[2]{\norm{#1}{H^1({#2})}}

\newcommand{\di}{\operatorname{div}}
\newcommand{\cur}{\operatorname{curl}}

\newcommand{\hdivO}{\boldsymbol{H}(\di,\Omega)}

\newcommand{\dx}{\medspace d\boldsymbol{x}}

\newcommand{\cu}{\nabla\times}
\newcommand{\cuScal}{{\mathbf{curl}}}

\newcommand{\ao}{a_{\omega_i^*}}

\newcommand{\intOm}{\int_{\Omega}}
\newcommand{\ma}[1]{\textcolor{black}{#1}}

\def\Xint#1{\mathchoice
	{\XXint\displaystyle\textstyle{#1}}%
	{\XXint\textstyle\scriptstyle{#1}}%
	{\XXint\scriptstyle\scriptscriptstyle{#1}}%
	{\XXint\scriptscriptstyle\scriptscriptstyle{#1}}%
	\!\int}
\def\XXint#1#2#3{{\setbox0=\hbox{$#1{#2#3}{\int}$ }
		\vcenter{\hbox{$#2#3$ }}\kern-.6\wd0}}

\def\dashint{\Xint-}

\newtheorem{theorem}{Theorem}[section]
\newtheorem{lemma}[theorem]{Lemma}
\newtheorem{corollary}[theorem]{Corollary}
\newtheorem{proposition}[theorem]{Proposition}
\newtheorem{assumption}{Assumption A$\!\!\!$}
\theoremstyle{definition}
\newtheorem{remark}[theorem]{Remark}
\theoremstyle{plain}

\numberwithin{equation}{section}

\title{A Mixed Multiscale Spectral Generalized Finite Element Method}
\date{\today}

\makeatletter
\let\@fnsymbol\@arabic
\makeatother

\author{Christian Alber\footnotemark[1] \and Chupeng Ma\footnotemark[2]
				\and
             Robert Scheichl\footnotemark[1]
             }

\begin{document}

\maketitle

\footnotetext[1]{Institute for Applied Mathematics and Interdisciplinary Center for Scientific Computing,
		Heidelberg University, Im Neuenheimer Feld 205, Heidelberg 69120, Germany 
  (\texttt{c.alber@uni-heidelberg.de}, \texttt{r.scheichl@uni-heidelberg.de})
                           }
\footnotetext[2]{School of Sciences, Great Bay University, Songshan Lake International Innovation Entrepreneurship A5, Dongguan 523000, China
  (\texttt{chupeng.ma@gbu.edu.cn}).
                          }

\begin{abstract}
We present a multiscale mixed finite element method for solving second order elliptic equations with general $L^{\infty}$-coefficients arising from flow in highly heterogeneous porous media. Our approach is based on a multiscale spectral generalized finite element method (MS-GFEM) and exploits the superior local mass conservation properties of mixed finite elements. Following the MS-GFEM framework, optimal local approximation spaces are built for the velocity field by solving local eigenvalue problems over generalized harmonic spaces. The resulting global velocity space is then enriched suitably to ensure inf-sup stability. We develop the mixed MS-GFEM for both continuous and discrete formulations, with Raviart-Thomas based mixed finite elements underlying the discrete method.
Exponential convergence with respect to local degrees of freedom is proven at both the continuous and discrete levels.
Numerical results are presented to support the theory and to validate the proposed method. 
\end{abstract} 

\bigskip

\section{Introduction}
\label{sec:intro}

One of the simplest models in porous media flow applications is Darcy's law \cite{whitaker1986flow}, yielding second order elliptic equations with typically strongly heterogeneous coefficients. The heterogeneities of the coefficients, stemming from heterogeneous porous media formations in, e.g., oil reservoirs, may span over several non-separated length scales. This multiscale coefficient structure poses numerical challenges for flow simulations. In applications, one is interested in solving the equations at the coarse scale. However, fine-scale coefficient variations often strongly influence global flow behaviour. Hence, conventional techniques would require resolving small-scale features to obtain reliable solutions, resulting in computationally intractable linear systems. To address this challenge, various approaches, including upscaling and multiscale methods, have been developed. Here, we focus on multiscale methods based on multiscale basis functions. The central idea of such methods is to construct coarse trial spaces consisting of problem adapted local basis functions that encode the fine-scale information. Additionally, in porous media flow applications, the Darcy velocity is the primary variable of interest and it is preferable for the computed Darcy velocity to be locally mass-conservative \cite{durlofsky1994accuracy}, motivating the use of mixed methods. 

Devising multiscale methods for Darcy's problems in mixed formulation is an active field of research. 
A pioneering work in the context of the Multiscale Finite Element Method (MsFEM \cite{hou1997multiscale}) is reported in \cite{chen2003mixed}, where local Neumann boundary value problems on coarse grid blocks are solved to construct multiscale FE bases for the velocity and the pressure is approximated by piecewise constant basis functions. Variants and further developments of the mixed MsFEM are presented in \cite{aarnes2004use,aarnes2006hierarchical,arbogast2007multiscale}, and similar coarse approximation spaces are adopted by multiscale methods \cite{arbogast2006subgrid,larson2009mixed,maalqvist2011multiscale} based on the Variational Multiscale Method (VMS \cite{hughes1995multiscale}). In \cite{chung2015mixed,chung2018constraint,chen2016least}, multiscale methods for flow equations in the mixed formulation are developed within the framework of the Generalized Multiscale Finite Element Method (GMsFEM \cite{efendiev2013generalized}). There, the multiscale basis functions are obtained by solving eigenproblems on a snapshot space. Another mixed multiscale method based on the Localized Orthogonal Decomposition (LOD \cite{maalqvist2014localization}) technique is proposed in \cite{hellman2016multiscale}. For more studies on this subject, see \cite{duran2019multiscale,wang2021comparison,yang2019multiscale,he2021generalized,guiraldello2018multiscale} and the references therein. We note that while there is an extensive literature on mixed multiscale methods, rigorous error estimates of these methods are scarce, especially for general rough coefficients. Earlier works in this direction were typically performed under the assumption of periodicity \cite{chen2003mixed,arbogast2006subgrid}. The error estimates in \cite{chung2015mixed,chung2018constraint,chen2016least,hellman2016multiscale} are presented for general coefficients without structural assumptions. However, the result in \cite{hellman2016multiscale} holds only in the two-dimensional case, and the rate of convergence in terms of the number of degrees of freedom is not derived in \cite{chung2015mixed,chung2018constraint,chen2016least}. 

In this paper, we design a novel mixed multiscale method for second order elliptic equations with general $L^{\infty}$-coefficients, based on the Multiscale Spectral Generalized Finite Element Method (MS-GFEM) first introduced in \cite{babuska2011optimal}. The MS-GFEM is designed in the framework of the GFEM \cite{melenk1996partition} with optimal local approximation spaces built from carefully designed local eigenproblems. The main idea of local approximations in MS-GFEM is to decompose the solution locally into a particular function and a 'harmonic' part. The particular function is defined to be a local solution of the underlying PDE with artificial boundary conditions, and the 'harmonic' part is approximated by means of the singular value decomposition of a compact operator. With such local approximations, the method achieves exponential convergence with respect to the number of local degrees of freedom. The main advantages of the MS-GFEM are two-fold. First, the local problems can be solved entirely in parallel. Secondly, the global coarse problem is typically of small size due to the local exponential convergence. See \cite{babuvska2020multiscale,babuvska2014machine,benezech2022scalable,ma2022error,ma2022novel,ma2023wavenumber,ma2022exponential,schleuss2022optimal} for improvements, generalizations, and practical implementations of the method. In particular, in \cite{ma2023unified}, an abstract framework applicable to various multiscale PDEs is established for the MS-GFEM and a sharper convergence rate is proved.

This paper aims to combine the superior local mass-conservation properties of mixed finite element methods with the advantages of the MS-GFEM. Following the ideas of the MS-GFEM, on each subdomain, we solve a local mixed problem accounting for the source term, and construct an optimal approximation space for solutions to the local homogeneous PDE. The local approximation space for the velocity field is built from selected eigenfunctions of a local eigenproblem over a subspace of divergence-free vector fields, while for the pressure the local approximation is obtained from the local velocity approximation via the solution of a local boundary value problem. The resulting global approximation spaces are enriched suitably to form the coarse trial spaces. We show that local approximation errors for the velocity field decay exponentially with the numbers of local basis functions as in the original MS-GFEM \cite{babuska2011optimal,ma2022novel}, and that the global error of our method (for both the velocity and pressure) is bounded by the local velocity approximation errors. Moreover, we extend the method to a discrete setting, as a model reduction method for solving mixed problems discretized with standard mixed finite elements on a fine grid, and obtain the same theoretical results. The results in this paper hold for general $L^{\infty}$-coefficients in two and three dimensions. In addition, thanks to the local particular functions, our method can straightforwardly handle fine-scale features in the problem caused by wells (modeled by the source term) in the context of hydrology and reservoir engineering, which typically require a separate treatment in other multiscale methods \cite{aarnes2004use,chen2003numerical}. 

This is the first time that the MS-GFEM is applied to multiscale problems in mixed formulation. Previous works on the method focused on elliptic problems in second-order formulation. In comparison, the design and analysis of the mixed MS-GFEM are significantly more challenging: (i) The coarse trial spaces require a careful design to ensure stability of the method. According to the Babuska--Brezzi theory, the pair of trial spaces should satisfy the inf-sup condition, which hinders us from simply using global approximation spaces as trial spaces as in the original MS-GFEM. For this purpose, guided by a careful analysis, we enrich the velocity space with the solution of a local problem for each local pressure basis function. We prove the inf-sup stability for the resulting global spaces, and demonstrate by numerical experiments that the enrichment is essential. We note that the enrichment step is inexpensive and remains fully parallel. Indeed, in each subdomain, the new local problems, after a discretization, lead to linear systems with the same matrix as that of the local problems for constructing the local pressure basis functions; they can be solved by direct methods. (ii) The proof of local exponential convergence of the mixed MS-GFEM is more tricky. In general, the key to this proof in MS-GFEM is to combine a Caccioppoli-type inequality and a weak approximation estimate -- two important properties of a generalized harmonic space associated with the problem. In the case of a second-order problem, the generalized harmonic space is a subspace of $H^{1}$, and the two properties can be proved directly using standard techniques. In the mixed case, however, the generalized harmonic space is a subspace of the space of divergence-free vector fields, and we have to lift it to the space of vector potentials for analysis. In addition, in the discrete setting, the vector potentials need to be further interpolated onto a N\'ed\'elec FE space for proving a discrete Caccioppoli inequality. The proof of exponential convergence in this setting thus needs a careful analysis of N\'ed\'elec FEs. The techniques of analysis and coarse space construction developed in this paper can thus also serve as a cornerstone for the mixed MS-GFEM in the context of other saddle point problems. 

The rest of this paper is structured as follows. In section \ref{sec:cont} the mixed MS-GFEM is introduced. We detail the construction of local approximation spaces and coarse trial spaces, and prove global approximation error estimates and quasi-optimality estimates. In section \ref{sec:discMSGFEM} we present the mixed MS-GFEM in a discrete setting with underlying Raviart-Thomas based mixed finite elements, along with discrete analogues of the results from section \ref{sec:cont}. Section \ref{sec:proofs} contains the proofs of the local exponential convergence of the mixed MS-GFEM in the continuous and discrete settings. Numerical experiments are presented in section \ref{sec:num} to evaluate the performance of the proposed method.

%
\section{Continuous mixed MS-GFEM}
\label{sec:cont}

\subsection{Mixed GFEM}
We consider the following second order elliptic equation in mixed form
\begin{equation}
	\label{dualStrongForm}
	\left\{\begin{array}{lll}
		A^{-1}\boldsymbol{u}^{\mathrm{e}} - \nabla p^{\mathrm{e}} &= 0 &\text{   in } \Omega, \\
		\operatorname{div}\boldsymbol{u}^{\mathrm{e}} &= -f  &\text{ in } \Omega, \\
		\boldsymbol{u}^{\mathrm{e}}\cdot\boldsymbol{\nu} &= g_N &\text{ on } \partial\Omega,
	\end{array}\right.
\end{equation}
where $\boldsymbol{\nu}$ denotes the outer unit normal to $\partial\Omega$. We assume that $\Omega\subset\mathbb{R}^d \,(d=2,3)$ is a bounded Lipschitz domain and that $A\in L^{\infty}(\Omega)$ satisfies
\begin{equation}
	0 < \alpha_0 < A < \alpha_1
\end{equation} 
for some constants $\alpha_0,\alpha_1\in\mathbb{R}$. Further, we suppose that $g_N\in L^2(\partial\Omega)$ and $f\in L^2(\Omega)$ satisfy the compatibility condition
\begin{equation*}
	\int_{\Omega} f\medspace d\boldsymbol{x}+ \int_{\partial\Omega}g_N \medspace ds= 0.
\end{equation*}
We consider the spaces
\begin{align}
	&\hdi{}{}{\Omega} := \{\boldsymbol{u}\in (L^2(\Omega))^d \medspace : \medspace \di\boldsymbol{u}\in L^2(\Omega)\},\nonumber \\
 &\hdi{0}{}{\Omega} := \{\boldsymbol{u} \in\hdi{}{}{\Omega}\medspace :\medspace\boldsymbol{u}\cdot\boldsymbol{\nu}= 0 \quad \text{on}\;\,\partial \Omega\}, \label{global_spaces} \\
     &\hdi{}{0}{\Omega} := \{\boldsymbol{u}\in \hdi{}{}{\Omega} \medspace : \medspace \operatorname{div}\boldsymbol{u} = 0\}, \nonumber
\end{align}
equipped with the inner product \nomenclature{$\hdi{}{}{\Omega}$}{Vector fields with divergence in $L^2(\Omega)$}
\begin{equation*}
	(\boldsymbol{u},\boldsymbol{v})_{\hdi{}{}{\Omega}} := (\boldsymbol{u},\boldsymbol{v})_{L^2(\Omega)} +(\di\boldsymbol{u}, \di\boldsymbol{v})_{L^2(\Omega)}.
\end{equation*}
We denote the norm induced by the inner product $(\cdot,\cdot)_{\hdi{}{}{\Omega}}$ by $\|\cdot\|_{\hdi{}{}{\Omega}}$.
\nomenclature{$\boldsymbol{u}\cdot\nu$}{Normal trace of $\boldsymbol{u}$}
\nomenclature{$H^{1/2}(\partial\Omega)$}{Fractional Sobolev space on boundary}
\nomenclature{$H^{-1/2}(\partial\Omega)$}{Dual space of $H^{1/2}(\partial\Omega)$}
\nomenclature{$\langle\cdot,\cdot\rangle_{\partial\Omega}$}{Duality pairing between $H^{-1/2}(\partial\Omega)$ and $H^{1/2}(\partial\Omega)$}
\nomenclature{$H^1_{\Gamma}(\Omega)$}{Sobolev funtions with vanishing trace on $\Gamma$}
\nomenclature{$\hdi{}{0}{\Omega} $}{Divergence free vector fields}
\nomenclature{$\hdi{\Gamma}{}{\Omega} $}{Vector fields with vanishing normal trace on $\Gamma$}
\nomenclature{$\hdi{0}{}{\Omega}$}{Vector fields with vanishing normal trace on $\partial\Omega$}
\nomenclature{$L_0^2(\Omega)$}{Lebesque functions with vanishing mean}
The variational formulation of (\ref{dualStrongForm}) is defined by: Find $\boldsymbol{u}^{\mathrm{e}} \in \hdi{}{}{\Omega}$ and $p^{\mathrm{e}}\in L^2_0(\Omega):=\{p\in L^{2}(\Omega):\medspace \int_{\Omega}p \dx =0 \}$ such that
\begin{equation}
	\label{dualWeakCont}
	\left\{\begin{array}{lll}
		a(\boldsymbol{u}^{\mathrm{e}},\boldsymbol{v}) + b(\boldsymbol{v},p^{\mathrm{e}}) &= 0 &\forall \boldsymbol{v}\in \hdi{0}{}{\Omega}, \\[1mm]
		b(\boldsymbol{u}^{\mathrm{e}},q) &= -(f,q)_{L^2(\Omega)}   & \forall q\in L^2(\Omega), \\[1mm]
		\boldsymbol{u}^{\mathrm{e}}\cdot\boldsymbol{\nu} &= g_N & \text{on}\;\, \partial\Omega,
	\end{array}\right.
\end{equation}
where
\begin{align}
	\label{mixedBilinearForms}
	a(\boldsymbol{u},\boldsymbol{v}) = \intOm A^{-1}\boldsymbol{u}\cdot\boldsymbol{v}\dx,
	\quad b(\boldsymbol{u},q) = \intOm q \operatorname{div}\boldsymbol{u} \dx.
\end{align}

Our goal is to use the ideas of the MS-GFEM to design an efficient method for approximating $(\boldsymbol{u}^{\mathrm{e}},p^{\mathrm{e}})$. The core is to construct multiscale approximation spaces $\boldsymbol{V}^{\mathrm{MS}}\subset \hdi{0}{}{\Omega}$ and $Q^{\mathrm{MS}}\subset L^2_0(\Omega)$, which will be used as test and trial spaces in (\ref{dualWeakCont}). To obtain an efficient and stable method, based on the general theory of mixed finite element methods, we have the following requirements for the spaces $\boldsymbol{V}^{\mathrm{MS}}$ and $Q^{\mathrm{MS}}$:
\begin{itemize}
	\item The bilinear form $a(\cdot,\cdot)$ has to be coercive on the kernel of the discrete divergence operator associated with the pair $\boldsymbol{V}^{\mathrm{MS}}/Q^{\mathrm{MS}}$.
	\item The pair $\boldsymbol{V}^{\mathrm{MS}}/Q^{\mathrm{MS}}$ has to be (uniformly) inf-sup stable with respect to $b(\cdot,\cdot)$.  
	\item The spaces $V^{\mathrm{MS}}$ and $Q^{\mathrm{MS}}$ should have good approximation properties, and be low-dimensional. 
\end{itemize}

To tackle coercivity, for $\gamma>0$, we introduce a modified variational formulation that reads: 
Find $\boldsymbol{u}^{\mathrm{e}} \in \hdivO$ and $p^{\mathrm{e}}\in L_0^2(\Omega)$ such that
\begin{equation}
	\label{globWeakAug}
	\left\{\begin{array}{lll}
		a^\gamma(\boldsymbol{u}^{\mathrm{e}},\boldsymbol{v}) + b(\boldsymbol{v},p^{\mathrm{e}}) &= -\gamma b(\boldsymbol{v},f)
		&\forall \boldsymbol{v}\in \hdi{0}{}{\Omega}, \\[1mm]
		b(\boldsymbol{u}^{\mathrm{e}},q) &= -(f,q)_{L^2(\Omega)}  & \forall q\in L^2(\Omega), \\[1mm]
		\boldsymbol{u}^{\mathrm{e}}\cdot\boldsymbol{\nu} &= g_N &\text{on } \,\partial\Omega.\\
	\end{array}\right.
\end{equation}
Here $a^{\gamma}(\boldsymbol{v},\boldsymbol{w}) := a(\boldsymbol{v},\boldsymbol{w}) + \gamma (\di\boldsymbol{v},\di\boldsymbol{w})_{L^2(\Omega)}$. Clearly, the new bilinear form $a^{\gamma}$ is globally elliptic on $\hdi{}{}{\Omega}$ and the inf-sup condition is satisfied for the pair $\hdi{0}{}{\Omega}/L^2_0(\Omega)$. Hence, problem (\ref{globWeakAug}) is well-posed by \cite[Theorem 4.2.3]{boffi2013mixed}. Moreover, it is easy to verify that problems (\ref{dualWeakCont}) and (\ref{globWeakAug}) have the same solution; see \cite[Chapter 1.5]{boffi2013mixed}. 

Let the multiscale test spaces $\boldsymbol{V}^{\mathrm{MS}}$ and $Q^{\mathrm{MS}}$ be given. Moreover, we let $\boldsymbol{u}^{\mathrm{par}}\in \hdi{}{}{\Omega}$ with $\boldsymbol{u}^{\mathrm{par}}\cdot \boldsymbol{\nu} = g_N$ on $\partial \Omega$ and $p^{\mathrm{par}}\in L^2_0(\Omega)$. We define the finite-dimensional Galerkin approximation of (\ref{globWeakAug}) by: Find $\boldsymbol{u}^{\mathrm{G}} = \boldsymbol{u}^{\mathrm{par}}+\boldsymbol{u}^{\mathrm{MS}}$ and $p^{\mathrm{G}}=p^{\mathrm{par}}+p^{\mathrm{MS}}$, with $\boldsymbol{u}^{\mathrm{MS}}\in \boldsymbol{V}^{\mathrm{MS}}$ and $p^{\mathrm{MS}}\in Q^{\mathrm{MS}}$, such that
\begin{equation}
	\label{GFEM}
	\left\{\begin{array}{lll}
		a^\gamma(\boldsymbol{u}^{\mathrm{MS}},\boldsymbol{v}) + b(\boldsymbol{v},p^{\mathrm{MS}}) &= -\gamma b(\boldsymbol{v},f) -(a^\gamma(\boldsymbol{u}^{\mathrm{par}},\boldsymbol{v}) + b(\boldsymbol{v},p^{\mathrm{par}}))
		&\forall \boldsymbol{v}\in \boldsymbol{V}^{\mathrm{MS}}, \\
		b(\boldsymbol{u}^{\mathrm{MS}},q) &= -(f,q)_{L^2(\Omega)} - b(\boldsymbol{u}^{\mathrm{par}},q)  & \forall q\in Q^{\mathrm{MS}}. \\
	\end{array}\right.
\end{equation}
The multiscale spaces and particular functions above will be constructed within the framework of the GFEM by gluing local spaces and local particular functions together. 

To achieve this, we first partition the computational domain $\Omega$ into pairwise disjoint sets $\{\omega_i^0\}_{i=1}^M$,
which are then extended to construct an overlapping decomposition $\{\omega_i\}_{i=1}^M$ of the domain. We denote by $\{\chi_i\}_{i=1}^M$ a partition of unity subordinate to this covering. For later use, we introduce a set of Lipschitz domains $\{\omega^{\ast}_i\}_{i=1}^M$ satisfying $\omega_i \subset \omega_i^*\subset\Omega$ and $\operatorname{dist}(\omega_i, \partial\omega_i^*\cap\Omega)>0$ for $i=1,\cdots,M$, which will be referred to as \textit{oversampling domains} in the rest of this paper.

For the partition of unity functions we assume
\begin{align}
	\label{PU}
	&0 \leq \chi_{i}(\boldsymbol{x}) \leq 1, \quad \sum_{i=1}^{M} \chi_{i}(\boldsymbol{x})=1, \quad \forall \boldsymbol{x} \in \Omega,\nonumber \\
	&\chi_{i}(\boldsymbol{x})=0, \quad \forall \boldsymbol{x} \in \Omega \setminus \omega_{i}, \quad i=1, \cdots, M, \\
	&\chi_{i} \in C^{1}\left(\omega_{i}\right),\quad \max _{\boldsymbol{x} \in \Omega}\left|\nabla \chi_{i}(\boldsymbol{x})\right| \leq \frac{C_{1}}{\operatorname{diam}\left(\omega_{i}\right)}, \quad i=1, \cdots, M .\nonumber
\end{align}
Moreover, we assume point-wise overlap conditions for $\{\omega_i\}_{i=1}^M$ and $\{\omega_i^*\}_{i=1}^M$:
\begin{equation}
	\label{pointwise_overlap}
	\exists\kappa, \kappa^*\in\mathbb{N}\quad \forall\boldsymbol{x}\in\Omega\quad \text{card}\{i\medspace|\medspace\boldsymbol{x}\in\omega_i\} \leq \kappa, \quad
	\text{card}\{i\medspace|\medspace\boldsymbol{x}\in\omega_i^*\} \leq \kappa^*,
\end{equation}
and define the following local spaces:
\begin{equation}\label{local_spaces}
\begin{array}{ll}
&H^{1}_{N}(\omega_i) := \{q\in H^{1}(\omega_i):\medspace q = 0\quad\text{on}\;\,\partial\omega_i\cap\partial\Omega \}, \\[2mm]
&\hdi{N}{}{\omega_i} := \{\boldsymbol{u} \in\hdi{}{}{\omega_i} :\medspace\boldsymbol{u}\cdot\boldsymbol{\nu}= 0 \quad \text{on}\;\,\partial\omega_i\cap\partial\Omega\}.
\end{array}
\end{equation}

Given local approximation spaces $\boldsymbol{S}^{\mathrm{v}}_{n_i}(\omega_i)\subset \hdi{N}{}{\omega_i}$, $S^{\mathrm{p}}_{n_i}(\omega_i^0)\subset L^{2}_{0}(\omega_i^0)$, and local velocity enrichment spaces $\boldsymbol{V}_{n_i}^{\operatorname{en}}(\omega_i^0)\subset \hdi{0}{}{\omega_i^0}$ with $\di \boldsymbol{V}_{n_i}^{\operatorname{en}}(\omega_i^0) = S^{\mathrm{p}}_{n_i}(\omega_i^0)$, we set $n=\sum_{i=1}^M n_i$ and define the global spaces 
\begin{align}
    \label{def_glob_spaces}
	&\boldsymbol{S}^{\mathrm{v}}_n(\Omega) := \Big\{ \sum_{i=1}^M \chi_i \boldsymbol{v}^i \medspace:\medspace \boldsymbol{v}^i\in \boldsymbol{S}^{\mathrm{v}}_{n_i}(\omega_i) \Big\}	
	,\quad S_n^{\mathrm{p}}(\Omega) := \Big\{ \sum_{i=1}^M {p^i} \medspace:\medspace {p^i}\in S^{\mathrm{p}}_{n_i}(\omega_i^0) \Big\},\nonumber\\
	&\boldsymbol{V}_n^{\operatorname{en}}(\Omega) := \Big\{ \sum_{i=1}^M  \boldsymbol{u}^i \medspace:\medspace \boldsymbol{u}^i\in \boldsymbol{V}_{n_i}^{\operatorname{en}}(\omega_i^0) \Big\},
\end{align}
The local spaces $\boldsymbol{S}^{\mathrm{v}}_{n_i}(\omega_i)$, $S^{\mathrm{p}}_{n_i}(\omega_i^0)$, and $\boldsymbol{V}_{n_i}^{\operatorname{en}}(\omega_i^0)$ will be defined in the next subsection. Note that we extend functions in $S^{\mathrm{p}}_{n_i}(\omega_i^0)$ and $\boldsymbol{V}_{n_i}^{\operatorname{en}}(\omega_i^0)$ by zero to $\Omega$. By construction, it is clear that $\boldsymbol{S}^{\mathrm{v}}_n(\Omega),\,\boldsymbol{V}_n^{\operatorname{en}}(\Omega)\subset \hdi{0}{}{\Omega}$ and $S_n^{\mathrm{p}}(\Omega)\subset L_{0}^{2}(\Omega)$. Similarly, given local particular functions $\boldsymbol{u}_{i}^{\mathrm{par}}\in \hdi{}{}{\omega_i}$ with $\boldsymbol{u}_{i}^{\mathrm{par}}\cdot \boldsymbol{\nu} = g_N$ on $\partial \omega_i\cap\partial \Omega$ and $p_{i}^{\mathrm{par}}\in L^{2}_{0}(\omega_i^0)$, we define the global particular functions $\boldsymbol{u}^{\mathrm{par}}$ and $p^{\mathrm{par}}$ by
\begin{align}
	\label{def_glob_par}
	\boldsymbol{u}^{\mathrm{par}} := \sum\limits_{i=1}^M \chi_i \boldsymbol{u}_i^{\mathrm{par}}, \qquad p^{\mathrm{par}} := \sum\limits_{i=1}^M p_i^{\mathrm{par}}.
\end{align} 
\begin{remark}
Note that functions in $S^{\mathrm{p}}_{n_i}(\omega_i^0)$ and the local particular functions $p_i^{\mathrm{par}}$ are only defined on the non-overlapping subdomains $\omega_i^0$, and are simply added together to form the global functions. This is reasonable since these global functions only need to be in $L^2(\Omega)$.
\end{remark}

When proving approximation results for the pressure, we will see that piecewise constant functions on $\Omega$ with respect to the partition $\{ \omega_i^0\}_{i=1}^{M}$ need to be added to the global pressure approximation space. In doing so, we need to enrich the velocity space accordingly to guarantee inf-sup stability. Let $Q_0^{\operatorname{RT}}$ be the space of functions on $\Omega$ which are piecewise constant with respect to the partition $\{ \omega_i^0\}_{i=1}^{M}$ and have vanishing mean, and let $V^{\operatorname{RT}}_0\subset \hdi{0}{}{\Omega}$ satisfy $\di V^{\operatorname{RT}}_0 = Q_0^{\mathrm{RT}}$. We define the desired multiscale approximation spaces by
\begin{align}
	\label{def_MS_spaces}
	\boldsymbol{V}^{\mathrm{MS}} 
	:= \boldsymbol{S}^{\mathrm{v}}_{n}(\Omega)	
	+ \boldsymbol{V}_{n}^{\operatorname{en}}(\Omega)
	+ \boldsymbol{V}^{\operatorname{RT}}_0
	\quad\text{and}\quad 
	Q^{\mathrm{MS}} 
	:= S^{\mathrm{p}}_{n}(\Omega)	
	+ Q^{\operatorname{RT}}_0.    
\end{align}

\begin{remark}\label{raviart-thomas-pair}
	In practice, the partition $\{\omega_i^0\}_{i=1}^M$ is often a coarse mesh consisting of simplicial or quadrilateral (hexahedral) elements and $V_0^{\mathrm{RT}}$ can be chosen to be the zeroth order Raviart-Thomas space on this coarse mesh. For more complicated partitions, we can solve local problems over the union of neighbouring subdomains to construct $V_0^{\mathrm{RT}}$.
\end{remark} 

We conclude this subsection with a few comments on the ideas behind the definition of the multiscale spaces.
	\begin{itemize}
		\item $\boldsymbol{S}^{\mathrm{v}}_n(\Omega)$ is used to approximate $\boldsymbol{u}^{\mathrm{e}}-\boldsymbol{u}^{\mathrm{par}}$. See \cref{theorem_nearly_exponential} and \cref{GlobLocErr}.
		\item $S^{\mathrm{p}}_n(\Omega)$ is used to approximate $p^{\mathrm{e}}-p^{\mathrm{par}}$. However, we need to add the space $Q_0^{\mathrm{RT}}$ so that the global pressure approximation error can be bounded by local velocity approximation errors; see \cref{globPresEr}. 
		\item  To obtain inf-sup stability, we enrich the velocity space by $\boldsymbol{V}_n^{\operatorname{en}}(\Omega)$ and $\boldsymbol{V}^{\operatorname{RT}}_0$; see \cref{IS_MS}.
	\end{itemize}

\subsection{Local approximations}
In this subsection, we detail the construction of the local particular functions and local approximation spaces. As in the MS-GFEM, the local particular functions are obtained by solving the underlying PDE on the oversampling domains with artificial boundary conditions. It turns out that the difference between the exact velocity solution and a local velocity solution solves the PDE (locally) with trivial right hand side, which motivates the definition of a generalized harmonic space. Then, the construction of the velocity approximation space boils down to identifying a low-dimensional approximation space for the generalized harmonic space. This will be achieved by the singular value decomposition of a compact operator whose singular vectors are selected to build the desired space. Given the velocity approximation spaces, we are able to define the local pressure approximation spaces by solving suitable local problems. All this can be constructed entirely in parallel without any communication.
\subsubsection{Local particular functions}
\label{subsec_locPart}

For arbitrary subsets $\omega\subset \Omega$, we define 
\begin{equation*}
	a_{\omega}(\boldsymbol{u},\boldsymbol{v}):= \int_{\omega} A^{-1}\boldsymbol{u}\cdot\boldsymbol{v} \dx, \qquad \|\boldsymbol{u}\|_{L^2(\omega;a)}:=\sqrt{a_{\omega}(\boldsymbol{u},\boldsymbol{u})},\qquad
	b_{\omega}(\boldsymbol{v},q) := \int_{\omega}q \di \boldsymbol{v} \dx.
\end{equation*}
To define the local particular functions, we recall the oversampling domains $\{\omega_i^{\ast}\}$ and consider the problem of finding ${\boldsymbol{{\psi}}_i} \in \hdi{}{}{\omega_i^*}$ and $ \phi_i \in L_0^2(\omega_i^*)$ such
that 
\begin{equation}
	\label{partWeakNeu}
	\left\{\begin{array}{lll}
		a_{\omega_i^*}({\boldsymbol{{\psi}}_i},\boldsymbol{v}) +b_{\omega_i^*}(\boldsymbol{v},\phi_i) &= 0 &\forall \boldsymbol{v}\in\hdi{0}{}{\omega_i^*}, \\
		b_{\omega_i^*}({\boldsymbol{{\psi}}_i},q) &= -(f,q)_{L^2(\omega_i^*)}  &\forall q \in L^2( \omega_i^*), \\
		{\boldsymbol{{\psi}}_i}\cdot\boldsymbol{\nu} &= g_N &\text{ on } \partial\Omega\cap \partial\omega_i^*,\\
		{\boldsymbol{{\psi}}_i}\cdot\boldsymbol{\nu} &= C_{comp} &\text{ on } \Omega\cap \partial\omega_i^*,
	\end{array}\right.
\end{equation}
where $C_{comp}=- \frac{1}{|\Omega\cap\partial\omega_i^*|}\left(\int_{\partial\Omega\cap\partial\omega_i^*}g_N\medspace ds +\int_{\omega_i^*}f \medspace d\boldsymbol{x}\right)$. Here we impose constant normal flux on the interior boundary such that the compatibility condition is satisfied. We define the local particular vector field and local particular pressure by 
\begin{equation}
	\label{def_loc_part}
	\boldsymbol{u}_i^{\mathrm{par}} := {\boldsymbol{{\psi}}_i}|_{\omega_i} 
	,\qquad
	p^{\mathrm{par}}_i :=
	\phi_i|_{\omega^0_i} - \dashint_{\omega_i^0}\phi_i\dx,
\end{equation}
where $\dashint_{\omega_i^0}\phi \dx$ denotes the mean value of the function $\phi$ on $\omega_i^0$.

By using appropriate test functions in (\ref{partWeakNeu}) and (\ref{dualWeakCont}) and recalling (\ref{local_spaces}), we deduce that $\boldsymbol{u}^{\mathrm{e}}|_{\omega_i^*} - \boldsymbol{{\psi}}_i$ lies in the \textbf{generalized $a$-harmonic space}:
\begin{equation*}
	\boldsymbol{H}_a(\omega_i^*) := \big\{\boldsymbol{u}\in\hdi{N}{0}{\omega_i^*}\medspace: \medspace\ao(\boldsymbol{u},\boldsymbol{v}) = 0 \quad  \forall\boldsymbol{v}\in\hdi{0}{0}{\omega_i^*}   \big\},
\end{equation*}
where $\hdi{N}{0}{\omega_i^*}$ and $\hdi{0}{0}{\omega_i^*}$ are defined following (\ref{global_spaces}) and (\ref{local_spaces}).
\begin{remark}
In problem (\ref{partWeakNeu}), we could also impose a zero boundary condition on $\Omega\cap\partial\omega_i^*$ for the pressure, instead of the boundary condition for the normal flux. The problem would then be to find $\boldsymbol{{\psi}}_i\in \hdi{}{}{\omega_i^*}$ and $\phi_i\in L^2(\omega_i^*)$ such that
	\begin{equation}
		\label{partWeakDir}
		\left\{\begin{array}{lll}
			a_{\omega_i^*}({\boldsymbol{{\psi}}_i},\boldsymbol{v}) +b_{\omega_i^*}(\boldsymbol{v},\phi_i) &= 0 &\forall \boldsymbol{v}\in\hdi{N}{}{\omega_i^*}, \\
			b_{\omega_i^*}({\boldsymbol{{\psi}}_i},q) &= -(f,q)_{L^2(\omega_i^*)}  &\forall q \in L^2( \omega_i^*), \\
			{\boldsymbol{{\psi}}_i}\cdot\boldsymbol{\nu} &= g_N &\text{ on } \partial\Omega\cap \partial\omega_i^*.
		\end{array}\right.
	\end{equation}
	Again, we see that $\boldsymbol{u}^{\mathrm{e}}|_{\omega_i^*}-\boldsymbol{{\psi}}_i\in \boldsymbol{H}_a(\omega_i^*)$. For the analysis (\ref{partWeakNeu}) is easier because we avoid technical complications of dealing with normal traces on parts of the boundary. However, (\ref{partWeakDir}) is easier for implementation since it is not necessary to compute $C_{comp}$. 
\end{remark}

\subsubsection{Local approximation spaces}
\label{sssec:local_approximation_spaces}
For convenience, we drop the subdomain index $i$ in $\omega_i,\omega_i^*$ and so forth. We have shown that the velocity solution $\boldsymbol{u}^{\mathrm{e}}|_{\omega^{\ast}}$ can be decomposed into two parts, the solution of a local problem and a vector field in $\boldsymbol{H}_a(\omega^*)$. To approximate the second part, we follow the ideas of the MS-GFEM to find optimal (finite-dimensional) spaces for approximating $\boldsymbol{H}_a(\omega^*)$ in $\omega$. To do this, let us introduce a new inner product 
\begin{equation*}
	(\boldsymbol{u},\boldsymbol{v})_{\hdi{}{}{\omega^*;a}} := (A^{-1}\boldsymbol{u},\boldsymbol{v})_{L^2(\omega^*)} + (\di\boldsymbol{u},\di\boldsymbol{v})_{L^2(\omega^*)}
\end{equation*}
for $\hdi{}{}{\omega^*}$, and denote the induced norm by $\hdinorm{\cdot}{\omega^*;a}$. Moreover, we consider the weighted $L^2$-norm
\begin{equation*} \lnorm{\cdot}{\omega^*;a}:=\lnorm{A^{-1/2}\cdot}{\omega^*}.  
\end{equation*}

A key step in finding the optimal approximation spaces is to show that the operator
\begin{equation}
	\label{defP}
	P : (\boldsymbol{H}_a(\omega^*), \hdinorm{\cdot}{\omega^*;a}) \to (\boldsymbol{H}_a(\omega), \hdinorm{\cdot}{\omega;a}), \quad \boldsymbol{v}\mapsto {\boldsymbol{v}|}_{\omega}
\end{equation}
is compact.	
For this purpose, a general framework has been provided in \cite{babuska2011optimal,ma2022novel,ma2023unified} which is based on a Caccioppoli type inequality and certain compact embedding results (e.g., from $H^1(\omega^*)$ to $L^2(\omega^*)$ in the scalar elliptic setting). However, when applying the framework to our setting, we immediately run into difficulties. First, the proof of Caccioppoli type inequalities relies on an application of the product rule for differentiation, and thus requires that the associated bilinear forms involve differential operators. In our case, however, the bilinear form $a_{\omega^{\ast}}(\cdot,\cdot)$ is just a weighted $L^2$ inner product. Secondly, functions in $\hdi{}{}{\omega^*}$ generally have much lower regularity than $H^1(\omega^*)$, and the embedding $\hdi{}{}{\omega^*} \to (L^2(\omega^*))^d$ is not compact. The idea to circumvent these issues is to make use of the fact that vector fields in $\boldsymbol{H}_a(\omega^*)$ are divergence free and, depending on the dimension $d$, go to their vector potentials or stream functions. To this end, we make a mild assumption that the oversampling domain $\omega^{\ast}$ satisfies the hypotheses on $\Omega$ in \cref{exVecPotnew}.

Let us first look at the two-dimensional case when divergence-free vector fields can be represented by stream functions (see \cite{girault2012finite} for definitions). Define 
\begin{equation*}
	\tilde{H}_a(\omega^*):= \{\phi \in H_N^1(\omega^*) \medspace:\medspace a_{\omega^*}(\cuScal\,\phi, \cuScal\, {\psi}) =0 \quad \forall {\psi}\in H_0^1(\omega^*)  \},
\end{equation*}
where 
$
	\cuScal \,\phi := (\frac{\partial\phi}{\partial x_2}, -\frac{\partial\phi}{\partial x_1}).
$ 
Then, by \cref{exVecPotnew}, $\boldsymbol{H}_a(\omega^*) = \cuScal\,\tilde{H}_a(\omega^*)$. A direct calculation shows that
\begin{equation}\label{stream-functions}
	\tilde{H}_a(\omega^*)=  H_{A^{-1}}(\omega^*):=\{\phi \in H_N^1(\omega^*) \medspace:\medspace a_{\omega^*}(\nabla\phi, \nabla {\psi}) =0 \quad \forall {\psi}\in H_0^1(\omega^*)  \}.
\end{equation}

\noindent In the three-dimensional case, we consider the spaces
\begin{equation}
\begin{array}{ll}
     & \hcu{}{}{\omega^{\ast}} := \big\{\boldsymbol{v}\in (L^2(\omega^{\ast}))^3 \medspace :\medspace \cu\boldsymbol{v}\in (L^2(\omega^{\ast}))^3\big\}, \\[2mm]
     & \hcu{0}{}{\omega^{\ast}} := \big\{\boldsymbol{v}\in\hcu{}{}{\omega^{\ast}} \medspace : \medspace \boldsymbol{v}\times \boldsymbol{\nu} = 0\quad\text{on}\,\;\partial \omega^{\ast}\big\},
\end{array}
\end{equation}
equipped with the inner product
$
	(\boldsymbol{u},\boldsymbol{v})_{\hcu{}{}{\omega^{\ast}}}:= (\boldsymbol{u},\boldsymbol{v})_{L^2(\omega^{\ast})} + (\cu\boldsymbol{u}, \cu\boldsymbol{v})_{L^2(\omega^{\ast})}.
$
Similar to $\tilde{H}_a(\omega^*)$, we define
\begin{align*}
	\boldsymbol{\tilde{H}}_a(\omega^*) := \big\{\boldsymbol{\tilde{u}}\in (H_{N}^1(\omega^*))^3  \medspace : \medspace  a_{\omega^*}(\cu\boldsymbol{\tilde{u}},\cu{\boldsymbol{\tilde{v}}}) = 0 \quad \forall{\boldsymbol{\tilde{v}}}\in\hcu{0}{}{\omega^*}
	\big\}.
\end{align*}
Again by \cref{exVecPotnew}, we see that the space $\boldsymbol{\tilde{H}}_a(\omega^*)$ is the lifting of $\boldsymbol{H}_a(\omega^*)$ to vector potentials, i.e., $\boldsymbol{H}_a(\omega^*) = \cu \boldsymbol{\tilde{H}}_a(\omega^*)$. Now we can prove Caccioppoli type inequalities for $\tilde{H}_a(\omega^*)$ and $\boldsymbol{\tilde{H}}_a(\omega^*)$, which are given in the following lemma and will play a fundamental role in the subsequent analysis. The proof is given in section~\ref{sec:proofs_nearly_continuous}.

\begin{lemma}
	\label{Caccio}
	Assume that $\eta\in W^{1,\infty}(\omega^*)$ with
	$\eta = 0$ on $\Omega\cap\partial\omega^*$.
	\begin{itemize}
		\item [(i)] Let $d=2$ and $\phi,\psi
		\in \tilde{H}_a(\omega^*)$. Then
		\begin{equation*}
			a_{\omega^{\ast}}(\cuScal (\eta \phi),\,\cuScal(\eta \psi))
			= a_{\omega^{\ast}}(\phi\cuScal\,\eta,\,\psi\cuScal\,\eta).
		\end{equation*}
		In particular,
		\begin{equation*}
			\lnorm{\cuScal(\eta\phi)}{\omega^*;a}
			\leq \norm{\nabla\eta}{L^{\infty}(\omega^*)} \lnorm{\phi}{\omega^*;a}.
		\end{equation*}
		\item [(ii)]Let $d=3$ and ${\boldsymbol{\tilde{u}}}, {\boldsymbol{\tilde{v}}}
		\in \boldsymbol{\tilde{H}}_a(\omega^*)$. Then
		\begin{equation*}
			a_{\omega^{\ast}}(\cu (\eta {\boldsymbol{\tilde{u}}}),\, \cu (\eta {\boldsymbol{\tilde{v}}}))
			= a_{\omega^{\ast}}({\boldsymbol{\tilde{u}}}\times
			\nabla\eta,\,{\boldsymbol{\tilde{v}}}\times \nabla\eta).
		\end{equation*}
		In particular,
		\begin{equation*}
			\lnorm{\cu(\eta{\boldsymbol{\tilde{u}}})}{\omega^*;a}
			\leq \norm{\nabla\eta}{L^{\infty}(\omega^*)} \lnorm{{\boldsymbol{\tilde{u}}}}{\omega^*;a}.
		\end{equation*}
	\end{itemize}
\end{lemma}

By means of Lemma~\ref{Caccio} and standard compactness results, we can prove the compactness of operator $P$.
\begin{proposition}
	The operator $P$ defined in (\ref{defP}) is compact. 
\end{proposition}
\begin{proof}
We only consider the three-dimensional case. The two-dimensional case can be proved similarly. Let $\left(\boldsymbol{u}_k\right)_k$ be a sequence in $\boldsymbol{H}_a(\omega^*)$ such that for all $k$,
	\begin{equation*}
		\hdinorm{\boldsymbol{u}_k}{\omega^*;a}
		\leq C.
	\end{equation*}
	We need to show that $\left(P\boldsymbol{u}_k\right)_k$ has a convergent subsequence in $\boldsymbol{H}_{a}(\omega)$. \Cref{exVecPotnew} implies that for every $\boldsymbol{u}_k$, there exists a 
	vector potential ${\boldsymbol{\tilde{u}}}_k \in (H^1_N(\omega^*))^3 $ such that 
	\begin{equation*}
		\cu{\boldsymbol{\tilde{u}}}_k = \boldsymbol{u}_k
	\end{equation*}
	and that there exists a constant $C_{pot}$ depending only on $\omega^*$ with
	\begin{equation*}
		\hnorm{{\boldsymbol{\tilde{u}}}_k}{\omega^*}
		\leq C_{pot}\lnorm{\boldsymbol{u}_k}{\omega^*}
		\leq C_{pot} \alpha_1^{1/2} \lnorm{\boldsymbol{u}_k}{\omega^*;a}
		\leq C_{pot} \alpha_1^{1/2}C.
	\end{equation*}
	By the Rellich-Kondrachov theorem, the sequence $({\boldsymbol{\tilde{u}}}_k)_k$ has a convergent subsequence in $(L^2(\omega^*))^3$. We still denote it by $({\boldsymbol{\tilde{u}}}_k)_k$. Now we apply \cref{Caccio} with ${\boldsymbol{\tilde{u}}}={\boldsymbol{\tilde{u}}}_k-{\boldsymbol{\tilde{u}}}_l$ (note that ${\boldsymbol{\tilde{u}}}_k\in\boldsymbol{\tilde{H}}_a(\omega^*)$, since $\boldsymbol{u}_k\in \boldsymbol{H}_a(\omega^*))$ to obtain
	\begin{equation*}
		\lnorm{\cu(\eta({\boldsymbol{\tilde{u}}}_k-{\boldsymbol{\tilde{u}}}_l))}{\omega^*;a}
		\leq \norm{\nabla\eta}{L^{\infty}(\omega^*)} \lnorm{{\boldsymbol{\tilde{u}}}_k-{\boldsymbol{\tilde{u}}}_l}{\omega^*;a}.
	\end{equation*}
	Since $\operatorname{dist}(\omega,\Omega\cap\partial\omega^*)>0$, we can choose $\eta$ such that $\eta=1$ in $\omega$ to conclude
	\begin{equation}
		\label{ineq_curl_by_l2}
		\lnorm{\cu({\boldsymbol{\tilde{u}}}_k-{\boldsymbol{\tilde{u}}}_l)}{\omega;a}
		\leq \norm{\nabla\eta}{L^{\infty}(\omega^*)} \lnorm{{\boldsymbol{\tilde{u}}}_k-{\boldsymbol{\tilde{u}}}_l}{\omega^*;a}.
	\end{equation}
Noting that $\lnorm{\cdot}{\omega^*;a}$ is a weighted $L^2(\omega^*)$-norm and $\left(\boldsymbol{\tilde{u}}_k\right)_k$ is Cauchy in $(L^2(\omega^*))^3$, we see that $(\cu {\boldsymbol{\tilde{u}}}_k)_{k}$ is a Cauchy sequence in $(L^2(\omega))^3$, thus converging to some $\boldsymbol{u}\in (L^2(\omega))^3$. Moreover, $\boldsymbol{u} \in \boldsymbol{H}_a(\omega)$ since $\cu {\boldsymbol{\tilde{u}}}_k=\boldsymbol{u}_k\in\boldsymbol{H}_a(\omega)$. Finally, since $\left(P\boldsymbol{u}_k\right)_k$ are divergence free, the $L^{2}$ convergence implies its convergence in $\boldsymbol{H}_a(\omega)$. 
\end{proof}

Next, we consider the best approximation of the image of operator $P$ by $n$-dimensional subspaces of $\boldsymbol{H}_a(\omega)$. This problem can be formulated by means of the Kolmogorov $n$-width \cite{pinkus2012n} as follows:
\begin{equation}\label{n-width}
d_{n}(\omega,\omega^{\ast}):=\inf_{Q(n)\subset \boldsymbol{H}_a(\omega)}\sup_{\boldsymbol{u}\in \boldsymbol{H}_a(\omega^{\ast})}\inf_{\boldsymbol{v}\in Q(n)} \frac{\Vert P\boldsymbol{u} - \boldsymbol{v}\Vert_{{\hdi{}{}{\omega;a}}}}{\Vert \boldsymbol{u}\Vert_{{\hdi{}{}{\omega^{\ast};a}}}},
\end{equation}
where the leftmost infimum runs over all $n$-dimensional subspaces of $\boldsymbol{H}_a(\omega)$. Since we are in a Hilbert space setting and operator $P$ is compact, we can characterize the $n$-width $d_{n}(\omega,\omega^{\ast})$ by means of the singular value decomposition of operator $P$.

\begin{lemma}
	\label{nwidth_char}
	For each $k\in\mathbb{N}$, let $\lambda_k$ and $\boldsymbol{v}_k$ be the $k$-th eigenvalue (arranged in increasing order) and the corresponding eigenfunction of the following problem
	\begin{equation}
		\label{EVP}
		a_{\omega^*}(\boldsymbol{v},\boldsymbol{\phi}) 
		= \lambda \,a_{\omega}(\boldsymbol{v}, \boldsymbol{\phi}) 
		\quad \forall \boldsymbol{\phi}\in \boldsymbol{H}_a(\omega^*).
	\end{equation}
	Then, $d_n(\omega,\omega^*) = \lambda_{n+1}^{-1/2}$, and the associated optimal approximation space is given by 
	\begin{equation}
		\mathrm{span}\{\boldsymbol{v}_1|_{\omega},\cdots,\boldsymbol{v}_n|_{\omega}\}. 
	\end{equation}
\end{lemma}
\begin{proof}
Use \cite[Theorem  2.2]{pinkus2012n} and the fact that $a_{\omega^*}(\boldsymbol{u},\boldsymbol{v})= (\boldsymbol{u},\boldsymbol{v})_{\hdi{}{}{\omega^*;a}}$ for $\boldsymbol{u},\boldsymbol{v}\in \boldsymbol{H}_a(\omega^*)$.		
\end{proof}
Now we are ready to define the desired local velocity approximation space.
\begin{proposition}
	Let the local approximation space for the velocity on $\omega$ be defined as 
	\begin{equation}
		\label{def_loc_vel}
		\boldsymbol{S}^{\mathrm{v}}_{n}(\omega) := \mathrm{span}\{\boldsymbol{v}_{1}|_{\omega}, \cdots, \boldsymbol{v}_{n}|_{\omega} \},
	\end{equation}
	where $\boldsymbol{v}_k$ denotes the $k$-th eigenfunction of the problem (\ref{EVP}). 
	Then, 
	\begin{equation}\label{local_estimate}
		\inf_{\boldsymbol{v}\in \boldsymbol{u}^{\mathrm{par}} + \boldsymbol{S}_{n}^{\mathrm{v}}(\omega)} \hdinorm{\boldsymbol{u}^{\mathrm{e}}-\boldsymbol{v}}{\omega;a} 
		\leq d_{n}(\omega,\omega^*) \hdinorm{\boldsymbol{u}^{\mathrm{e}}-\boldsymbol{{\psi}}}{\omega^*;a},
	\end{equation}
	where $\boldsymbol{\psi}$ and $\boldsymbol{u}^{\mathrm{par}}$ are defined in (\ref{partWeakNeu}) and (\ref{def_loc_part}), respectively.
\end{proposition}
\begin{proof}
Since $\boldsymbol{u}^{\mathrm{e}}|_{\omega^*}-\boldsymbol{{\psi}} \in \boldsymbol{H}_a(\omega^*)$, estimate (\ref{local_estimate}) follows immediately from the definition of the $n$-width and \cref{nwidth_char}.
\end{proof}

\begin{remark}\label{stream-functions-formulation}
In the two-dimensional case, by (\ref{stream-functions}), we see that $d_n(\omega,\omega^*)$ can be rewritten as     
\begin{equation}\label{equiv-nwidth}
d_{n}(\omega,\omega^{\ast})=\inf_{Q(n)\subset H_{A^{-1}}(\omega)/\mathbb{R}}\sup_{u\in H_{A^{-1}}(\omega^{\ast})/\mathbb{R}}\inf_{v\in Q(n)} \frac{\Vert \nabla u - \nabla v\Vert_{L^2(\omega;a)}}{\Vert\nabla u\Vert_{L^2(\omega^{\ast};a)}},  
\end{equation}
which was considered in \cite{babuska2011optimal}. Correspondingly, the local approximation space $\boldsymbol{S}^{\mathrm{v}}_{n}(\omega) = \big\{\cuScal\, v_{1}|_{\omega},\cdots, \cuScal\, v_{n}|_{\omega}\big\}$, where $v_{k}$ is the $k$-th eigenfunction of the problem
\begin{equation}\label{eig_stream_func}
a_{\omega^*}(\nabla v,\nabla \varphi) = \lambda \,a_{\omega}(\nabla v,\nabla \varphi) 
		\quad \forall \varphi \in H_{A^{-1}}(\omega^*)/\mathbb{R}.    
\end{equation}
This alternative formulation is very useful in practice: the eigenproblem (\ref{eig_stream_func}) is easier to solve than (\ref{EVP}) and an efficient solution technique has been developed in \cite{ma2022error}.
\end{remark}

Next we construct the local approximation space for the pressure from the local velocity approximation space $\boldsymbol{S}^{\mathrm{v}}_n(\omega)= \mathrm{span}\{\boldsymbol{v}_{1}|_{\omega}, \cdots, \boldsymbol{v}_{n}|_{\omega} \}$. For each $\boldsymbol{v}_k$, we define $(\tilde{\boldsymbol{v}}_{k},p_k)\in \hdi{}{}{\omega^0}\times L^2_0(\omega^0)$ to be the unique solution of
\begin{equation}
	\label{eq_pressure_recon}
	\left\{\begin{array}{lll}
		a_{\omega^0}(\tilde{\boldsymbol{v}}_{k},\boldsymbol{w}) + b_{\omega^0}(\boldsymbol{w},p_k)&= 0 &\forall \boldsymbol{w}\in \hdi{0}{}{\omega^0}, \\
		b_{\omega^0}(\tilde{\boldsymbol{v}}_{k},q) &= 0  &\forall q\in L^2(\omega^0), \\
		\tilde{\boldsymbol{v}}_{k}\cdot\boldsymbol{\nu} &= \boldsymbol{v}_k\cdot\boldsymbol{\nu} &\text{ on } \partial\omega^0.
	\end{array}\right.
\end{equation}
The local pressure approximation space is then defined by
\begin{equation}
	\label{def_loc_pres}
	S_n^{\mathrm{p}}(\omega^0):= \text{span}\{p_1,\cdots,p_n \}.
\end{equation}

We conclude this subsection by stating the exponential decay of the $n$-width $d_{n}(\omega,\omega^*)$, which will be proved in section \ref{sec:proofs_nearly_continuous}.  
\begin{theorem}
	\label{theorem_nearly_exponential}
	For every $\epsilon >0$, there exists $n_{\epsilon}$, such that for any $n>n_{\epsilon}$ 
	\begin{equation*}
		d_n(\omega,\omega^*) 
		\leq e^{-n^{\left(\frac{1}{d+1}-\epsilon\right)}}.
	\end{equation*} 
\end{theorem}


	\subsection{Inf-sup stability}
	\label{globAppSpace}
	
 The local spaces in the preceding subsection are constructed based on approximation considerations. In general, the resulting global spaces (see (\ref{def_glob_spaces})) fail to satisfy the inf-sup stability. We deal with this issue in this subsection. The basic idea is to suitably enrich the multiscale space for the velocity such that the discrete divergence operator is surjective. To this end, we add for each local pressure basis function a vector field, whose divergence is exactly this pressure function, to the velocity space. These vector fields can be computed by solving suitable local problems.
	
Recall the local pressure space $S_n^{\mathrm{p}}(\omega^0)= \text{span}\{p_1,\cdots,p_n \}$ defined above. For each $p_{k}$, let $(\boldsymbol{u}_k,\tilde{p}_k)\in \hdi{0}{}{\omega^0}\times L^2_0(\omega^0)$ be the unique solution of
	\begin{equation}
		\label{enrProb}
		\left\{\begin{array}{lll}
			a_{\omega^0}({\boldsymbol{u}_k},\boldsymbol{w}) +b_{\omega^0}(\boldsymbol{w},\tilde{p}_k) &= 0 &\forall \boldsymbol{w}\in\hdi{0}{}{\omega^0}, \\
			b_{\omega^0}(\boldsymbol{u}_k,q) &= (p_k,q)_{L^2(\omega^0)}  &\forall q\in L^2(\omega^0),\\
			{\boldsymbol{u}_k}\cdot\boldsymbol{\nu} &= 0 &\text{ on } \partial\omega^0.
		\end{array}\right.
	\end{equation}
	This Neumann problem is well posed since $p_k\in L^2_0(\omega^0)$ and we can apply \cite[Theorem 4.2.3]{boffi2013mixed} to obtain the stability estimate
	\begin{equation}\label{stability_enrichement}
		\hdinorm{\boldsymbol{u}_k}{\omega^0;a}
		\leq \frac{2}{\beta(\omega^0)} \lnorm{p_k}{\omega^0}.
	\end{equation}
	Here $\beta(\omega^0)$ denotes the inf-sup constant for the pair $\hdi{0}{}{\omega^0}/L^2_0(\omega^0)$ equipped with the norms $\hdinorm{\cdot}{\omega^0;a}/\Vert\cdot\Vert_{L^{2}(\omega^{0})}$. We define the local velocity enrichment space by
	\begin{equation}
		\label{def_loc_enr}
		\boldsymbol{V}_n^{\operatorname{en}}(\omega^0) := \text{span}\{\boldsymbol{u}_1,\cdots,\boldsymbol{u}_n \}. 
	\end{equation}
\begin{remark}
Note that the local problems (\ref{eq_pressure_recon}) and (\ref{enrProb}) are similar in form (introducing a new variable $ \hat{\boldsymbol{v}}_{k} = \tilde{\boldsymbol{v}}_{k}-{\boldsymbol{v}}_{k}$ in (\ref{eq_pressure_recon})). In practice, after a discretization, they yield a set of linear systems with the same matrix that can be solved efficiently by direct methods.
\end{remark}

So far, we have constructed the local approximation spaces and the local velocity enrichment spaces. Now we can show the inf-sup stability for our multiscale spaces $\boldsymbol{V}^{\mathrm{MS}}$ and $Q^{\mathrm{MS}}$. Recall that 
\begin{equation*}
 \boldsymbol{V}^{\mathrm{MS}} 
	= \boldsymbol{S}^{\mathrm{v}}_{n}(\Omega)	
	+ \boldsymbol{V}_{n}^{\operatorname{en}}(\Omega)
	+ \boldsymbol{V}^{\operatorname{RT}}_0
	\quad\text{and}\quad 
	Q^{\mathrm{MS}} 
	= S^{\mathrm{p}}_{n}(\Omega)	
	+ Q^{\operatorname{RT}}_0,       
\end{equation*}
where $\boldsymbol{S}^{\mathrm{v}}_{n}(\Omega)$ ($S^{\mathrm{p}}_{n}(\Omega)$) is the global velocity (pressure) approximation space, $\boldsymbol{V}_{n}^{\operatorname{en}}(\Omega)$ is the global velocity enrichment space, and $V^{\operatorname{RT}}_0/Q_0^{\mathrm{RT}}$ is the Raviart-Thomas pair with respect to the partition $\{\omega_i^0\}_{i=1}^M$ (see \cref{raviart-thomas-pair}).
	\begin{theorem}
		\label{IS_MS}
		The pair $\boldsymbol{V}^{\mathrm{MS}}/Q^{\mathrm{MS}}$ is inf-sup stable, i.e.
		\begin{equation*}
			\sup_{\boldsymbol{u}^{\mathrm{MS}}\in \boldsymbol{V}^{\mathrm{MS}}}\frac{b(\boldsymbol{u}^{\mathrm{MS}},p^{\mathrm{MS}})}{\hdinorm{\boldsymbol{u}^{\mathrm{MS}}}{\Omega;a}\lnorm{p^{\mathrm{MS}}}{\Omega}}
			\geq \beta^{\mathrm{MS}}, \quad \forall p^{\mathrm{MS}}\in Q^{\mathrm{MS}}.
		\end{equation*}
		Here, $\beta^{\mathrm{MS}}$ is given by
		\begin{equation*}
	\ma{\beta^{\mathrm{MS}} := \Big(\frac{2}{\beta_{\mathrm{min}}}+\frac{1}{\beta^{\operatorname{RT}}}\Big)^{-1}},
		\end{equation*}
		where ${\beta_{\mathrm{min}}} := \min\limits_{i=1,\cdots,M} \beta(\omega_i^0)$, and $\beta^{\operatorname{RT}}$ denotes the inf-sup constant for the pair $V^{\operatorname{RT}}_0/Q_0^{\mathrm{RT}}$ equipped with the norms $\hdinorm{\cdot}{\Omega;a}/\Vert\cdot\Vert_{L^{2}(\Omega)}$.
	\end{theorem}
 
 \begin{remark}
We can derive a lower bound for $\beta^{\mathrm{MS}}$ with an explicit dependence on the coefficient $A$. Let $\widetilde{\alpha}_0:={\rm min}\{\alpha^{1/2}_{0},1\}$, and define $\widetilde{\beta}({\omega_i^{0}})$ and $\widetilde{\beta}^{\operatorname{RT}}$ to be the inf-sup constants as above but with respect to the standard norms $\hdinorm{\cdot}{\omega_{i}^0}/\Vert\cdot\Vert_{L^{2}(\omega_{i}^{0})}$ and $\hdinorm{\cdot}{\Omega}/\Vert\cdot\Vert_{L^{2}(\Omega)}$, respectively. It is easy to see that
\begin{equation*}
{\beta}^{\operatorname{RT}}\geq \widetilde{\alpha}_0 \widetilde{\beta}^{\operatorname{RT}},\quad {\beta}({\omega_i^{0}}) \geq \widetilde{\alpha}_0 \widetilde{\beta}({\omega_i^{0}}), \quad i=1,\cdots,M.
\end{equation*}
Therefore, it holds
\begin{equation*}
 \beta^{\mathrm{MS}}\geq \widetilde{\alpha}_0 \Big(\frac{2}{\widetilde{\beta}_{\mathrm{min}}}+\frac{1}{\widetilde{\beta}^{\operatorname{RT}}}\Big)^{-1}, \quad \text{with}\quad {\widetilde{\beta}_{\mathrm{min}}} := \min\limits_{i=1,\cdots,M} \widetilde{\beta}(\omega_i^0).
\end{equation*}
 \end{remark}
 
      Before proceeding to the proof, let us outline the two main observations for showing this result. First, by construction, we have inf-sup stability for the local spaces $\boldsymbol{V}_{n_i}^{\mathrm{en}}(\omega_i^0) / S_{n_i}^{\mathrm{p}}(\omega_i^0)$, which can be extended to the global spaces $\boldsymbol{V}_n^{\mathrm{en}}(\Omega)/S_n^{\mathrm{p}}(\Omega)$, since the $\omega_i^0$'s are pairwise disjoint. Secondly, the $L^2$-orthogonality of $S_n^{\mathrm{p}}(\Omega)$ and $Q_0^{\mathrm{RT}}$ allows us to extend inf-sup stability to the whole multiscale spaces. To prove \cref{IS_MS}, we need the following two lemmas.
        \begin{lemma}
		\label{IS_approx_loc}
		For each $p\in S^{\mathrm{p}}_n(\omega^0)$, there exists $\boldsymbol{u}\in \boldsymbol{V}_n^{\operatorname{en}}(\omega^0)$ such that
		\begin{equation*}
			\hdinorm{\boldsymbol{u}}{\omega^0;a} 
			\leq \frac{2 }{\beta(\omega^0)} \lnorm{p}{\omega^0}
			\quad\text{and}\quad
			b(\boldsymbol{u},q) = (p,q)_{L^2(\omega^0)} \;\medspace\forall q\in L^2(\omega^0).
		\end{equation*}
	\end{lemma}
	\begin{proof}
		Every $p\in S^{\mathrm{p}}_n(\omega^0)$ has the form
		$
		p = \sum_{k=1}^n \mu_k p_k$, where $\mu_k\in\mathbb{R}. 
		$
		For $k=1,\cdots,n$, let $\boldsymbol{u}_k\in \boldsymbol{V}_n^{\operatorname{en}}(\omega^0)$ and $\tilde{p}_k\in L^2_0(\omega^0)$ be from (\ref{enrProb}). Define  
		$
		\boldsymbol{u} = \sum_{k=1}^n \mu_k \boldsymbol{u}_k \in \boldsymbol{V}_n^{\operatorname{en}}(\omega^0)  
		\text{ and }
		\tilde{p} = \sum_{k=1}^n \mu_k \tilde{p}_k.
		$
		The linearity of (\ref{enrProb}) implies that $(\boldsymbol{u},\tilde{p})$ satisfies 
		\begin{equation*}
			\left\{\begin{array}{lll}
				a_{\omega^0}(\boldsymbol{u},\boldsymbol{w}) +b_{\omega^0}(\boldsymbol{w},\tilde{p}) &= 0 &\forall \boldsymbol{w}\in\hdi{0}{}{\omega^0}, \\
				b_{\omega^0}(\boldsymbol{u}_,q) &= (p,q)_{L^2(\omega^0)}  &\forall q\in L^2(\omega^0),\\
				\boldsymbol{u}\cdot\boldsymbol{\nu} &= 0 &\text{ on } \partial\omega^0.
			\end{array}\right.
		\end{equation*}
		Using a standard stability estimate (see \cite[Theorem 4.2.3]{boffi2013mixed}) gives the desired result.
	\end{proof}
	
	\begin{lemma}
		\label{IS_approx}
		For each $p\in S^{\mathrm{p}}_n(\Omega)$, there exists $\boldsymbol{u}\in \boldsymbol{V}_n^{\operatorname{en}}(\Omega)$ such that 
		\begin{equation*}
			\hdinorm{\boldsymbol{u}}{\Omega;a} 
			\leq \frac{2}{\beta_{\mathrm{min}}} \lnorm{p}{\Omega}
			\quad\text{and}\quad
			b(\boldsymbol{u},q) = (p,q)_{L^2(\Omega)}\; \medspace\forall q\in L^2(\Omega), 
		\end{equation*}
		where $\beta_{\mathrm{min}} := \min\limits_{i=1,\cdots,M} \beta(\omega_i^0)$.
	\end{lemma}
	\begin{proof}
		Let $p\in S_n^{\mathrm{p}}(\Omega)$ be given by $p=\sum_{i=1}^M p_i$ with $p_i\in S_{n_i}^{\mathrm{p}}(\omega_i^0)$.	Applying \cref{IS_approx_loc} yields that for each $p_i$, there exists a $\boldsymbol{u}_i\in \boldsymbol{V}^{\operatorname{en}}_{n_i}(\omega_i^0)$ such that 
		\begin{equation*}
			\hdinorm{\boldsymbol{u}_i}{\omega_i^0;a} 
			\leq \frac{2 }{\beta(\omega_i^0)} \lnorm{p_i}{\omega_i^0}
			\quad\text{and}\quad
			b_{\omega_i^0}(\boldsymbol{u}_i,q) = (p_i,q)_{L^2(\omega_i^0)}\; \medspace\forall q\in L^2(\omega_i^0).
		\end{equation*}
		Define $\boldsymbol{u}:= \sum_{i=1}^M \boldsymbol{u}_i\in \boldsymbol{V}^{\operatorname{en}}_n(\Omega)$ and it holds that $b(\boldsymbol{u},q) = (p,q)_{L^2(\Omega)}$ for all $q\in L^2(\Omega)$. For the inequality we use the fact that $\boldsymbol{u}_i$ and $p_i$ are supported in the pairwise disjoint sets $\omega_i^0$ to compute
		\begin{align}
			\label{crucial_est_ven}
			\hdinorm{\boldsymbol{u}}{\Omega;a}^2 
			= \sum_{i=1}^M \hdinorm{\boldsymbol{u}_i}{\omega_i^0;a}^2
			\leq \sum_{i=1}^M \frac{4}{\beta(\omega_i^0)^2}\lnorm{p_i}{\omega_i^0}^2
			\leq \frac{4}{\beta_{\mathrm{min}}^2}\lnorm{p}{\Omega}^2.
		\end{align}
	\end{proof}

 Now we can prove \cref{IS_MS}.
	
	\begin{proof}[Proof of \cref{IS_MS}]
		Let $p^{\mathrm{MS}}=p + c\in Q^{\mathrm{MS}}$ with  $p\in S^{\mathrm{p}}_n(\Omega)$ and $c\in Q^{\operatorname{RT}}_0$. Due to \cref{IS_approx}, there exists $\boldsymbol{u}\in \boldsymbol{V}^{\operatorname{en}}(\Omega)$ with
		\begin{equation}
			\label{IS_proof_ap}
		\hdinorm{\boldsymbol{u}}{\Omega;a}
			\leq \frac{2}{\beta_{\mathrm{min}}}\lnorm{p}{\Omega} \quad \text{and}\quad	b(\boldsymbol{u},q)=(p,q)_{L^2(\Omega)}\;\medspace \forall q\in L^2(\Omega).
		\end{equation}
		Moreover, the assumption $\di \boldsymbol{V}_0^{\mathrm{RT}}=Q_0^{\mathrm{RT}}$ implies the existence of $\boldsymbol{u}^{\operatorname{RT}}\in \boldsymbol{V}^{\operatorname{RT}}_0$ with
		\begin{equation}\label{IS_proof_RT}
		\hdinorm{\boldsymbol{u}^{\operatorname{RT}}}{\Omega;a}
			\leq \frac{1}{\beta^{\operatorname{RT}}}
			\lnorm{c}{\Omega}\quad\text{and}\quad	b(\boldsymbol{u}^{\operatorname{RT}},q)=(c,q)_{L^2(\Omega)}\;\medspace \forall q\in L^2(\Omega).
		\end{equation}
Note that here we have equipped $\boldsymbol{V}_0^{\mathrm{RT}}$ with the weighted norm $\hdinorm{\cdot}{\Omega;a}$. Furthermore, we notice that the spaces $S_n^{\mathrm{p}}(\Omega)$ and $Q_0^{\operatorname{RT}}$ are $L^2(\Omega)$-orthogonal, because on every partition domain $\omega^0$, the integral of functions in $S_n^{\mathrm{p}}(\Omega)$  vanishes and functions in $Q_0^{\operatorname{RT}}$ are constant. Now, choosing $q=p$ in (\ref{IS_proof_RT}), $q=c$ in (\ref{IS_proof_ap}) and using the $L^2(\Omega)$-orthogonality of $c$ and $p$, we see that  $\boldsymbol{u}^{\mathrm{MS}} = \boldsymbol{u} + \boldsymbol{u}^{\operatorname{RT}}$ satisfies
		\begin{align*}
			b(\boldsymbol{u}^{\mathrm{MS}},p^{\mathrm{MS}})
			&= b(\boldsymbol{u},p) + b(\boldsymbol{u}^{\operatorname{RT}},c) + (c,p)_{L^2(\Omega)} + (p,c)_{L^2(\Omega)}\\
			&= \lnorm{p}{\Omega}^2 + \lnorm{c}{\Omega}^2
			= \lnorm{p^{\mathrm{MS}}}{\Omega}^2.
		\end{align*}
		Using the above equality, the orthogonality of $p$ and $c$, the triangle inequality, and the inequalities in (\ref{IS_proof_ap}) and (\ref{IS_proof_RT}), we conclude
		\begin{align*}
			&\frac{b(\boldsymbol{u}^{\mathrm{MS}},p^{\mathrm{MS}})}{\hdinorm{\boldsymbol{u}^{\mathrm{MS}}}{\Omega;a}\lnorm{p^{\mathrm{MS}}}{\Omega}}
			=\frac{\lnorm{p^{\mathrm{MS}}}{\Omega}}{\hdinorm{\boldsymbol{u}^{\mathrm{MS}}}{\Omega;a}}
			=\frac{\sqrt{\lnorm{p}{\Omega}^2+\lnorm{c}{\Omega}^2}}{\hdinorm{\boldsymbol{u}^{\mathrm{MS}}}{\Omega;a}}\\
			&\geq \frac{\sqrt{\lnorm{p}{\Omega}^2+\lnorm{c}{\Omega}^2}}{\hdinorm{\boldsymbol{u}}{\Omega;a} + \hdinorm{\boldsymbol{u}^{\operatorname{RT}}}{\Omega;a}}
			\geq \frac{\sqrt{\lnorm{p}{\Omega}^2+\lnorm{c}{\Omega}^2}}{2/\beta_{\mathrm{min}} \lnorm{p}{\Omega} + 1/\beta^{\operatorname{RT}} \lnorm{c}{\Omega}}\\
			&\geq 
			\frac{1}{2/\beta_{\mathrm{min}} + 1/\beta^{\operatorname{RT}} }
			= \beta^{\mathrm{MS}}.
		\end{align*}
	\end{proof}

\subsection{Global approximation error estimates}
In this subsection, we will show that the global approximation errors for the velocity and pressure are bounded by the local velocity approximation errors, thereby decaying exponentially with respect to the local degrees of freedom. To give the result, we need the following stability estimates for problems (\ref{dualWeakCont}) and (\ref{partWeakNeu}): 
	\begin{equation}\label{eq:stability_inequality}
 \begin{array}{ll} 
 &\lnorm{\boldsymbol{{u}}^{\rm e}}{\Omega;a}
		\leq C_{\operatorname{stab}} \left( \lnorm{f}{\Omega} + \lnorm{g_N}{\partial\Omega}\right),\\[2mm]
		&\lnorm{\boldsymbol{{\psi}}_i}{\omega_i^*;a}
		\leq C^i_{\operatorname{stab}} \left( \lnorm{f}{\omega_i^*} + \lnorm{g_N}{\partial\Omega\cap\partial\omega_i^*}\right)\quad 1\leq i\leq M.		
  \end{array}
	\end{equation}
	\begin{theorem}
		\label{GlobLocErr}
		Assume that there exist $\boldsymbol{v}^{i}\in \boldsymbol{S}^{\mathrm{v}}_{n_i}(\omega_i)$ for $i=1,\cdots,M$ such that
		\begin{equation}\label{local-div-estimate}							    
  \hdinorm{\boldsymbol{u}^{\mathrm{e}}-\boldsymbol{u}_i^{\mathrm{par}}-\boldsymbol{v}^{i}}{\omega_i;a} 
			\leq \epsilon_i \hdinorm{\boldsymbol{u}^{\mathrm{e}} - \boldsymbol{{\psi}}_i}{\omega_i^*;a}.
		\end{equation}
	Let $\boldsymbol{v}=\sum_{i=1}^{M}\chi_{i}\boldsymbol{v}^{i} \in \boldsymbol{S}^{\mathrm{v}}_{n}(\Omega)$ and $\displaystyle \epsilon_{\rm max}=\max_{1\leq i\leq M}\epsilon_i$. Then,
		\begin{align*}							    	   	
  \lnorm{\boldsymbol{u}^{\mathrm{e}}-\boldsymbol{u}^{\mathrm{par}}-\boldsymbol{v}}{\Omega;a} 
			&\leq 2 \sqrt{\kappa\kappa^*}\tilde{C}\epsilon_{\rm max} ,\\[1mm]
			\lnorm{\di\left(\boldsymbol{u}^{\mathrm{e}}-\boldsymbol{u}^{\mathrm{par}}-\boldsymbol{v}\right)}{\Omega} 
			&\leq 2\sqrt{\kappa\kappa^*}\tilde{C} \alpha_1^{1/2} \big(\max\limits_{1\leq i\leq M}\norm{\nabla\chi_i}{L^\infty(\omega_i)}\big)\epsilon_{\rm max} ,\\[1mm]
			\inf_{p\in Q^{\mathrm{MS}}} \lnorm{p^{\mathrm{e}}-p^{\mathrm{par}}-p}{\Omega}
			&\leq  2\sqrt{\kappa^*} \tilde{C} \beta^{-1}_{\mathrm{min}} \epsilon_{\rm max} ,
		\end{align*}
		where
		$
		\tilde{C} =\left( C_{\operatorname{stab}} + C_{\operatorname{max}} \right) \left( \lnorm{f}{\Omega} + \lnorm{g_N}{\partial\Omega} \right)$ with  $C_{\operatorname{max}} = \max_{1\leq i\leq M} C_{\operatorname{stab}}^i$, $\kappa$ and $\kappa^*$ are defined in (\ref{pointwise_overlap}), and $\beta_{\mathrm{min}}$ is as in \cref{IS_MS}.
	\end{theorem}

 Before proving the theorem, we give two auxiliary lemmas.
        \begin{lemma}
		\label{lemma_pres_recon}
		For any $\tilde{\boldsymbol{v}}\in \boldsymbol{S}_n^{\mathrm{v}}(\omega)$, there exists $p\in S_n^{\mathrm{p}}(\omega^0)$ such that  
		\begin{equation*}
			a_{\omega^0}(\tilde{\boldsymbol{v}},\boldsymbol{w}) + b_{\omega^0}(\boldsymbol{w},p) = 0 \quad \forall \boldsymbol{w}\in \hdi{0}{}{\omega^0}.
		\end{equation*}
	\end{lemma}
	\begin{proof}
For each $k=1,\cdots,n$, let $(\boldsymbol{v},p_{k})$ be the solution of (\ref{eq_pressure_recon}). Since ${\boldsymbol{v}_k}|_{\omega}\in \boldsymbol{S}_n^{\mathrm{v}}(\omega)\subset\hdi{}{0}{\omega}$, we have
	$
	{\boldsymbol{v}_k}|_{\omega^0} - \boldsymbol{v} \in \hdi{0}{0}{\omega^0}.
	$
	The first equation of (\ref{eq_pressure_recon}) together with  ${\boldsymbol{v}_k}\in \boldsymbol{H}_a(\omega^*)$ implies
	$
	a_{\omega^0}({\boldsymbol{v}_k}|_{\omega^0}-\boldsymbol{v},\boldsymbol{w}) = 0$ for all $\boldsymbol{w}\in \hdi{0}{0}{\omega^0}$.
	Inserting $\boldsymbol{w}= {\boldsymbol{v}_k|}_{\omega^0}-\boldsymbol{v}$ in the equation above, we conclude that $\boldsymbol{v}={\boldsymbol{v}_k|}_{\omega^0}$ and consequently
	$
	a_{\omega^0}(\boldsymbol{v}_k,\boldsymbol{w}) + b_{\omega^0}(\boldsymbol{w},p_k) = 0 $ for all  $\boldsymbol{w}\in \hdi{0}{}{\omega^0}$.
	The desired result follows by using the linearity of the problem and the fact that $\boldsymbol{S}_n^{\mathrm{v}}(\omega)={\rm span}\{\boldsymbol{v}_k\}_{k=1}^{M}$.
	\end{proof}
 
 Using \cref{lemma_pres_recon}, it can be shown that $S_n^{\mathrm{p}}(\omega^0)$ has approximation properties which are essentially as good as those of $\boldsymbol{S}_n^{\mathrm{v}}(\omega)$.
	\begin{lemma}
		\label{globPresEr}
		Let $p^{\mathrm{par}}$ be the global particular pressure defined by (\ref{def_glob_par}), and $Q^{\mathrm{MS}}$ be the global pressure approximation space defined by (\ref{def_MS_spaces}). Then, 
		\begin{equation*}
			\inf_{p\in Q^{\mathrm{MS}}} \lnorm{p^{\mathrm{e}}-p^{\mathrm{par}}-p}{\Omega}^2
			\leq 	\frac{1}{\beta_{\mathrm{min}}^2}	\sum_{i=1}^{M} \inf_{\boldsymbol{v}^i\in \boldsymbol{S}^{\mathrm{v}}_{n_i}(\omega_i)} \lnorm{\boldsymbol{u}^{\mathrm{e}}-\boldsymbol{u}_i^{\mathrm{par}}-\boldsymbol{v}^i}{\omega_i^0;a}^2,
		\end{equation*}
		where $\beta_{\mathrm{min}}$ is defined in \cref{IS_MS}.
	\end{lemma}
	\begin{proof}
		Let $\boldsymbol{v}^i\in \boldsymbol{S}^{\mathrm{v}}_{n_i}(\omega_i)$ for $i=1,\cdots,M$. By \cref{lemma_pres_recon}, there exist $p^i\in S^{\mathrm{p}}_{n_i}(\omega_i^0)$ with	
		\begin{equation*}
			a_{\omega_i^0}(\boldsymbol{v}^i,\boldsymbol{w}) + b_{\omega_i^0}(\boldsymbol{w},p^i) = 0 \quad\forall\boldsymbol{w}\in\hdi{0}{}{\omega_i^0}. 
		\end{equation*}
		On the other hand, we can deduce from (\ref{partWeakNeu}) and (\ref{dualWeakCont}) that
		\begin{equation*}
			a_{\omega_i^*}(\boldsymbol{u}^{\mathrm{e}}-\boldsymbol{{\psi}}_i,\boldsymbol{w}) + b_{\omega_i^*}(\boldsymbol{w},p^{\mathrm{e}}-\phi_i) = 0 \quad\forall\boldsymbol{w}\in\hdi{0}{}{\omega_i^*}. 
		\end{equation*}
		Subtracting the two equations above yields the error equation
		\begin{equation}
			\label{globPres_errEq0_proof}
			a_{\omega_i^0}(\boldsymbol{u}^{\mathrm{e}}-\boldsymbol{u}_i^{\mathrm{par}}-\boldsymbol{v}^i,\boldsymbol{w}) + b_{\omega_i^0}(\boldsymbol{w},p^{\mathrm{e}}-p_i^{\mathrm{par}}-p^i) = 0\quad\forall\boldsymbol{w}\in\hdi{0}{}{\omega_i^0}. 
		\end{equation}
		Let us choose $c^i = \dashint_{\omega_i^0}p^{\mathrm{e}} \dx$. It holds
		$b_{\omega_i^0}(\boldsymbol{w},c^i) = 0$ for all $\boldsymbol{w}\in\hdi{0}{}{\omega_i^0}$. Adding this equation to (\ref{globPres_errEq0_proof}), we get 
		\begin{equation}
			\label{globPres_errEq_proof}
			a_{\omega_i^0}(\boldsymbol{u}^{\mathrm{e}}-\boldsymbol{u}_i^{\mathrm{par}}-\boldsymbol{v}^i,\boldsymbol{w}) + b_{\omega_i^0}(\boldsymbol{w},p^{\mathrm{e}}-p_i^{\mathrm{par}}-p^i-c^i) = 0 \quad\forall\boldsymbol{w}\in\hdi{0}{}{\omega_i^0}. 
		\end{equation}
		Since $p^{\mathrm{e}}-p_i^{\mathrm{par}}-p^i-c^i \in L^2_0(\omega_i^0)$, we can use the inf-sup stability of the pair $\hdi{0}{}{\omega_i^0}/L^2_0(\omega_i^0)$ and (\ref{globPres_errEq_proof}) to conclude
		\begin{align}
			\label{loc_est_proof}
			\beta_{\mathrm{min}} 
			&\leq \sup_{\boldsymbol{w}\in\hdi{0}{}{\omega_i^0}}
			\frac{b_{\omega_i^0}(\boldsymbol{w},p^{\mathrm{e}}-p_i^{\mathrm{par}}-p^i-c^i)}{\hdinorm{\boldsymbol{w}}{\omega_i^0;a}\lnorm{p^{\mathrm{e}}-p_i^{\mathrm{par}}-p^i-c^i}{\omega_i^0}}\nonumber\\
			&\leq \sup_{\boldsymbol{w}\in\hdi{0}{}{\omega_i^0}}
			\frac{|a_{\omega_i^0}(\boldsymbol{u}^{\mathrm{e}}-\boldsymbol{u}_i^{\mathrm{par}}-\boldsymbol{v}^i,\boldsymbol{w})|}{\hdinorm{\boldsymbol{w}}{\omega_i^0;a}\lnorm{p^{\mathrm{e}}-p_i^{\mathrm{par}}-p^i-c^i}{\omega_i^0}}\nonumber\\
			&\leq \frac{\lnorm{\boldsymbol{u}^{\mathrm{e}}-\boldsymbol{u}_i^{\mathrm{par}}-\boldsymbol{v}^i}{\omega_i^0;a}}{\lnorm{p^{\mathrm{e}}-p_i^{\mathrm{par}}-p^i-c^i}{\omega_i^0}}.
		\end{align}
		Let us write
		$p = \sum_{i=1}^{M}(p^i+\mathbbm{1}_{\omega_i^0}c^i)$
		with the indicator function $\mathbbm{1}_{\omega_i^0}$. 
		We have $p\in Q^{\mathrm{MS}}$. Summing up (\ref{loc_est_proof}) for $i=1,2,\cdots,M$, we see that
		\begin{align*}
			\lnorm{p^{\mathrm{e}}-p^{\mathrm{par}}-p}{\Omega}^2
			&= \sum_{i=1}^M\lnorm{p^{\mathrm{e}}-p_i^{\mathrm{par}}-p^i-c^i}{\omega_i^0}^2\\
			&\leq \sum_{i=1}^M \beta_{\mathrm{min}}^{-2}\lnorm{\boldsymbol{u}^{\mathrm{e}}-\boldsymbol{u}_i^{\mathrm{par}}-\boldsymbol{v}^i}{\omega_i^0;a}^2.
		\end{align*}
		Since the $\boldsymbol{v}^i$'s are arbitrary, we obtain the desired result immediately. 
	\end{proof}
	\begin{remark}
		In (\ref{loc_est_proof}) we see why we need to add the Raviart-Thomas pair to the multiscale spaces: The error equation (\ref{globPres_errEq_proof}) holds only for $\boldsymbol{w}\in\hdi{0}{}{\omega_i^0}$. Using the stability of the pair $\hdi{}{}{\omega_i^0}/L^2(\omega_i^0)$ would not allow us to switch from $b(\cdot,\cdot)$ to $a(\cdot,\cdot)$.
	\end{remark}
	
	\begin{proof}[Proof of \cref{GlobLocErr}]
Noting that (\ref{local-div-estimate}) is indeed the weighted $L^{2}$-norm error estimate, one readily sees that
		\begin{align}
			\label{u_glob_proof}
			\begin{split}
				\lnorm{\boldsymbol{u}^{\mathrm{e}}-\boldsymbol{u}^{\mathrm{par}}-\boldsymbol{v}}{\Omega;a}^2
				\leq \kappa\Big(\max_{1\leq i\leq M}\epsilon_i^2\Big)\sum_{i=1}^{M}\lnorm{\boldsymbol{u}^{\mathrm{e}} - \boldsymbol{{\psi}}_i}{\omega_i^*;a}^2	
			\end{split}
		\end{align}
		and for the divergence part
		\begin{align}
			\label{divu_glob_proof}
			\lnorm{\di(\boldsymbol{u}^{\mathrm{e}}-\boldsymbol{u}^{\mathrm{par}}-\boldsymbol{v})}{\Omega}^2
			\leq \kappa \alpha_1 \Big(\max\limits_{1\leq i\leq M} \epsilon^{2}_i\norm{\nabla\chi_i}{L^\infty(\omega_i)}^{2}\Big)\sum_{i=1}^{M}\lnorm{\boldsymbol{u}^{\mathrm{e}} - \boldsymbol{{\psi}}_i}{\omega_i^*;a}^2.
		\end{align}
		Furthermore, we can use the stability estimates (\ref{eq:stability_inequality}) to deduce
		\begin{align}
			\label{u_min_part_proof}
			\sum_{i=1}^{M}\lnorm{\boldsymbol{u}^{\mathrm{e}} - \boldsymbol{{\psi}}_i}{\omega_i^*;a}^2
			\leq 4 \kappa^* \left( C_{\operatorname{stab}}^2 + C_{\operatorname{max}}^2 \right) \left( \lnorm{f}{\Omega}^{2} + \lnorm{g_N}{\partial\Omega}^{2} \right).
		\end{align}
		The estimates for the velocity now follow from (\ref{u_glob_proof}), (\ref{divu_glob_proof}) and (\ref{u_min_part_proof}), and the estimate for the pressure is a consequence of \cref{globPresEr}, (\ref{local-div-estimate}), and (\ref{u_min_part_proof}).
	\end{proof}

 \subsection{Quasi-optimal convergence}
 By now, we have shown that our multiscale spaces are inf-sup stable and obtained coerciveness on the kernel by introducing the augmented formulation (\ref{globWeakAug}). We can thus apply the celebrated Babuska--Brezzi theory (see, e.g., \cite[Theorem 5.2.2]{boffi2013mixed}) to prove quasi-optimal convergence for our method. Before giving the result, let us recall the parameter $\gamma$ in (\ref{globWeakAug}) and introduce two constants: 
 \begin{equation*}
\gamma_{\rm min} = \min\{1,\gamma\}, \quad \gamma_{\rm max} =  \max\{1, \gamma\}.     
 \end{equation*}

\begin{theorem}\label{quasi-optimality}
Let $(\boldsymbol{u}^{\mathrm{e}},p^{\mathrm{e}})$ be the solution of problem (\ref{globWeakAug}). Problem (\ref{GFEM}) has a unique solution $(\boldsymbol{u}^{\mathrm{G}},p^{\mathrm{G}})$ with $\boldsymbol{u}^{\mathrm{G}} \in \boldsymbol{u}^{\mathrm{par}} +\boldsymbol{V}^{\mathrm{MS}}$ and $p^{\mathrm{G}}\in p^{\mathrm{par}}+Q^{\mathrm{MS}}$. Moreover, 
			\begin{align*}
				\hdinorm{\boldsymbol{u}^{\mathrm{e}}-{u}^{\mathrm{G}}}{\Omega;\,a} 
				\leq \gamma_{\rm min}^{-1}E_p^{\mathrm{MS}}
				+ \Big(\frac{2 \gamma_{\rm max}}{\gamma_{\rm min}}+ 
				\frac{2 \gamma_{\rm max}^{1/2}}{\gamma_{\rm min}^{1/2}\beta^{\mathrm{MS}}} \Big) E_{\boldsymbol{u}}^{\mathrm{MS}},\\
				\lnorm{p^{\mathrm{e}}-p^{\mathrm{G}}}{\Omega} 
				\leq \frac{3 \gamma_{\rm max}^{1/2}}{\gamma_{\rm min}^{1/2}\beta^{\mathrm{MS}}} E_p^{\mathrm{MS}}  + \Big(\frac{2 \gamma_{\rm max}^{3/2}}{\gamma_{\rm min}^{1/2}\beta^{\mathrm{MS}}}+ 
				\frac{\gamma_{\rm max}}{(\beta^{\mathrm{MS}})^2} \Big) E_{\boldsymbol{u}}^{\mathrm{MS}},
			\end{align*}
where $\beta^{\mathrm{MS}}$ is the inf-sup constant associated with the pair $\boldsymbol{V}^{\mathrm{MS}}/Q^{\mathrm{MS}}$, and 
	\begin{align*}
		E_p^{\mathrm{MS}} := \inf_{q\in Q^{\mathrm{MS}}}\lnorm{p^{\mathrm{e}}-p^{\mathrm{par}}-q}{\Omega}
		,\quad
		E_{\boldsymbol{u}}^{\mathrm{MS}} := \inf_{\boldsymbol{v}\in \boldsymbol{V}^{\mathrm{MS}}}\hdinorm{\boldsymbol{u}^{\mathrm{e}}-\boldsymbol{u}^{\mathrm{par}}-\boldsymbol{v}}{\Omega;\,a}.
	\end{align*}
\end{theorem}
		\begin{proof}
             Use \cref{IS_MS}, \cite[Theorems 4.2.3, 5.2.2]{boffi2013mixed}, and the fact that
             \begin{equation*}
              a^{\gamma}(\boldsymbol{u}, \boldsymbol{u}) \geq \gamma_{\rm min} \hdinorm{\boldsymbol{u}}{\Omega;\,a}^{2},\quad a^{\gamma}(\boldsymbol{u}, \boldsymbol{v}) \leq \gamma_{\rm max} \hdinorm{\boldsymbol{u}}{\Omega;\,a} \hdinorm{\boldsymbol{v}}{\Omega;\,a}.
             \end{equation*} 
		\end{proof}

While \cref{quasi-optimality} implies that we should choose a modest parameter $\gamma>0$, numerical experiments in section \ref{sec:num} show that our method works well with very small $\gamma$, even with $\gamma = 0$. Moreover, the error for the velocity in the weight $L^{2}$-norm decreases for smaller $\gamma$. Unfortunately, we are not able to prove a uniform discrete coerciveness condition for our method in the case $\gamma = 0$.
  
 
%
\section{Fully discrete mixed MS-GFEM}
\label{sec:discMSGFEM}
The mixed MS-GFEM in the previous section was developed in the continuous setting as a multiscale discretization scheme. In this section, nevertheless, we will present the method in a FE discrete setting for approximating solutions to discretized mixed problems. Similar local problems will be defined which can be solved directly without a second-level discretization. 

We assume that $\Omega$ is a Lipschitz polyhedral domain in $\mathbb{R}^{2}$ or $\mathbb{R}^{3}$. Let $\{\mathcal{T}_h\}_h$ be a shape-regular family of simplicial meshes on $\Omega$. The mesh size $h:=\max_{K\in\mathcal{T}_h}h_K$ with $h_K:=\operatorname{diam}(K)$ is assumed to be small enough to resolve the fine-scale features of the coefficient $A$. Let $\boldsymbol{V}_h \subset \hdi{}{}{\Omega}$ be the zeroth order Raviart-Thomas space on $\mathcal{T}_{h}$ and define $\boldsymbol{V}_{h,0}:= \boldsymbol{V}_h \cap \hdi{0}{}{\Omega}$. Moreover, we denote by $Q_h$ the space of piecewise constant functions and set $Q_{h,0}:=Q_h\cap L_0^2(\Omega)$. For simplicity, we assume that $g^N=0$. The standard mixed finite element approximation for (\ref{dualWeakCont}) is as follows: Find $\boldsymbol{u}_h^{\mathrm{e}} \in \boldsymbol{V}_{h,0}$ and $p_h^{\mathrm{e}} \in Q_{h,0}$ such that
\begin{equation}
	\label{neuWeakDisc}
	\left\{\begin{array}{lll}
		a(\boldsymbol{u}_h^{\mathrm{e}},{\boldsymbol{v}_h}) + b({\boldsymbol{v}_h},p_h^{\mathrm{e}}) &= 0 &\forall {\boldsymbol{v}_h}\in \boldsymbol{V}_{h,0}, \\
		b(\boldsymbol{u}_h^{\mathrm{e}},q_h) &= -(f,q_h)_{L^2(\Omega)}   & \forall q_h\in Q_h. \\
	\end{array}\right.
\end{equation}

Next we describe the mixed MS-GFEM in this setting. As in  section \ref{sec:cont}, we define the partition, the overlapping and oversampling domains $\{\omega_i^0\}, \{\omega_i\}$ and $\{\omega_i^*\}$, and assume that they are resolved by the mesh $\mathcal{T}_{h}$. For any subdomain $\omega\subset\Omega$, we define 
\begin{align*}
	\begin{split}
		&\boldsymbol{V}_h(\omega) = \{\boldsymbol{v}_h|_{\omega} \medspace : \medspace \boldsymbol{v}_h\in \boldsymbol{V}_h\},\\
           &Q_h(\omega) = \{q_h|_{\omega} \medspace : \medspace q_h\in Q_h\}, \quad  Q_{h,0}(\omega)=Q_h(\omega)\cap L_0^2(\omega),\\
		&\boldsymbol{V}_{h, N}\left(\omega\right)=\left\{\boldsymbol{v}_h \in \boldsymbol{V}_{h}\left(\omega\right): \boldsymbol{v}_h\cdot \boldsymbol{\nu}=0 \text { on } \partial \omega \cap \partial\Omega\right\}, \\
		&\boldsymbol{V}_{h, 0}\left(\omega\right)=\left\{\boldsymbol{v}_h \in \boldsymbol{V}_{h}\left(\omega\right): \boldsymbol{v}_h\cdot \boldsymbol{\nu} = 0 \text{ on }\partial\omega\right\}, \\
		& \boldsymbol{K}_h(\omega) = \{\boldsymbol{v}_h\in \boldsymbol{V}_h(\omega) \medspace : \medspace  b_{\omega}(\boldsymbol{v}_h,q_h) = 0\; \;\,\forall q_h\in Q_h\}, \\
		&\boldsymbol{K}_{h, N}\left(\omega\right)=\left\{\boldsymbol{v}_h \in \boldsymbol{K}_{h}\left(\omega\right): \boldsymbol{v}_h\cdot \boldsymbol{\nu}=0 \text { on } \partial \omega \cap \partial\Omega\right\}, \\
		&\boldsymbol{K}_{h, 0}\left(\omega\right)=\left\{\boldsymbol{v}_h \in \boldsymbol{K}_{h}\left(\omega\right): \boldsymbol{v}_h\cdot \boldsymbol{\nu} = 0 \text{ on }\partial\omega\right\}, \\
		&\boldsymbol{W}_{h}\left(\omega\right)=\left\{\boldsymbol{u}_h \in \boldsymbol{K}_{h, N}(\omega): a_{\omega}(\boldsymbol{u}_h,\boldsymbol{v}_h)=0\;\;\; \forall \boldsymbol{v}_h \in \boldsymbol{K}_{h, 0}\left(\omega\right)\right\}.
	\end{split}
\end{align*}
In particular, we note that functions in $\boldsymbol{K}_h(\omega)$ are divergence free, i.e., $\boldsymbol{K}_h(\omega) = \boldsymbol{V}_h(\omega)\cap \hdi{}{0}{\omega}$. The space $\boldsymbol{W}_{h}\left(\omega\right)$ is the discrete analogue of the generalized $a$-harmonic space $\boldsymbol{H}_a(\omega)$.

\subsection{Mixed GFEM}
As in the continuous setting, we consider a modified discrete problem: For $\gamma>0$, find $\boldsymbol{u}^{\mathrm{e}}_h \in \boldsymbol{V}_{h,0}$ and $p^{\mathrm{e}}_h\in Q_{h,0}$ such that
\begin{equation}
	\label{globWeakAug_disc}
	\left\{\begin{array}{lll}
		a^\gamma(\boldsymbol{u}^{\mathrm{e}}_h,\boldsymbol{v}_h) + b(\boldsymbol{v}_h,p^{\mathrm{e}}_h) &= -\gamma b(\boldsymbol{v}_h,f)
		&\forall \boldsymbol{v}_h\in \boldsymbol{V}_{h,0}, \\
		b(\boldsymbol{u}^{\mathrm{e}}_h,q_h) &= -(f,q_h)_{L^2(\Omega)}  & \forall q_h\in Q_h. \\
	\end{array}\right.
\end{equation}
With the global ellipticity of the bilinear form $a^{\gamma}$ and the inf-sup stability of the pair $\boldsymbol{V}_{h,0}/Q_{h,0}$, problem (\ref{globWeakAug_disc}) is well-posed with the same solution as (\ref{neuWeakDisc}).

Given local approximation spaces $\boldsymbol{S}^{\mathrm{v}}_{h,n_i}(\omega_i)\subset \boldsymbol{V}_h(\omega_i)$, $S^{\mathrm{p}}_{h,n_i}(\omega_i^0)\subset  Q_{h,0}(\omega^{0}_i)$, and local velocity enrichment spaces $\boldsymbol{V}_{h,n_i}^{\operatorname{en}}(\omega_i^0)\subset \boldsymbol{V}_{h, 0}(\omega^{0}_i)$, we define the global spaces similarly as before by
\begin{align}
	\label{def_glob_spaces_disc}
	&\boldsymbol{S}^{\mathrm{v}}_{h,n}(\Omega) := \Big\{ \sum_{i=1}^M \Pi_h(\chi_i \boldsymbol{v}_h^i) \medspace|\medspace \boldsymbol{v}_h^i\in \boldsymbol{S}^{\mathrm{v}}_{h,n_i}(\omega_i) \Big\}	
	,\nonumber\\
 &S_{h,n}^{\mathrm{p}}(\Omega) := \Big\{ \sum_{i=1}^M {p_h^i} \medspace|\medspace {p_h^i}\in S^{\mathrm{p}}_{h,n_i}(\omega_i^0) \Big\},\\
	&\boldsymbol{V}_{h,n}^{\operatorname{en}}(\Omega) := \Big\{ \sum_{i=1}^M  \boldsymbol{u}_h^i \medspace|\medspace \boldsymbol{u}_h^i\in \boldsymbol{V}_{h,n_i}^{\operatorname{en}}(\omega_i^0) \Big\},\nonumber
\end{align}
where  $\Pi_h$ is the Raviart--Thomas interpolation operator on $\Omega$ \ma{(see \cite[Example 2.5.3]{boffi2013mixed} for the definition)}.
Moreover, given the local particular functions $\boldsymbol{u}_{h,i}^{\mathrm{par}}$ and $p_{h,i}^{\mathrm{par}}$, we define the global particular functions by
\begin{align}
	\label{def_glob_par_disc}
	\boldsymbol{u}_h^{\mathrm{par}} := \sum\limits_{i=1}^M \Pi_h(\chi_i \boldsymbol{u}_{h,i}^{\mathrm{par}}), \quad p_h^{\mathrm{par}} := \sum\limits_{i=1}^M p_{h,i}^{\mathrm{par}}.
\end{align}
Let $\boldsymbol{V}^{\mathrm{RT}}_0/Q^{\mathrm{RT}}_0$ be the Raviart--Thomas pair with respect to the partition $\{\omega_i^{0}\}$ as \ma{in section \ref{sec:cont}}. We define the multiscale approximation spaces by
\begin{align}
	\label{def:discreteMSspaces}
	\boldsymbol{V}_h^{\mathrm{MS}} 
	:= \boldsymbol{S}^{\mathrm{v}}_{h,n}(\Omega)	
	+ \boldsymbol{V}_{h,n}^{\operatorname{en}}(\Omega)
	+ \boldsymbol{V}^{\mathrm{RT}}_0
	\quad\text{and}\quad 
	Q_h^{\mathrm{MS}} 
	:= S^{\mathrm{p}}_{h,n}(\Omega)	
	+ Q^{\mathrm{RT}}_0.    
\end{align}
The GFEM approximation of problem (\ref{globWeakAug_disc}) with the above multiscale spaces is defined by: 
Find $\boldsymbol{u}_h^{\mathrm{G}} \in \boldsymbol{u}_h^{\mathrm{par}}+\boldsymbol{V}_h^{\mathrm{MS}}$ and $p_h^{\mathrm{G}}\in p_h^{\mathrm{par}}+Q_h^{\mathrm{MS}}$ such that
\begin{equation}
	\label{GFEM_disc}
	\left\{\begin{array}{lll}
		a^\gamma(\boldsymbol{u}_h^{\mathrm{G}},\boldsymbol{v}_h) + b(\boldsymbol{v}_h,p_h^{\mathrm{G}}) &= -\gamma b(\boldsymbol{v}_h,f) 
		&\forall \boldsymbol{v}_h\in \boldsymbol{V}_h^{\mathrm{MS}}, \\[2mm]
		b(\boldsymbol{u}_h^{\mathrm{G}},q_h) &= -(f,q_h)_{L^2(\Omega)}  & \forall q_h\in Q_h^{\mathrm{MS}}. \\
	\end{array}\right.
\end{equation}


\subsection{Local approximations} 
In this subsection, we construct the local particular functions and local approximation spaces in the same spirit as the continuous mixed MS-GFEM.

\subsubsection{Local particular functions}
\label{subsec_locPart_disc}
We consider the following local problems similar to (\ref{partWeakNeu}): Find ${{\boldsymbol{\psi}}_{h,i}} \in \boldsymbol{V}_{h}(\omega_i^*)$ and $ \phi_{h,i} \in Q_{h,0}(\omega_i^*)$ such
that 
\begin{equation}
	\label{partWeakDisc}
	\left\{\begin{array}{lll}
		a_{\omega_i^*}({\boldsymbol{\psi}}_{h,i},{\boldsymbol{v}_h}) + b_{\omega_i^*}({\boldsymbol{v}_h},\phi_{h,i}) &= 0 &\forall {\boldsymbol{v}_h}\in \boldsymbol{V}_{h,0}(\omega_i^*), \\
		b_{\omega_i^*}({\boldsymbol{\psi}}_{h,i},q_h) &= -(f,q_h)_{L^2(\omega_i^*)}   & \forall q_h\in Q_h(\omega_i^*), \\
		{\boldsymbol{\psi}}_{h,i}\cdot \boldsymbol{\nu} &= 0, & \text{on } \partial\omega_i^*\cap\partial\Omega,\\
		{\boldsymbol{\psi}}_{h,i}\cdot \boldsymbol{\nu} &= C_{comp}, & \text{on } \partial\omega_i^*\cap\Omega,\\
	\end{array}\right.
\end{equation}
where the constant $C_ {comp}$ is chosen such that the compatibility condition is satisfied. 
Then, we define the local particular functions by
\begin{equation}
	\label{def_loc_part_disc}
	\boldsymbol{u}_{h,i}^{\mathrm{par}} := {\boldsymbol{\psi}}_{h,i}|_{\omega_i} \quad \text{ and }\quad p^{\mathrm{par}}_{h,i} :=
	\phi_{h,i}|_{\omega^0_i} - \dashint_{\omega_i^0}\phi_{h,i}\dx.
\end{equation}
Combining (\ref{neuWeakDisc}) and (\ref{partWeakDisc}), we see that $\boldsymbol{u}_h^{\mathrm{e}}|_{\omega_i^*} - {\boldsymbol{\psi}}_{h,i} \in \boldsymbol{W}_h(\omega_i^*)$, where $\boldsymbol{W}_h(\omega_i^*)$ is the discrete generalized $a$-harmonic space.

\subsubsection{Local approximation spaces}
\label{velEigProbDisc}
We consider the operator 
\begin{align}
	\label{defPh}
	P_h : (\boldsymbol{W}_h(\omega_i^*), \hdinorm{\cdot}{\omega_i^*;a}) 
	\longrightarrow (\boldsymbol{W}_{h}(\omega_i), \hdinorm{\cdot}{\omega_i;a})
	,\quad \boldsymbol{v}_h\mapsto \boldsymbol{v}_h|_{\omega_i},
\end{align}
and define the associated Kolmogorov $n$-width $d_{h,n}(\omega_i,\omega_i^*)$ similarly to (\ref{n-width}).
Clearly, $P_h$ is compact and we have the following result.
\begin{proposition}\label{local-estimate-discrete}
	Let the local approximation space for the velocity on $\omega_i$ be defined as
	\begin{equation}
		\label{def_loc_vel_disc}
		\boldsymbol{S}^{\mathrm{v}}_{h,n_i}(\omega_i) := {\rm span}\,\{\boldsymbol{v}_{h,1}|_{\omega_i}, \cdots, \boldsymbol{v}_{h,n_i}|_{\omega_i} \},
	\end{equation}
	where $\boldsymbol{v}_{h,k}$ denotes the $k$-th eigenfunction of the eigenproblem 
	\begin{equation}
		\label{EVP1_disc}
		a_{\omega_i^*}(\boldsymbol{v}_{h,k},{\boldsymbol{\phi}_h}) 
		= \lambda a_{\omega_i}(\boldsymbol{v}_{h,k}, {\boldsymbol{\phi}_h}) 
		\quad \forall {\boldsymbol{\phi}_h}\in \boldsymbol{W}_h(\omega_i^*),
	\end{equation}
and let $\boldsymbol{u}_h^{\mathrm{e}}$ be as in (\ref{neuWeakDisc}). Then, 
	\begin{equation}
		\inf_{\boldsymbol{v}_h\in \boldsymbol{u}_{h,i}^{\mathrm{par}} + \boldsymbol{S}_{h,n_i}^{\mathrm{v}}(\omega_i)} \hdinorm{\boldsymbol{u}_h^{\mathrm{e}}-\boldsymbol{v}_h}{\omega_i;a} 
		\leq d_{h,n_i}(\omega_i,\omega_i^*) \hdinorm{\boldsymbol{u}_h^{\mathrm{e}}-{\boldsymbol{\psi}}_{h,i}}{\omega_i^*;a}.
	\end{equation}
\end{proposition}

Next, we use the space $\boldsymbol{S}^{\mathrm{v}}_{h,n_i}(\omega_i)$ to construct the local pressure approximation space. Let $\{\boldsymbol{v}_{h,k}\}$ be as in \cref{local-estimate-discrete}. For each $k=1,\cdots,n_i$, we define $(\tilde{\boldsymbol{v}}_{h,k},p_{h,k})\in \boldsymbol{V}_h({\omega_i^0})\times Q_{h,0}(\omega_i^0)$ to be the solution of
\begin{equation}
	\label{eq_pressure_recon_disc}
	\left\{\begin{array}{lll}
		a_{\omega_i^0}(\tilde{\boldsymbol{v}}_{h,k},\boldsymbol{w}_h) + b_{\omega_i^0}(\boldsymbol{w}_h,p_{h,k}) &= 0 &\forall \boldsymbol{w}_h\in \boldsymbol{V}_{h,0}(\omega_i^0), \\
		b_{\omega_i^0}(\tilde{\boldsymbol{v}}_{h,k},q_h) &= 0  &\forall q_h\in Q_{h}(\omega_i^0), \\
		\tilde{\boldsymbol{v}}_{h,k}\cdot\boldsymbol{\nu} &= \boldsymbol{v}_{h,k}\cdot\boldsymbol{\nu} &\text{on } \partial\omega_i^0.
	\end{array}\right.
\end{equation}
The local pressure approximation space is defined by
\begin{equation}
	\label{def_loc_pres_disc}
	S^{\mathrm{p}}_{h,n_i}(\omega_i^0):=\text{span}\{ p_{h,1}, \cdots, p_{h,n_i}\}.
\end{equation}

The following theorem shows that the $n$-width $d_{h,n}(\omega_i,\omega_i^*)$ has a similar exponential decay rate as its continuous counterpart. The proof is given in section \ref{sec:proofs_nearlyExponentialDisc}.
\begin{theorem}
	\label{theorem_nearly_exponential_disc}
	For every $\epsilon >0$, there exists $n_{\epsilon}$, such that for any $n>n_{\epsilon}$ and $h$ small enough,
	\begin{equation*}
		d_{h,n}(\omega_i,\omega_i^*) 
		\leq e^{-n^{\left(\frac{1}{d+1}-\epsilon\right)}}.
	\end{equation*} 
\end{theorem}

\subsection{Inf-sup stability}
\ma{As in the continuous case}, we enrich the multiscale space for the velocity to obtain inf-sup stability. Recall (\ref{def_loc_pres_disc}) and for each $p_{h,k}$, we denote by $(\boldsymbol{u}_{h,k},\tilde{p}_{h,k})\in \boldsymbol{V}_{h,0}(\omega_i^0)\times Q_{h,0}(\omega_i^0)$ the unique solution of
\begin{equation}
	\label{enrProb_disc}
	\left\{\begin{array}{lll}
		a_{\omega_i^0}(\boldsymbol{u}_{h,k},\boldsymbol{v}_h) +b_{\omega_i^0}(\boldsymbol{v}_h,\tilde{p}_{h,k}) &= 0 &\forall \boldsymbol{v}_h\in \boldsymbol{V}_{h,0}(\omega_i^0), \\
		b_{\omega_i^0}(\boldsymbol{u}_{h,k},q_h) &= (p_{h,k},q_h)_{L^2(\omega_i^0)}  &\forall q_h\in Q_{h}(\omega_i^0), \\
		{\boldsymbol{u}_{h,k}}\cdot\boldsymbol{\nu} &= 0 &\text{ on } \partial\omega_i^0.
	\end{array}\right.
\end{equation}
We define the local velocity enrichment space by
\begin{equation}
	\label{def_loc_enr_disc}
	\boldsymbol{V}_{h,n_i}^{\operatorname{en}}(\omega_i^0) := \text{span}\{\boldsymbol{u}_{h,1},\cdots,\boldsymbol{u}_{h,n_i} \}. 
\end{equation}
The following inf-sup stability can be proved analogously to \cref{IS_MS}.
\begin{theorem}
	\label{IS_MS_disc}
	The pair $\boldsymbol{V}_h^{\mathrm{MS}}/Q_h^{\mathrm{MS}}$ is inf-sup stable, i.e.
	\begin{equation*}
		\sup_{\boldsymbol{u}_h^{\mathrm{MS}}\in \boldsymbol{V}_h^{\mathrm{MS}}}\frac{b(\boldsymbol{u}_h^{\mathrm{MS}},p_h^{\mathrm{MS}})}{\hdinorm{\boldsymbol{u}_h^{\mathrm{MS}}}{\Omega;a}\lnorm{p_h^{\mathrm{MS}}}{\Omega}}
		\geq \beta_h^{\mathrm{MS}}:= \ma{\Big(\frac{2}{\beta_{h,\mathrm{min}}}+\frac{1}{\beta^{\operatorname{RT}}}\Big)^{-1}}, \quad \forall p_h^{\mathrm{MS}}\in Q_h^{\mathrm{MS}},
	\end{equation*}
 where ${\beta_{h,\mathrm{min}}} := \min\limits_{i=1,\cdots,M} \beta_h(\omega_i^0)$, with $\beta_h(\omega_i^0)$ denoting the inf-sup constant for the pair $\boldsymbol{V}_{h,0}(\omega_i^0)/Q_{h,0}(\omega_i^0)$ \ma{equipped with the norms $\hdinorm{\cdot}{\omega_{i}^0;a}/\Vert\cdot\Vert_{L^{2}(\omega_{i}^{0})}$}, and $\beta^{\operatorname{RT}}$ is as in \cref{IS_MS}.
\end{theorem}
\begin{remark}
\ma{Although for clarity they contain an} index $h$, the discrete inf-sup constants are actually independent of $h$ \ma{for a shape-regular family of meshes as assumed}. 
\end{remark}
\subsection{Global approximation  error estimates}
As in the continuous setting, we assume the stability of the discrete problems (\ref{neuWeakDisc}) and (\ref{partWeakDisc}):
\begin{align}
	\label{eq:stability_inequality_disc}
	\lnorm{\boldsymbol{{u}}_h^e}{\Omega;a}
	\leq C_{\operatorname{stab},h} \lnorm{f}{\Omega},\quad \lnorm{\boldsymbol{{\psi}}_{h,i}}{\omega_i^*;a}
	\leq C^i_{\operatorname{stab},h} \lnorm{f}{\omega_i^*} \quad 1\leq i\leq M.
\end{align}
In practice, the constants $C_{\operatorname{stab},h}$ and $C^i_{\operatorname{stab},h}$ do not depend on $h$.  
\begin{theorem}
	\label{GlobLocErr_disc2}
	Assume that there exist $\boldsymbol{v}_h^i\in \boldsymbol{S}^{\mathrm{v}}_{h,n_i}(\omega_i)$ for $i=1,\cdots,M$ such that
	\begin{equation*}							    \hdinorm{\boldsymbol{u}^{\mathrm{e}}_h-\boldsymbol{u}_{h,i}^{\mathrm{par}}-\boldsymbol{v}_h^i}{\omega_i;a} 
		\leq \epsilon_i \hdinorm{\boldsymbol{u}^{\mathrm{e}}_h - {\boldsymbol{\psi}}_{h,i}}{\omega_i;a}.
	\end{equation*}
	Let $\boldsymbol{v}_{h}=\sum_{i=1}^{M}\Pi_{h}(\chi_{i}\boldsymbol{v}^{i}_{h}) \in \boldsymbol{S}^{\mathrm{v}}_{h,n}(\Omega)$ and $\displaystyle \epsilon_{\rm max}=\max_{1\leq i\leq M}\epsilon_i$. Then, 
	\begin{align*}
		\lnorm{\boldsymbol{u}_h^{\mathrm{e}}-\boldsymbol{u}_h^{\mathrm{par}}-\boldsymbol{v}_h}{\Omega;a}
		&\leq C\sqrt{\kappa\kappa^*} \big(\alpha_1/\alpha_0\big)^{1/2}\tilde{C}_{h}
		\epsilon_{\rm max},\\
		\lnorm{\di\left(\boldsymbol{u}_h^{\mathrm{e}}-\boldsymbol{u}_h^{\mathrm{par}}-\boldsymbol{v}_h\right)}{\Omega}&\leq \sqrt{2\kappa\kappa^*}\alpha_1^{1/2}\big(\max\limits_{1\leq i\leq M}\norm{\nabla\chi_i}{L^\infty(\omega_i)}\big)\tilde{C}_{h}\epsilon_{\rm max},\\ 
        \inf_{p_h\in Q_h^{\mathrm{MS}}} \lnorm{p^{\mathrm{e}}_h-p_h^{\mathrm{par}}-p_h}{\Omega}&\leq \sqrt{2\kappa^*} \beta^{-1}_{h,\mathrm{min}}\tilde{C}_{h} \epsilon_{\rm max},
	\end{align*}
	where $C$ depends only on the shape regularity of the mesh $\mathcal{T}_h$, $\tilde{C}_{h}=\big(C_{\operatorname{stab},h} + C_{\operatorname{max},h}\big)\lnorm{f}{\Omega}$ with $C_{\operatorname{max},h} = \max_{1\leq i\leq M} C_{\operatorname{stab},h}^i$,  $\kappa$ and $\kappa^*$ are defined in (\ref{pointwise_overlap}), and $\beta_{h,\mathrm{min}}$ is defined in \cref{IS_MS_disc}.
\end{theorem}
To prove \cref{GlobLocErr_disc2}, we need the following lemma concerning stability of the interpolation operator $\Pi_h$.
	\begin{lemma}
		\label{lem_approx_RT}
		Let $\boldsymbol{u}_h\in \boldsymbol{K}_h(\omega_i)$ and $\chi_{i}$ be the partition of unity function supported on $\omega_i$. Then,
		\begin{align}
\big\Vert{\Pi_h(\chi_{i} \boldsymbol{u}_h)}\big\Vert_{L^{2}(\omega_i)} &\leq C \lnorm{\boldsymbol{u}_h}{\omega_i},\label{L2-stability}\\
\big\Vert{\rm div}\,\Pi_h(\chi_{i} \boldsymbol{u}_h)\big\Vert_{L^{2}(\omega_i)} &\leq \norm{\nabla\chi_i}{L^\infty(\omega_i)} \lnorm{\boldsymbol{u}_h}{\omega_i},\label{div-stability}
  \end{align}
where $C>0$ depends only on the shape regularity of the mesh. 
  \end{lemma}
	\begin{proof}
		We use interpolation error estimates (see \cite[Proposition 2.5.1]{boffi2013mixed}) and an inverse estimate to deduce
		\begin{align*}
			\begin{split}
				&\lnorm{\Pi_h(\chi_{i} \boldsymbol{u}_h)-\chi_{i} \boldsymbol{u}_h}{\omega_i}^{2}
				= \sum_{K\cap\omega_i\neq\emptyset}\lnorm{\Pi_h(\chi_i \boldsymbol{u}_h)-\chi_i \boldsymbol{u}_h}{K}^{2}
				\leq C\sum_{K\cap\omega_i\neq\emptyset}h^{2}_K|\chi_i \boldsymbol{u}_h|^{2}_{H^1(K)}\\
				&\leq C  \sum_{K\cap\omega_i\neq\emptyset}(\norm{\chi_i}{L^{\infty}(\omega_i)}+h_K \norm{\nabla\chi}{L^\infty(\omega_i)})^{2}\lnorm{\boldsymbol{u}_h}{K}^{2}
				\leq C  \sum_{K\cap\omega_i\neq\emptyset}\lnorm{\boldsymbol{u}_h}{K}^{2},
			\end{split}
		\end{align*}
		where we have assumed that $h_{K}\norm{\nabla\chi}{L^\infty(\omega_i)}\leq 1$ without loss of generality. Inequality (\ref{L2-stability}) follows immediately from the above estimate. We proceed to prove the second inequality. Let $P_{h}:L^{2}(\Omega)\rightarrow Q_{h}$ be the $L^{2}(\Omega)$ projection. It follows from the commuting property of $\Pi_h$ and $P_{h}$ (see \cite[pp.149]{monk2003finite}) that
  \begin{equation*}
 \big\Vert{\rm div}\,\Pi_h(\chi_{i} \boldsymbol{u}_h)\big\Vert_{L^{2}(\omega_i)} = \Vert P_{h}\big({\rm div}\,(\chi_{i} \boldsymbol{u}_h)\big)\big\Vert_{L^{2}(\omega_i)}\leq  \Vert{\rm div}\,(\chi_{i} \boldsymbol{u}_h)\big\Vert_{L^{2}(\omega_i)},   
  \end{equation*}
which gives (\ref{div-stability}) since $\boldsymbol{u}_h\in \boldsymbol{K}_h(\omega_i)$ is divergence free. 
\end{proof}
	
	\begin{proof}[Proof of \cref{GlobLocErr_disc2}]
Note that $\boldsymbol{u}_h^{\mathrm{e}}|_{\omega_i}-\boldsymbol{u}_{h,i}^{\mathrm{par}}-\boldsymbol{v}_{h,i}\in \boldsymbol{K}_h(\omega_i)$ for $i=1,\cdots,M$.	With the definition of $\boldsymbol{u}_h^{\mathrm{par}}$ and $\boldsymbol{v}_h$ in mind and using (\ref{L2-stability}) and (\ref{pointwise_overlap}), we see that
		\begin{align}
			\label{pf_l2_err_disc_1}
			\begin{split}
				&\lnorm{\boldsymbol{u}_h^{\mathrm{e}}-\boldsymbol{u}_h^{\mathrm{par}}-\boldsymbol{v}_h}{\Omega;a}^{2}\leq
				\kappa\sum_{i=1}^{M}\big\Vert \Pi_h\big(\chi_{i}(\boldsymbol{u}_h^{\mathrm{e}}-\boldsymbol{u}_{h,i}^{\mathrm{par}}-\boldsymbol{v}_{h,i})\big)\big\Vert^{2}_{L^2(\omega_i;a)}\\
				&\leq C\kappa\big(\frac{\alpha_1}{\alpha_{0}}\big) \sum_{i=1}^{M}\big\Vert \boldsymbol{u}_h^{\mathrm{e}}-\boldsymbol{u}_{h,i}^{\mathrm{par}}-\boldsymbol{v}_{h,i}\big\Vert^{2}_{L^2(\omega_i;a)}.
			\end{split}
		\end{align}
Similarly, for the divergence part, we use (\ref{div-stability}) to deduce
		\begin{align}
			\begin{split}
				\label{pf_l2_err_disc_2}
				&\lnorm{\di\left(\boldsymbol{u}_h^{\mathrm{e}}-\boldsymbol{u}_h^{\mathrm{par}}-\boldsymbol{v}_h\right)}{\Omega}^{2}
				\leq \kappa\sum_{i=1}^{M}\big\Vert {\rm div}\,\Pi_h\big(\chi_{i}(\boldsymbol{u}_h^{\mathrm{e}}-\boldsymbol{u}_{h,i}^{\mathrm{par}}-\boldsymbol{v}_{h,i})\big)\big\Vert^{2}_{L^{2}(\omega_i)}\\
				& \leq \kappa\alpha_{1}\max\limits_{1\leq i\leq M}\Vert \nabla\chi_{i}\Vert^{2}_{L^{\infty}(\omega_i)} \sum_{i=1}^{M} \big\Vert \boldsymbol{u}_h^{\mathrm{e}}-\boldsymbol{u}_{h,i}^{\mathrm{par}}-\boldsymbol{v}_{h,i}\big\Vert^{2}_{L^2(\omega_i;a)}.
			\end{split}
		\end{align}
		With (\ref{pf_l2_err_disc_1}) and (\ref{pf_l2_err_disc_2}), the error estimates for the divergence follows from similar arguments as in \cref{GlobLocErr}. The estimate for the pressure can be proved exactly as in the continuous setting.
	\end{proof}

With the global ellipticity of the bilinear form $a^{\gamma}$ and the discrete inf-sup stability (\cref{IS_MS_disc}) in hand, we can obtain the identical quasi-optimality estimates for our method as \ma{for the continuous setting in} \cref{quasi-optimality}. \ma{Having now} developed the discrete analogue of the continuous mixed MS-GFEM, and \ma{having} proved similar global error estimates for the discrete method, \ma{it} remains to prove the exponential decay of the Kolmogrov $n$-widths \ma{for both settings}. As will be seen in the next section, the proof \ma{for the discrete case} is much more complicated \ma{than the proof for the continuous one}. \ma{It} needs a careful analysis of the underlying finite element spaces.

\section{Proof of exponential decay of the Kolmogorov $n$-widths}\label{sec:proofs}
\subsection{Continuous setting}
We start with proving \cref{Caccio} which gives the Caccioppoli inequality.
\label{sec:proofs_nearly_continuous}
\begin{proof}[Proof of \cref{Caccio}]
	The proofs for (i) and (ii) use the same arguments and hence we only consider (ii). A direct calculation shows
	\begin{align}
		\label{Caccio_product}
		\begin{split}
			a_{\omega^{\ast}}( \cu (\eta {\boldsymbol{\tilde{u}}}),\cu (\eta {\boldsymbol{\tilde{v}}}))
			&= a_{\omega^{\ast}}(\nabla\eta\times{\boldsymbol{\tilde{u}}},\nabla\eta\times{\boldsymbol{\tilde{v}}})- a_{\omega^{\ast}}(\eta\cu{\boldsymbol{\tilde{u}}},\nabla\eta\times{\boldsymbol{\tilde{v}}})\\
			&+ a_{\omega^{\ast}}(\nabla\eta\times{\boldsymbol{\tilde{u}}},\eta\cu{\boldsymbol{\tilde{v}}}) +a_{\omega^{\ast}}(\cu{\boldsymbol{\tilde{u}}},\cu(\eta^2{\boldsymbol{\tilde{v}}})).
		\end{split}
	\end{align}
	Since $\eta\in W^{1,\infty}(\omega^*)$ with
	$\eta = 0$ on $\Omega\cap\partial\omega^*$ and $\boldsymbol{\tilde{v}}\in (H^1_N(\omega^*))^3$, it follows that
	$\eta^2{\boldsymbol{\tilde{v}}}\in (H^1_0(\omega^*))^3\subset \hcu{0}{}{\omega^*}$, and thus $a_{\omega^{\ast}}(\cu{\boldsymbol{\tilde{u}}},\cu(\eta^2{\boldsymbol{\tilde{v}}}))=0$ since ${\boldsymbol{\tilde{u}}}\in \boldsymbol{\tilde{H}}_a(\omega^*)$. Therefore, equality (\ref{Caccio_product}) becomes 
	\begin{align}
       \label{Caccio_product_reduced}
		\begin{split}
			&a_{\omega^{\ast}}( \cu (\eta {\boldsymbol{\tilde{u}}}),\cu (\eta {\boldsymbol{\tilde{v}}}))
			= a_{\omega^{\ast}}(\nabla\eta\times{\boldsymbol{\tilde{u}}},\nabla\eta\times{\boldsymbol{\tilde{v}}})\\
			&\quad - a_{\omega^{\ast}}(\eta\cu{\boldsymbol{\tilde{u}}},\nabla\eta\times{\boldsymbol{\tilde{v}}})+ a_{\omega^{\ast}}(\nabla\eta\times{\boldsymbol{\tilde{u}}},\eta\cu{\boldsymbol{\tilde{v}}}).
		\end{split}
	\end{align}
	Exchanging the roles of ${\boldsymbol{\tilde{u}}}$ and ${\boldsymbol{\tilde{v}}}$ above and adding the resulting equality to (\ref{Caccio_product_reduced}), we see that $a_{\omega^{\ast}}(\cu (\eta {\boldsymbol{\tilde{u}}}),\, \cu (\eta {\boldsymbol{\tilde{v}}}))= a_{\omega^{\ast}}({\boldsymbol{\tilde{u}}}\times \nabla\eta,\,{\boldsymbol{\tilde{v}}}\times \nabla\eta)$. 
\end{proof}

Next we prove \cref{theorem_nearly_exponential}. As shown in \cite{ma2022novel,ma2023unified,babuska2011optimal}, the proof of the exponential decay of the Kolmogorov $n$-width $d_n(\omega,\omega^*)$ hinges on a Caccioppoli-type inequality and a weak approximation estimate. A key difference, however, is that here we do not prove these properties for the generalized $a$-harmonic functions, but instead for the corresponding vector potentials. With the Caccioppoli inequality (\cref{Caccio}) in hand, now it remains to prove the weak approximation estimate.

\begin{lemma}
	\label{l2_approx_vecpot}
	Let $d=3$, and let $\mathcal{D}^{\ast}\subset \Omega$ be a Lipschitz domain. There exists an $m$-dimensional space $\tilde{R}_{m}(\mathcal{D}^{\ast}) \subset \boldsymbol{\tilde{H}}_a(\mathcal{D}^{\ast})$
	such that 
	\begin{equation*}
		\inf\limits_{\boldsymbol{\tilde{v}}\in \tilde{R}_{m}(\mathcal{D}^{\ast})}\lnorm{\boldsymbol{\tilde{u}}-\boldsymbol{\tilde{v}}}{\mathcal{D}^{\ast}}
		\leq Cm^{-1/d}|\mathcal{D}^*|^{1/d}|\boldsymbol{\tilde{u}}|_{H^1(\mathcal{D}^{\ast})} \quad\forall \boldsymbol{\tilde{u}}\in \boldsymbol{\tilde{H}}_a(\mathcal{D}^{\ast}),
	\end{equation*}
        where $C>0$ depends on the shape of $\mathcal{D}^{\ast}$, but not on its size.
\end{lemma}
\begin{proof}
Denote by
\begin{align*}
		\boldsymbol{V}(\mathcal{D}^{\ast}) &= \big\{\boldsymbol{u}\in \boldsymbol{\tilde{H}}_a(\mathcal{D}^{\ast}):(u_i, 1)_{L^{2}(\mathcal{D}^{\ast})} = 0, \;\; i=1,\cdots,d \big\},\\
  \boldsymbol{W}(\mathcal{D}^{\ast}) &= \big\{\boldsymbol{u}\in (H^{1}(\mathcal{D}^{\ast}))^{d}:(u_i, 1)_{L^{2}(\mathcal{D}^{\ast})} = 0, \;\; i=1,\cdots,d \big\}.
	\end{align*}
Since $\mathcal{D}^{\ast}$ is a Lipschitz domain, we see that $\boldsymbol{V}(\mathcal{D}^{\ast})$ and $\boldsymbol{W}(\mathcal{D}^{\ast})$ are Hilbert spaces when equipped with the $H^{1}$ semi-norm $|\cdot|_{H^{1}(\mathcal{D}^{\ast})}$. Moreover, the embedding $\big(\boldsymbol{V}(\mathcal{D}^{\ast}), |\cdot|_{H^{1}(\mathcal{D}^{\ast})}\big)\to ((L^2(\mathcal{D}^{\ast}))^d,\lnorm{\cdot}{\mathcal{D}^{\ast}})$ is compact, and we denote its Kolmogrov $m$-width by $d_{m}(\boldsymbol{V})$. For each $j\in\mathbb{N}$, let $\boldsymbol{\tilde\phi}_{j}$ be the $j$-th eigenfunction of the eigenproblem
	\begin{equation*}
		(\boldsymbol{\tilde\phi},\boldsymbol{\tilde{v}})_{L^2(\mathcal{D}^{\ast})} = \lambda(\nabla\boldsymbol{\tilde\phi},\nabla\boldsymbol{\tilde{v}})_{L^2(\mathcal{D}^{\ast})} \quad \forall\boldsymbol{\tilde{v}}\in \boldsymbol{V}(\mathcal{D}^{\ast}).
	\end{equation*}
Then by \cite[Theorem 2.5]{pinkus2012n}, $\boldsymbol{V}^{\rm opt}_{m}(\mathcal{D}^{\ast}) = \operatorname{span}\{\boldsymbol{\tilde{\phi}}_{1},\cdots,\boldsymbol{\tilde{\phi}}_{m}\}$ is the optimal approximation space associated with $d_{m}(\boldsymbol{V})$ and thus
	\begin{equation}
		\label{best_epsilon_space}
		\inf\limits_{\boldsymbol{\tilde{v}}\in \boldsymbol{V}^{\rm opt}_{m}(\mathcal{D}^{\ast})}\lnorm{\boldsymbol{\tilde{u}}-\boldsymbol{\tilde{v}}}{\mathcal{D}^{\ast}}
		\leq d_{m}(\boldsymbol{V}) |\boldsymbol{\tilde{u}}|_{H^{1}(\mathcal{D}^{\ast})} \quad\forall \boldsymbol{\tilde{u}}\in \boldsymbol{V}(\mathcal{D}^{\ast}).
	\end{equation}
We note that the inclusion $\boldsymbol{V}(\mathcal{D}^{\ast})\subset \boldsymbol{W}(\mathcal{D}^{\ast})$ implies $d_{m}(\boldsymbol{V}) \leq d_{m}(\boldsymbol{W})$, where $d_{m}(\boldsymbol{W})$ denotes the $m$-width of the compact embedding $\big(\boldsymbol{W}(\mathcal{D}^{\ast}), |\cdot|_{H^{1}(\mathcal{D}^{\ast})}\big)\to (L^2(\mathcal{D}^{\ast}))^d$. Classical approximation theory (e.g., finite elements) further gives that
\begin{equation}\label{best_epsilon_space2}
d_{m}(\boldsymbol{V})\leq d_{m}(\boldsymbol{W}) \leq Cm^{-1/d}|\mathcal{D}^*|^{1/d},    
\end{equation}
where $C>0$ depend on the shape of $\mathcal{D}^*$. Let $\tilde{R}_{m}(\mathcal{D}^{\ast})= \boldsymbol{V}^{\rm opt}_{m-d}(\mathcal{D}^{\ast}) \oplus \mathbb{R}^{d}$. Combining (\ref{best_epsilon_space}) and (\ref{best_epsilon_space2}), and noting that $\boldsymbol{\tilde{H}}_a(\mathcal{D}^{\ast}) = \boldsymbol{V}(\mathcal{D}^{\ast})\oplus \mathbb{R}^{d}$, we get the desired result.
 \end{proof}

Combining \cref{l2_approx_vecpot} and the Caccioppoli inequality leads to the following intermediate approximation result.

  	\begin{lemma}
		\label{lemma_induction_start}
		Let $d=3$, and let $\mathcal{D}\subset\mathcal{D}^{\ast}$ be subsets of $\Omega$ with $\delta={\rm dist}(\mathcal{D},\partial \mathcal{D}^{\ast}\setminus \partial \Omega)>0$. Moreover, we assume that $\mathcal{D}^{\ast}$ is a Lipschitz domain satisfying the hypotheses of \cref{exVecPotnew}. There exists an $m$-dimensional space $R_m(\mathcal{D}^{\ast})\subset \boldsymbol{H}_a(\mathcal{D}^{\ast})$ such that for all ${\boldsymbol{{u}}}\in \boldsymbol{{H}}_a(\mathcal{D}^{\ast})$,
		\begin{equation}\label{aux-estimate} 
			\inf_{{\boldsymbol{{w}}}\in R_m(\mathcal{D}^{\ast})}\lnorm{ {\boldsymbol{{u}}}-{\boldsymbol{{w}}}}{\mathcal{D};a}
			\leq  C\big(\alpha_{1}/\alpha_{0}\big)^{1/2}m^{-1/d}|\mathcal{D}^{\ast}|^{1/d}\delta^{-1} \lnorm{{\boldsymbol{{u}}}}{\mathcal{D}^*;a},
		\end{equation}
		where $C>0$ depends on the shape of $\mathcal{D}^{\ast}$.
	\end{lemma}
	\begin{proof}	  
		Let ${\boldsymbol{{u}}}\in \boldsymbol{{H}}_a(\mathcal{D}^{\ast})$. By \cref{exVecPotnew}, there 
		exists a vector potential ${\boldsymbol{\tilde{v}}}\in (H^1_N(\mathcal{D}^{\ast}))^3$ with
		$
			\cu{\boldsymbol{\tilde{v}}} = \boldsymbol{u}
		$
		such that
		\begin{equation}
			\label{ineq_vec_pot_proof}
			|{\boldsymbol{\tilde{v}}}|_{H^{1}(\mathcal{D}^*)} \leq C\lnorm{\boldsymbol{u}}{\mathcal{D}^{\ast}},
		\end{equation}
  where $C>0$ depends only on the shape of $\mathcal{D}^{\ast}$. We notice that $\boldsymbol{\tilde{v}}\in\boldsymbol{\tilde{H}}_a(\mathcal{D}^{\ast})$. Hence, by \cref{l2_approx_vecpot}, we can choose an ${\boldsymbol{\tilde{w}}}\in \tilde{R}_{m}(\mathcal{D}^{\ast}) \subset \boldsymbol{\tilde{H}}_a(\mathcal{D}^{\ast}) $ with 
		\begin{equation}
			\label{ineq_FEM_approx_proof}
			\lnorm{{\boldsymbol{\tilde{v}}}- {\boldsymbol{\tilde{w}}}}{\mathcal{D}^{\ast}} 
			\leq Cm^{-1/d}|\mathcal{D}^*|^{1/d}|\boldsymbol{\tilde{v}}|_{H^1(\mathcal{D}^{\ast})}.
		\end{equation}
Choose a cut-off function $\eta \in W^{1,\infty}(\mathcal{D}^{\ast})$ such that $\eta =1$ in $\mathcal{D}$, $\eta =0$ on $\Omega\cap \partial \mathcal{D}^{\ast}$, and $\Vert \nabla \eta\Vert_{L^{\infty}(\mathcal{D}^{\ast})}\leq C\delta^{-1}$. Noting that ${\boldsymbol{\tilde{v}}}-{\boldsymbol{\tilde{w}}} \in \boldsymbol{\tilde{H}}_a(\mathcal{D}^{\ast})$, we can apply \cref{Caccio}, (\ref{ineq_vec_pot_proof}), and (\ref{ineq_FEM_approx_proof}) to obtain
		\begin{align}
			\lnorm{\cu ({\boldsymbol{\tilde{v}}}-{\boldsymbol{\tilde{w}}})}{\mathcal{D};a}
			&\leq\lnorm{\cu (\eta({\boldsymbol{\tilde{v}}}-{\boldsymbol{\tilde{w}}}))}{\mathcal{D}^{\ast};a} \leq  \norm{\nabla\eta}{L^\infty(\mathcal{D}^{\ast})}  
			\lnorm{{\boldsymbol{\tilde{v}}}-{\boldsymbol{\tilde{w}}}}{\mathcal{D}^{\ast};a}\nonumber\\[0.5ex]  
			& \leq  C{\alpha_0^{-1/2}} m^{-1/d}|\mathcal{D}^*|^{1/d}\delta^{-1}|\boldsymbol{\tilde{v}}|_{H^1(\mathcal{D}^{\ast})} \label{inter_approx} \\[0.5ex]
			& \leq   C\big(\alpha_{1}/\alpha_{0}\big)^{1/2}m^{-1/d}|\mathcal{D}^{\ast}|^{1/d}\delta^{-1} \lnorm{{\boldsymbol{{u}}}}{\mathcal{D}^{\ast};a}.\nonumber 
		\end{align}
		Let $R_m(\mathcal{D}^{\ast}):=\cu \tilde{R}_m(\mathcal{D}^{\ast})$ and $\boldsymbol{w}:=\cu \boldsymbol{\tilde{w}}$. Recalling that $\cu\boldsymbol{\tilde{H}}_a(\mathcal{D}^{\ast}) = \boldsymbol{H}_a(\mathcal{D}^{\ast})$, we immediately have $R_m(\mathcal{D}^{\ast})\subset \boldsymbol{H}_a(\mathcal{D}^{\ast})$, and the desired result follows from (\ref{inter_approx}).
	\end{proof}

 Now we are ready to prove \cref{theorem_nearly_exponential}.
	\begin{proof}[Proof of \cref{theorem_nearly_exponential}]
 When $d=2$, the $n$-width $d_{n}(\omega,\omega^{\ast})$ can be rewritten as (\ref{equiv-nwidth}). In this case, the exponential decay was proved in \cite[Theorems 3.3 and 3.7]{babuska2011optimal}. The case $d=3$ can be proved similarly by iterating the estimate (\ref{aux-estimate}) on a family of nested domains betweens $\omega$ and $\omega^{\ast}$. The details can be found in \cite{babuska2011optimal,ma2022novel,ma2023wavenumber}.
\end{proof}

	\subsection{Discrete setting}
	\label{sec:proofs_nearlyExponentialDisc}
	
 	In this section, we shall prove \cref{theorem_nearly_exponential_disc}. As in the continuous case, we need to lift the discrete generalized $a$-harmonic space to vector potentials, and prove a discrete Caccioppoli inequality and a weak approximation estimate. The proof, however, is significantly more complicated than its continuous counterpart, which is mainly due to the discrete Caccioppoli inequality. To make it more precise, let us define the space of vector potentials
	\begin{align}\label{discrete_vector_potentials}
		\boldsymbol{\tilde{W}}_h(\omega^*) := \mathcal{S}_{N}(\boldsymbol{W}_h(\omega^*)),
	\end{align}
	where $\mathcal{S}_{N}: \hdi{N}{0}{\omega^*}\to (H^1_N(\omega^*))^d$ is the operator from \cref{exVecPotnew}. It is important to note that $\boldsymbol{\tilde{W}}_h(\omega^*)$ is not contained in $\tilde{\boldsymbol{H}}_a(\omega^*)$, and thus \cref{Caccio} does not hold in this case. On the other hand, Caccioppoli-type inequalities can also be proved for discrete harmonic functions in a finite element setting, where a superapproximation property of the underlying finite element space plays a key role. But here $\boldsymbol{\tilde{W}}_h(\omega^*)$ is not a finite element space! A simple trick to fix this problem is to interpolate $\boldsymbol{\tilde{W}}_h(\omega^*)$ onto a suitable finite element space, and a natural choice in this case is the N\'ed\'elec elements. Then, a Caccioppoli inequality can be proved for the interpolants by means of a superapproximation result for the N\'ed\'elec elements. Due to the interpolation, the proof of \cref{theorem_nearly_exponential_disc} relies heavily on properties of the N\'ed\'elec elements, especially the following commuting property of interpolation operators. Let $\mathcal{N}_0(\omega^*,\mathcal{T}_h)$ be the space of the lowest-order N\'ed\'elec elements on $\omega^*$, and $\Sigma_h$ be the associated interpolation operator (it is well defined on $\boldsymbol{\tilde{W}}_h(\omega^*)$ by \cite[Lemma 5.38]{monk2003finite}). Then, it holds that (see, e.g., \cite[Lemma 5.40]{monk2003finite})
	\begin{equation}\label{commuting-property}
		\cu \Sigma_h \boldsymbol{\tilde{u}} = \Pi_h \cu \boldsymbol{\tilde{u}} = \cu \boldsymbol{\tilde{u}}\quad \forall \boldsymbol{\tilde{u}}\in \boldsymbol{\tilde{W}}_h(\omega^*),
	\end{equation}
	where $\Pi_h$ is the Raviart-Thomas interpolation operator onto the space $\boldsymbol{V}_h(\omega^*)$. 
 
	We start by stating the abovementioned superapproximation result. The proof follows the ideas of \cite{demlow2011local} where a similar result was proved for Lagrange elements.

	\begin{lemma}
		\label{superapproximation}
		Let $\eta \in C^{\infty}(\omega^*)$ with $|\eta|_{W^{j,\infty}(\omega^*)}\leq C\delta^{-j}$ for $0\leq j\leq 2$. Then for each $\boldsymbol{\tilde{u}}_h\in \mathcal{N}_0(\omega^*,\mathcal{T}_h)$ and $K\in\mathcal{T}_h$ with $h_K:=\operatorname{diam}(K)\leq \textcolor{black}{\delta}$,
		\begin{align*}
			\big\Vert\cu\left[\eta^2\boldsymbol{\tilde{u}}_h - \Sigma_h(\eta^2\boldsymbol{\tilde{u}}_h)\right]\big\Vert_{L^{2}(K)}
			\leq C\left(\frac{h_K}{\delta}\lnorm{\cu(\eta \boldsymbol{\tilde{u}}_h)}{K} + \frac{h_K}{\delta^2}\lnorm{\boldsymbol{\tilde{u}}_h}{K}\right),
		\end{align*}
		where the constant $C$ depends only on the shape-regularity of the mesh.
	\end{lemma}   
\begin{proof}
Note that the interpolant $\Sigma_h(\eta^2\boldsymbol{\tilde{u}}_h)$ is well defined. By \cite[Proposition 2.5.6]{boffi2013mixed} there exists a constant $C$, depending on the shape of $K$, such that
	\begin{align}
		\label{pf_sup_1}
		\begin{split}
			&\big\Vert\cu\left[\eta^2\boldsymbol{\tilde{u}}_h - \Sigma_h(\eta^2\boldsymbol{\tilde{u}}_h)\right]\big\Vert_{L^{2}(K)}\\[0.5ex]
			&\leq C h_K |\cu(\eta^2\boldsymbol{\tilde{u}}_h)|_{H^1(K)} \leq C h_K^{5/2} |\cu(\eta^2\boldsymbol{\tilde{u}}_h)|_{W^{1,\infty}(K)}\\
			&\leq C h_K^{5/2} \sum\limits_{i=1}^3 \Big[
			\norm{\eta^2 \cu(\partial_i \boldsymbol{\tilde{u}}_h)}{L^{\infty}(K)}
			+\norm{\nabla(\eta^2)\times\partial_i \boldsymbol{\tilde{u}}_h}{L^{\infty}(K)} \\
			& \qquad +\norm{\partial_i(\eta^2)\cu \boldsymbol{\tilde{u}}_h}{L^{\infty}(K)}
			+\norm{\nabla(\partial_i(\eta^2))\times \boldsymbol{\tilde{u}}_h}{L^{\infty}(K)}
			\Big].
		\end{split}
	\end{align}
	Next, we investigate the four terms appearing on the right hand side above. The first term $\norm{\eta^2 \cu(\partial_i \boldsymbol{\tilde{u}}_h)}{L^{\infty}(K)}$ vanishes, because every component of $\boldsymbol{\tilde{u}}_h$ is a polynomial of degree at most one. The fourth term can be bounded by
	\begin{align}\label{pf_1_term}
		Ch_K^{5/2}\norm{\nabla(\partial_i(\eta^2))\times \boldsymbol{\tilde{u}}_h}{L^{\infty}(K)}
		\leq C \frac{h_K}{\delta^2} \lnorm{\boldsymbol{\tilde{u}}_h}{K}
	\end{align}
	using inverse estimates (cf. \cite[Lemma 12.1]{Ern_2021} and \cite[Lemma 5.17]{monk2003finite}).
	The third term can be treated similarly to the second term, which we consider now. Let us define $\hat\eta := \dashint_K \eta \medspace \dx$. Then $\norm{\eta-\hat\eta}{L^{\infty}(K)}\leq Ch_K|\eta|_{W^{1,\infty}(K)}\leq C\frac{h_K}{\delta}$. Using inverse estimates, we have
	\begin{align}
		\label{pf_sup_2}
		\begin{split}
			&\norm{\nabla(\eta^2)\times\partial_i \boldsymbol{\tilde{u}}_h}{L^{\infty}(K)}
			\leq 2 \norm{\nabla\eta}{L^{\infty}(K)}\norm{\eta\partial_i \boldsymbol{\tilde{u}}_h}{L^{\infty}(K)}\\
			& \leq \frac{C}{\delta}\left(\norm{(\eta -\hat\eta)\partial_i \boldsymbol{\tilde{u}}_h}{L^{\infty}(K)} 
			+ \norm{\hat\eta\partial_i \boldsymbol{\tilde{u}}_h}{L^{\infty}(K)}\right)\\
			& \leq C  \frac{h_K}{\delta^{2}}\norm{\partial_i \boldsymbol{\tilde{u}}_h}{L^{\infty}(K)} 
			+ \frac{C}{\delta}\norm{\hat\eta\cu \boldsymbol{\tilde{u}}_h}{L^{\infty}(K)}\\
			& \leq C  \frac{h_K^{-3/2}}{\delta^{2}} \norm{\boldsymbol{\tilde{u}}_h}{L^{2}(K)} 
			+ C\frac{h_K^{-3/2}}{\delta}\left(\lnorm{\cu [(\hat\eta-\eta)\boldsymbol{\tilde{u}}_h]}{K}
			+ \lnorm{\cu (\eta \boldsymbol{\tilde{u}}_h)}{K}\right)\\
			& \leq C  \frac{h_K^{-3/2}}{\delta^{2}} \norm{\boldsymbol{\tilde{u}}_h}{L^{2}(K)} 
			+ C\frac{h_K^{-3/2}}{\delta}\lnorm{\cu (\eta \boldsymbol{\tilde{u}}_h)}{K}.
		\end{split}
	\end{align}
	In the third inequality of (\ref{pf_sup_2}), we used the fact that there is a constant $C$, independent of $K$ and $\boldsymbol{\tilde{u}}_h$, such that 
	\begin{align*}
		\norm{\partial_i \boldsymbol{\tilde{u}}_h}{L^{\infty}(K)} \leq C \norm{\cu \boldsymbol{\tilde{u}}_h}{L^{\infty}(K)}.
	\end{align*} 
	To see this, we note that $\boldsymbol{\tilde{u}}_h$ is of the form 
	\begin{align*}
		\boldsymbol{\tilde{u}}_h = a_1 {\bm e}_1 + a_2 {\bm e}_2 + a_3 {\bm e}_3 + b_1 (0,z,-y)^T +b_2 (-z,0,x)^T + b_3(y,-x,0)^T
	\end{align*}
	where ${\bm e}_1,{\bm e}_2,{\bm e}_3$ is the canonical basis of $\mathbb{R}^3$ and $a_i,b_i\in\mathbb{R}$. Then
	\begin{align*}
		\norm{\cu \boldsymbol{\tilde{u}}_h}{L^{\infty}(K)} = 2 (|b_1|+|b_2|+|b_3|),
	\end{align*} 
	and clearly
	\begin{equation*}
		\norm{\partial_i \boldsymbol{\tilde{u}}_h}{L^{\infty}(K)} \leq C (|b_1|+|b_2|+|b_3|).
	\end{equation*}
	In the last inequality of (\ref{pf_sup_2}), we used the product rule and an inverse estimate to deduce 
	\begin{align*}
		\lnorm{\cu [(\hat\eta-\eta)\boldsymbol{\tilde{u}}_h]}{K}
		&\leq \left( |\eta|_{W^{1,\infty}(K)}\lnorm{\boldsymbol{\tilde{u}}_h}{K}
		+ \norm{\hat\eta-\eta}{L^{\infty}(K)}\lnorm{\cu\boldsymbol{\tilde{u}}_h}{K} \right)\\
		&\leq C\left(\delta^{-1}\lnorm{\boldsymbol{\tilde{u}}_h}{K}
		+ h_K \delta^{-1}\lnorm{\cu\boldsymbol{\tilde{u}}_h}{K} \right)\\
		&\leq C\delta^{-1}\lnorm{\boldsymbol{\tilde{u}}_h}{K}.
	\end{align*}
	Treating the third term of (\ref{pf_sup_1}) similarly as above and inserting (\ref{pf_1_term}) and (\ref{pf_sup_2}) into (\ref{pf_sup_1}), we prove the lemma. 
\end{proof}

With the superapproximation result in hand, now we can prove the discrete Caccioppoli inequality.
	
	\begin{lemma}
		\label{CaccioDisc}
		Let $\omega\subset\omega^*$ be subdomains of $\Omega$ with $\delta:=\operatorname{dist}(\omega,\partial\omega^*\setminus\partial\Omega)>0$. In addition, let $\max\limits_{K\cap\omega^*\neq\emptyset}h_K\leq \textcolor{black}{\delta}$. Then, for each $\boldsymbol{\tilde{u}}_h:=\Sigma_h \boldsymbol{\tilde{u}}$ with $\boldsymbol{\tilde{u}}\in\boldsymbol{\tilde{W}}_h(\omega^*)$,
		\begin{equation*}
			\lnorm{\cu \boldsymbol{\tilde{u}}_h}{\omega;a}
			\leq C\delta^{-1} \lnorm{{\boldsymbol{\tilde{u}}_h}}{\omega^*},
		\end{equation*}
		where $C>0$ depends on $\alpha_0,\alpha_1$ and the shape-regularity of the mesh, but not on $h$. 
	\end{lemma}
	\begin{proof}
		Let $\eta\in C^{\infty}(\omega^*)$ be a cut-off function that satisfies $0\leq\eta\leq 1$ and
		\begin{equation*}
			\eta=1 \;\;\text{ in } \omega,
			\quad\eta=0 \;\;\text{ on } \partial\omega^*\setminus \partial\Omega,
			\quad  |\eta|_{W^{j,\infty}(\omega^*)}\leq C\delta^{-j},
			\quad 0\leq j\leq 2.
		\end{equation*}
		A direct calculation shows
		\begin{align}\label{discrete-identity}
			\lnorm{\cu (\eta \boldsymbol{\tilde{u}}_h)}{\omega^*;a}^2
			= a_{\omega^{\ast}}(\nabla\eta\times{\boldsymbol{\tilde{u}}_h},\nabla\eta\times{\boldsymbol{\tilde{u}}_h}) 
			+ a_{\omega^{\ast}}(\cu{\boldsymbol{\tilde{u}}_h},\cu(\eta^2{\boldsymbol{\tilde{u}}_h})),
		\end{align}
		and we only need to consider the second term on the right-hand side. Using the commuting property (\ref{commuting-property}) yields
		\begin{align}\label{curl-harmonic}
			\cu \boldsymbol{\tilde{u}}_h = \cu \boldsymbol{\tilde{u}} \in \boldsymbol{W}_h(\omega^*).
		\end{align}
		Next we show that $\cu \Sigma_h(\eta^2 \boldsymbol{\tilde{u}}_h)  \in \boldsymbol{K}_{h,0}(\omega^*)$. Only the vanishing of the normal trace needs some comments. Because $\boldsymbol{\tilde{u}}$ vanishes on $\partial\Omega\cap\partial\omega^*$, ${\bm \nu}\times \Sigma_h \boldsymbol{\tilde{u}} = {\bm 0}$ on $\partial\Omega\cap\partial\omega^*$ (see \cite[Lemma 5.35]{monk2003finite}), where ${\bm \nu}$ denotes the unit outward normal. Noting that $\eta = 0$ on $\partial \omega^{\ast}\cap \Omega$, we further have ${\bm \nu}\times \eta^{2}\boldsymbol{\tilde{u}}_{h} = {\bm \nu}\times \Sigma_{h}(\eta^{2}\boldsymbol{\tilde{u}}_{h}) = {\bm 0}$ on $\partial \omega^{\ast}$. Therefore, $\cu \Sigma_h(\eta^2 \boldsymbol{\tilde{u}}_h) \in \boldsymbol{V}_{h,0}(\omega^*)$ (see \cite[Proposition 2.3.6]{boffi2013mixed}). Using (\ref{curl-harmonic}) and the definition of $\boldsymbol{W}_h(\omega^*)$ gives
  \begin{equation}\label{inter-harmonic}
  a_{\omega^{\ast}}(\cu{\boldsymbol{\tilde{u}}_h},\cu\Sigma_{h}(\eta^2{\boldsymbol{\tilde{u}}_h})) = 0.    
  \end{equation}  
It follows from (\ref{inter-harmonic}), \cref{superapproximation}, an inverse estimate that
		\begin{align*}
			& a_{\omega^{\ast}}(\cu{\boldsymbol{\tilde{u}}_h},\cu(\eta^2{\boldsymbol{\tilde{u}}_h})) = a_{\omega^{\ast}}(\cu{\boldsymbol{\tilde{u}}_h}, \cu(\eta^2{\boldsymbol{\tilde{u}}_h}) - \cu \Sigma_h(\eta^2{\boldsymbol{\tilde{u}}_h}) )\\
			&\leq C \sum_{K\cap\omega^*\neq\emptyset}h_K\lnorm{\cu \boldsymbol{\tilde{u}}_h}{K}\left(\delta^{-1}\lnorm{\cu(\eta \boldsymbol{\tilde{u}}_h)}{K} + \delta^{-2}\lnorm{\boldsymbol{\tilde{u}}_h}{K}\right)\\
			&\leq C \sum_{K\cap\omega^*\neq\emptyset}\lnorm{\boldsymbol{\tilde{u}}_h}{K}\left(\delta^{-1}\lnorm{\cu(\eta \boldsymbol{\tilde{u}}_h)}{K} + \delta^{-2}\lnorm{\boldsymbol{\tilde{u}}_h}{K}\right)\\
			&\leq C \delta^{-2} \lnorm{\boldsymbol{\tilde{u}}_h}{\omega^*}^2
			+ \frac{1}{2} \lnorm{\cu(\eta \boldsymbol{\tilde{u}}_h)}{\omega^*;a}^2.
		\end{align*}
Inserting the estimate above into (\ref{discrete-identity}) yields the desired result.
	\end{proof}

Next we prove a discrete analogue of \cref{l2_approx_vecpot}. Since the Caccioppoli inequality was proved for the interpolants, we need to modify the weak approximation estimate accordingly, so that they can be combined and used iteratively for proving \cref{theorem_nearly_exponential_disc}.

	\begin{lemma}
		\label{l2_approx_vecpotDisc}
Let $m\in\mathbb{N}$. There exists an $m$-dimensional space $\tilde{R}_{h,m}(\omega^*)\subset \boldsymbol{\tilde{W}}_h(\omega^*)$, such that for any $\boldsymbol{\tilde{u}}\in \boldsymbol{\tilde{W}}_h(\omega^*)$,
		\begin{align}
			\inf\limits_{\boldsymbol{\tilde{v}}\in \tilde{R}_{h,m}(\omega^*)}
			\lnorm{\Sigma_h \boldsymbol{\tilde{u}} - \Sigma_h \boldsymbol{\tilde{v}}}{\omega^*}\leq Cm^{-1/d}|\omega^*|^{1/d}|\boldsymbol{\tilde{u}}|_{H^1(\omega^{\ast})},
		\end{align}
		 where $C>0$ depends on the shape of $\omega^{\ast}$ and the shape-regularity of the mesh $\mathcal{T}_{h}$.
	\end{lemma}
	\begin{proof}
	For convenience, we assume that $|\cdot|_{H^{1}(\omega^{\ast})}$ is a norm on $\boldsymbol{\tilde{W}}_h(\omega^*)$. Otherwise we can remove the kernel space similarly as in \cref{l2_approx_vecpot}. Define the linear operator 
\begin{equation*}
L=\Sigma_h|_{\boldsymbol{\tilde{W}}_h(\omega^*)}
			: \big(\boldsymbol{\tilde{W}}_h(\omega^*),|\cdot|_{H^{1}(\omega^{\ast})}\big) \to \big((L^2(\omega^*))^3, \lnorm{\cdot}{\omega^*}\big).    
\end{equation*}
  Clearly, $L$ is compact because it has finite dimensional rank. We denote the Kolmogorov width of $L$ by $d_m(L)$. To estimate $d_m(L)$, we need two auxiliary results. First, we note that there exists an $m$-dimensional space $Y(m)\subset (L^2(\omega^*))^3$ (e.g. finite element space) such that
		\begin{align}\label{L2_error_estimate}
			\inf _{\boldsymbol{\tilde{v}} \in Y(m)} \lnorm{\boldsymbol{\tilde{u}}-\boldsymbol{\tilde{v}}}{ \omega^*} \leq Cm^{-1/d}|\omega^*|^{1/d} |\boldsymbol{\tilde{u}}|_{H^{1}(\omega^{*})} \quad \forall \boldsymbol{\tilde{u}} \in (H^1\left(\omega^{*}\right))^3. 
		\end{align}
Second, using the fact that $\nabla \times \boldsymbol{\tilde{W}}_h(\omega^*) \subset \boldsymbol{V}_h(\omega^*)$ and  \cite[Theorem 5.41]{monk2003finite}, we have the following $L^{2}$ error estimate for $\Sigma_{h}$:
\begin{equation}\label{L2-interpolation-error}
\lnorm{\Sigma_h \boldsymbol{\tilde{u}} - \boldsymbol{\tilde{u}}}{\omega^*} \leq Ch |\boldsymbol{\tilde{u}}|_{H^{1}(\omega^{*})}\quad \forall \boldsymbol{\tilde{u}}\in \boldsymbol{\tilde{W}}_h(\omega^*).  
\end{equation}
Now we are ready to estimate $d_m(L)$. Using (\ref{L2_error_estimate}) and (\ref{L2-interpolation-error}) gives 
		\begin{align}
			\begin{split}
				d_{m}(L) 
				&\leq \sup\limits_{\boldsymbol{\tilde{u}}\in \boldsymbol{\tilde{W}}_h(\omega^*)}
				\inf\limits_{\boldsymbol{\tilde{v}}\in Y(m)}
				\frac{\lnorm{\Sigma_h \boldsymbol{\tilde{u}} - \boldsymbol{\tilde{v}}}{\omega^*}}{|\boldsymbol{\tilde{u}}|_{H^{1}(\omega^{*})}}\\
				&\leq \sup\limits_{\boldsymbol{\tilde{u}}\in \boldsymbol{\tilde{W}}_h(\omega^*)}
				\inf\limits_{\boldsymbol{\tilde{v}}\in Y(m)}
				\frac{\lnorm{\Sigma_h \boldsymbol{\tilde{u}} - \boldsymbol{\tilde{u}}}{\omega^*}}{|\boldsymbol{\tilde{u}}|_{H^{1}(\omega^{*})}}
				+\frac{\lnorm{\boldsymbol{\tilde{u}} - \boldsymbol{\tilde{v}}}{\omega^*}}{|\boldsymbol{\tilde{u}}|_{H^{1}(\omega^{*})}}\\
				&\leq C h +Cm^{-1/d}|\omega^*|^{1/d} \leq Cm^{-1/d}|\omega^*|^{1/d},
			\end{split}
		\end{align}
		where we have used the fact that $h\leq Cm^{-1/d}|\omega^*|^{1/d}$ since $m$ would not be larger than the dimension of $\boldsymbol{V}_h(\omega^*)$. Finally, we define $\tilde{R}_{h,m}(\omega^*)$ to be the space spanned by the first $m$ right singular vectors of $L$, and the desired result follows from the characterization of the Kolmogrov $n$-width. 
	\end{proof}

Combining \cref{CaccioDisc,l2_approx_vecpotDisc} and using (\ref{commuting-property}), we can prove a similar intermediate approximation result as \cref{lemma_induction_start} with analogous arguments.
	\begin{lemma}
		\label{lemma_induction_startDisc}
		Let $\omega\subset\omega^*$ be subdomains of $\Omega$ with $\delta:=\operatorname{dist}(\omega,\partial\omega^*\setminus\partial\Omega)>0$, and let $\max\limits_{K\cap\omega^*\neq\emptyset}h_K\leq \textcolor{black}{\delta}$. There exists an $m$-dimensional space $R_{m,h}(\omega^*)\subset \boldsymbol{W}_h(\omega^*)$ such that for all ${\boldsymbol{u}_h}\in \boldsymbol{W}_h(\omega^*)$, 
		\begin{equation*} 
			\inf_{\boldsymbol{v}_h\in R_{h,m}(\omega^*)}\lnorm{{\boldsymbol{u}_h}-\boldsymbol{v}_h}{\omega;a}
			\leq Cm^{-1/d}|\omega^*|^{1/d}\delta^{-1} \lnorm{\boldsymbol{u}_h}{\omega^*;a},
		\end{equation*}
		where $C$ depends on $\alpha_0, \alpha_1$, the shape-regularity of the mesh,and the shape of $\omega^{\ast}$. 
	\end{lemma}

	\begin{proof}[Proof of \cref{theorem_nearly_exponential_disc}]
	For $d=2$, we have an equivalent definition of the $n$-width $d_{h,n}(\omega_i,\omega_i^{\ast})$ as in the continuous case (see \cref{stream-functions-formulation}) by using \cref{lem:discrete_stream_function}. The exponential decay then follows from \cite{ma2022error}. For $d=3$, the proof is identical to that in the continuous case by using \cref{lemma_induction_startDisc}.

	\end{proof}

%
\section{Numerical Experiments}
\label{sec:num}

This section presents numerical experiments to validate the effectiveness of the proposed method and to confirm our theoretical analysis. In the considered examples, the computational domain $\Omega$ is the unit square $(0,1)^2$, and the fine FE mesh with mesh-size $h$ is defined on a uniform Cartesian grid. The domain is first partitioned into $M=m^2$ nonoverlapping subdomains $\{\omega_i^0\}_{i=1}^M$, where $m$ is the number of subdomains in each direction. These domains are then extended by two layers of fine mesh elements to construct the overlapping subdomains $\{\omega_i\}_{i=1}^M$, and the oversampling domains $\omega_i^*$ are obtained by adding $\ell$ layers of fine mesh elements to $\omega_i$. We construct the MS-GFEM solution $(\boldsymbol{u}_h^{\mathrm{G}}, p_h^{\mathrm{G}})$ by selecting $n_{loc}$ eigenfunctions per subdomain for building the local velocity spaces. In addition, homogeneous Dirichlet boundary conditions are imposed in the interior boundary to construct the local particular functions. We denote by $(\boldsymbol{u}^{\mathrm{e}}_h,p^{\mathrm{e}}_h)$ the fine-scale solution computed using the lowest-order Raviart-Thomas elements, and define the relative errors by 
\begin{equation}\label{relative-errors}
	\textbf{error}_v := \frac{\norm{\boldsymbol{u}_h^{\mathrm{G}}-\boldsymbol{u}^{\mathrm{e}}_h}{L^2(\Omega;a)}}{\lnorm{\boldsymbol{u}^{\mathrm{e}}_h}{\Omega;a}}
	,\qquad 
	\textbf{error}_p := \frac{\lnorm{p_h^{\mathrm{G}}-p^{\mathrm{e}}_h}{\Omega}}{\lnorm{p^{\mathrm{e}}_h}{\Omega}}.
\end{equation}
We only consider homogeneous Neumann boundary conditions $g_N=0$. The implementation is based on \cite{weinberg2018fast}.
\paragraph{Example 1}
We first consider problem (\ref{dualStrongForm}) with a constant permeability coefficient $A=1$ and a source term $f(x,y) = 2  \pi^2  cos(\pi x)  cos(\pi y)$. We choose $h=1/100$, and take $\gamma = 1$ for the augmentation parameter. In \Cref{fig:example1_relerr}, we show the convergence of the mixed MS-GFEM method to the fine-scale solution $(\boldsymbol{u}^{\mathrm{e}}_h,p^{\mathrm{e}}_h)$. Clearly, we see that both the velocity and pressure errors decay exponentially with respect to the number of local bases $n_{loc}$ as predicted by the theory, and the decay rates \ma{improve} with increasing oversampling sizes.
\begin{figure}
	\centering
	\subfloat{{\includegraphics[scale=0.4]{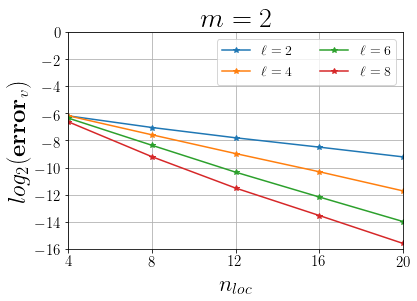} }}%
	\subfloat{{\includegraphics[scale=0.4]{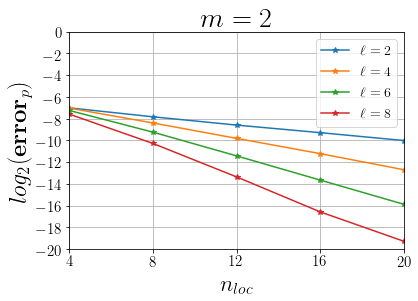} }}%
	\caption{Example 1: $\log_{2}(\textbf{error}_v)$ (left) and $\log_{2}(\textbf{error}_p)$ (right) for different numbers of local bases $n_{loc}$ and oversampling layers $\ell$. The oversampling size is $\ell h$.}
	\label{fig:example1_relerr}%
\end{figure}

\paragraph{Example 2}
This example is taken from \cite{chen2016least}. Here, we consider a high contrast permeability coefficient $A$, which is visualized in \Cref{fig:example2_A} (left), and take $\gamma = \alpha_1^{-1} = 10^{-3}$ and a finer underlying mesh with $h=1/240$. The source
term $f$ is the same as in Example 1. \Cref{fig:example2_nloc_l} shows the relative errors between the mixed MS-GFEM solution $(\boldsymbol{u}_h^{\mathrm{G}}, p_h^{\mathrm{G}})$ and the fine-scale solution $(\boldsymbol{u}^{\mathrm{e}}_h,p^{\mathrm{e}}_h)$ for different numbers of local bases and oversampling layers. We observe that even with a high contrast coefficient, numerical results of our method still agree well with theoretical predictions. Recall that we enriched the velocity multiscale space for proving inf-sup stability for the method. One might ask whether the enrichment is an artifact of the analysis or it is really necessary for the method to perform well. To test this, we omit the enrichment space in the construction of the multiscale space, and display the resulting errors in \Cref{fig:example2_nloc_l_Ven}. The results clearly show that the method without enrichment is highly unstable and inaccurate. Therefore, the enrichment of the velocity space is indeed a practical necessity. 
\begin{figure}
	\captionsetup[subfloat]{labelformat=empty}
	\centering
	\subfloat[\centering $A$]{{\includegraphics[scale=0.137]{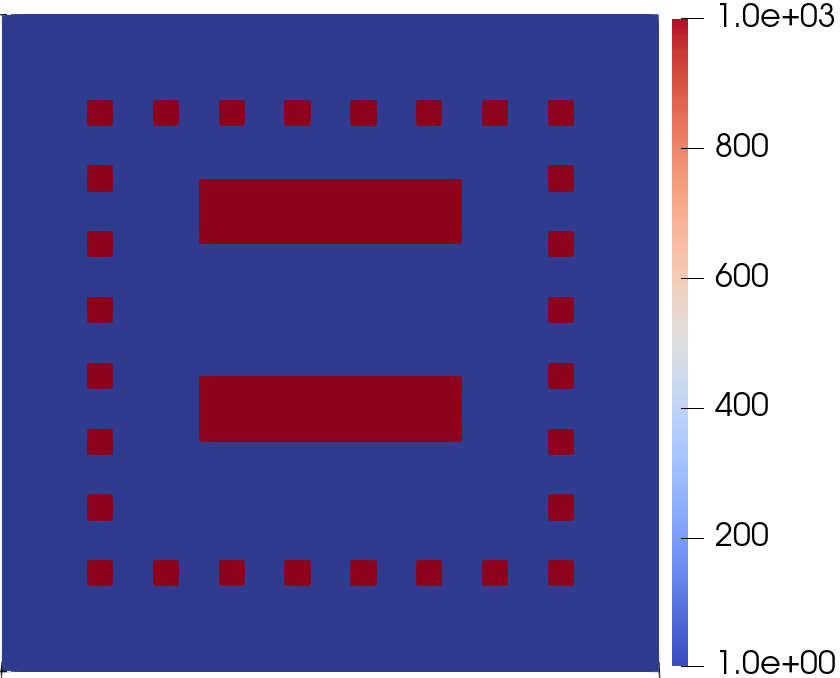} }}%
	\subfloat[\centering $log_{10}(A)$]{{\includegraphics[scale=0.165]{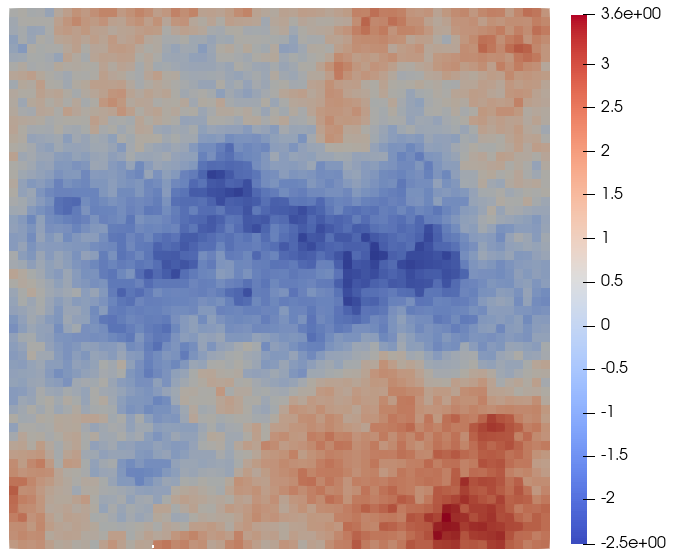} }}%
    \subfloat[\centering $\boldsymbol{u}_h^G$]{{\includegraphics[scale=0.11]{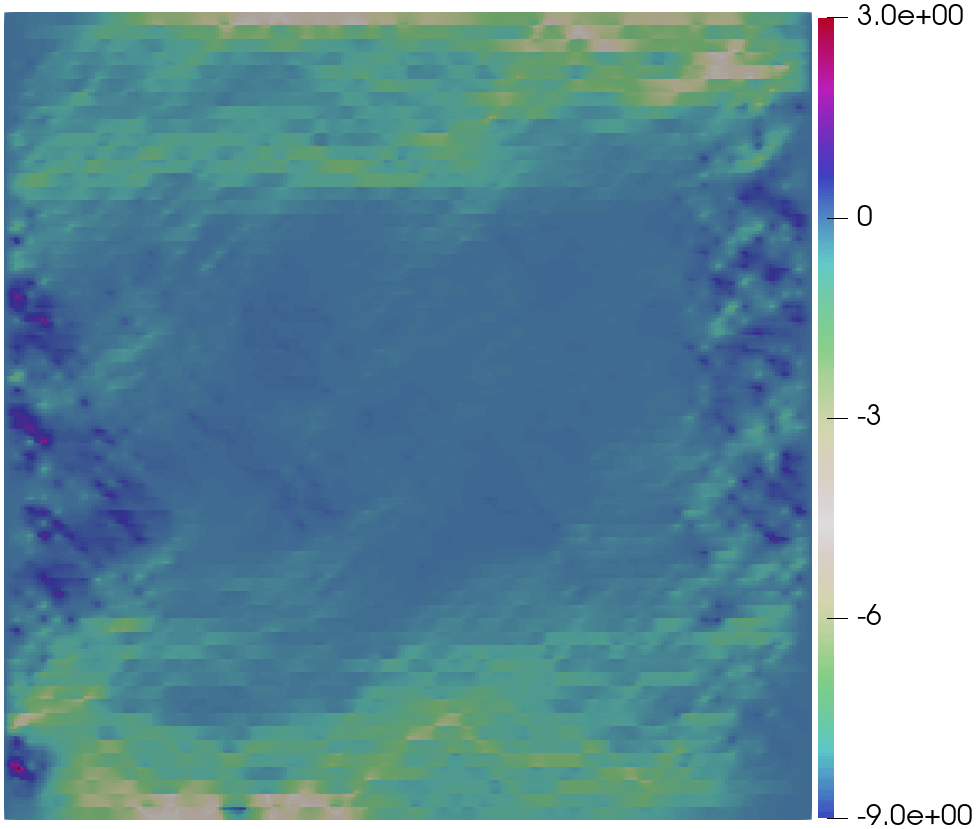} }}%
	\caption{Coefficient $A$ for example 2 (left) and 3 (middle). The $x$-component of $\boldsymbol{u}_h^G$ (right) for example 3 with $m=6, l = 15, n_{loc}=6$. }
	\label{fig:example2_A}%
\end{figure}
 
\begin{figure} 
	\centering
	\subfloat{{\includegraphics[scale=0.40]{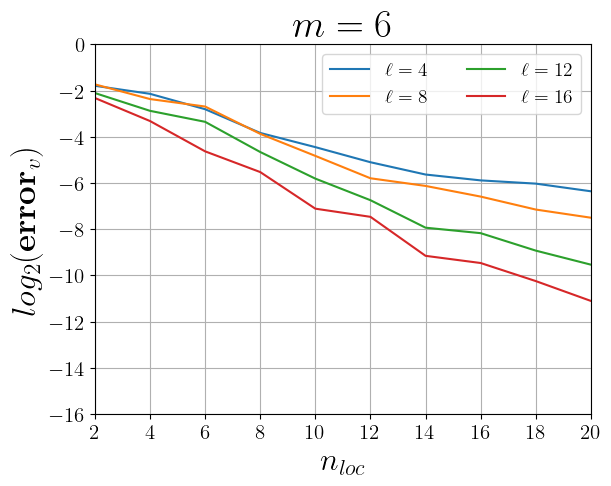} }}%
	\quad
	\subfloat{{\includegraphics[scale=0.40]{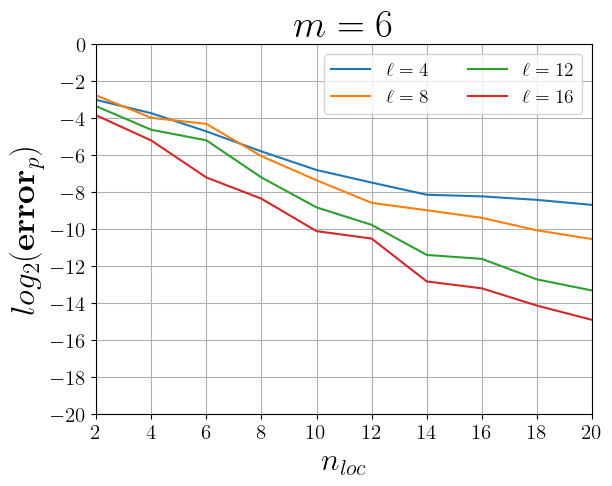} }}%
	\caption{Example 2: $\log_{2}(\textbf{error}_v)$ (left) and $\log_{2}(\textbf{error}_p)$ (right) for $m=6$.}
	\label{fig:example2_nloc_l}%
\end{figure}

\begin{figure}
	\centering
	\subfloat{{\includegraphics[scale=0.40]{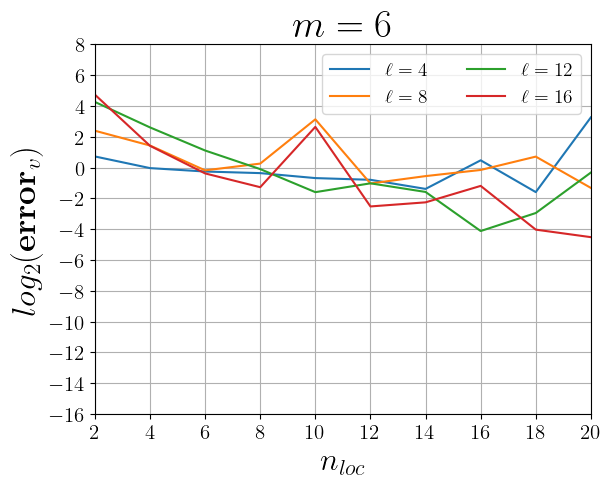} }}%
	\quad
	\subfloat{{\includegraphics[scale=0.40]{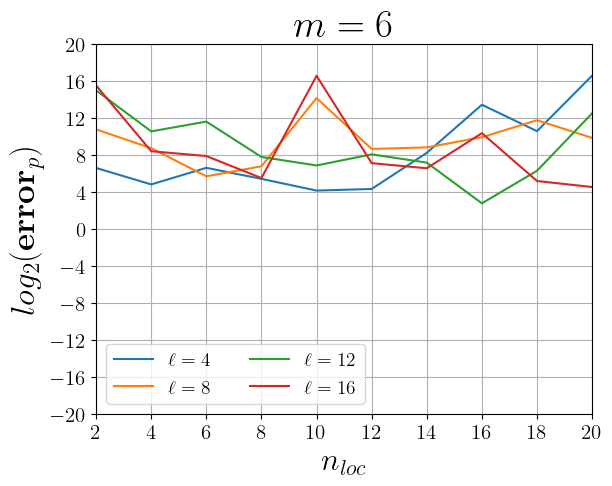} }}%
	\caption{Example 2: $\log_{2}(\textbf{error}_v)$ (left) and $\log_{2}(\textbf{error}_p)$ (right) without enrichment.}
	\label{fig:example2_nloc_l_Ven}%
\end{figure}
\paragraph{Example 3}
Finally, we consider a more realistic coefficient $A$ as illustrated in \Cref{fig:example2_A} (right), a part of the top layer of the SPE10 benchmark. We take $h=1/240$, and set the source term as
\begin{equation*}
	f(x,y)=
	\left\{\begin{array}{lll}
		100 \quad &\text{if } x < 0.1 \text{ and } y < 0.1  , \\
		-100 \quad &\text{if } x > 0.9\text{ and }  y > 0.9, \\
		0 \quad &\text{otherwise.}\\
	\end{array}\right.
\end{equation*}
For the augmentation parameter, we choose $\gamma = \alpha_1^{-1}$. \Cref{fig:example3_nloc_l} shows the decay of the error with respect to the number of local bases and oversampling size. To test how the method behaves with respect to different augmentation parameters, we vary $\gamma$ from $10^{-6}$ to $10^{3}$ and plot the relative errors in \cref{fig:example3_gamma_nloc}. It can be seen that the $L^{2}$ errors of the velocity and pressure decrease for smaller $\gamma$ within a certain range, and they remain basically unchanged for $\gamma$ small or large enough. On the other hand, the $L^{2}$ error of the divergence of the velocity ($\textbf{error}_{div}$, defined similarly to (\ref{relative-errors})) decreases slightly for larger $\gamma$ less than $10^{-1}$. These results show that in practice our method can work well without augmentation.

To demonstrate the efficiency of our method, we compute the mixed MS-GFEM solution $(\boldsymbol{u}_h^{\mathrm{G}},p_h^{\mathrm{G}})$ with $m=6$, $l=15$, and $n_{loc}=6$ (see \cref{fig:example2_A}). The relative errors $\textbf{error}_v$ and $\textbf{error}_p$ are $1.68\%$ and $0.24\%$, respectively, and the linear systems for computing $(\boldsymbol{u}_{h}^{\mathrm{e}}, p_{h}^{\mathrm{e}})$ and $(\boldsymbol{u}_h^{\mathrm{G}},p_h^{\mathrm{G}})$ involve $173280$ and $744$ degrees of freedom, respectively. Hence, the mixed MS-GFEM gives a good approximation with the number of degrees of freedom reduced by a factor of more than $200$.

\begin{figure}
	\centering
	\subfloat{{\includegraphics[scale=0.40]{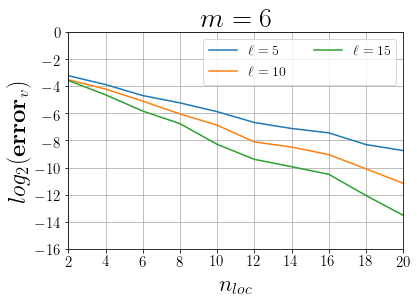} }}%
	\quad
	\subfloat{{\includegraphics[scale=0.40]{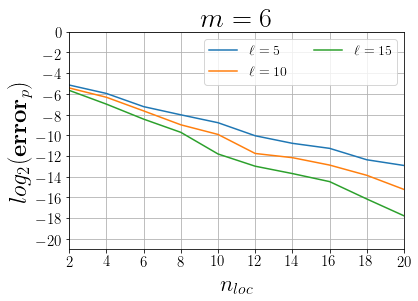} }}%
	\caption{Example 3: $\log_{2}(\textbf{error}_v)$ (left) and $\log_{2}(\textbf{error}_p)$ (right) for $m=6$.}
	\label{fig:example3_nloc_l}%
\end{figure}
\begin{figure}
	\centering
	\subfloat{{\includegraphics[scale=0.25]{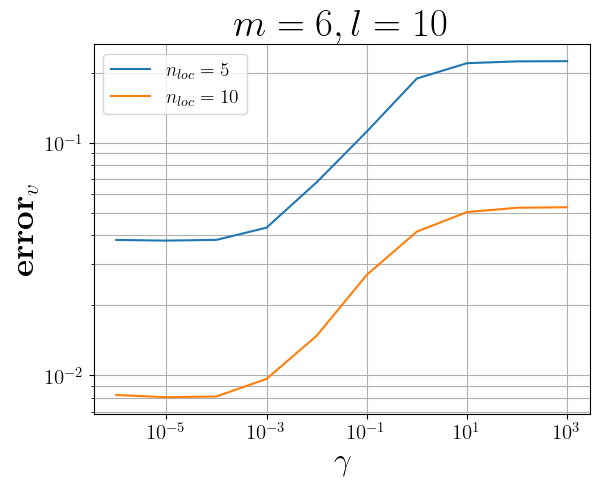} }}%
	\subfloat{{\includegraphics[scale=0.25]{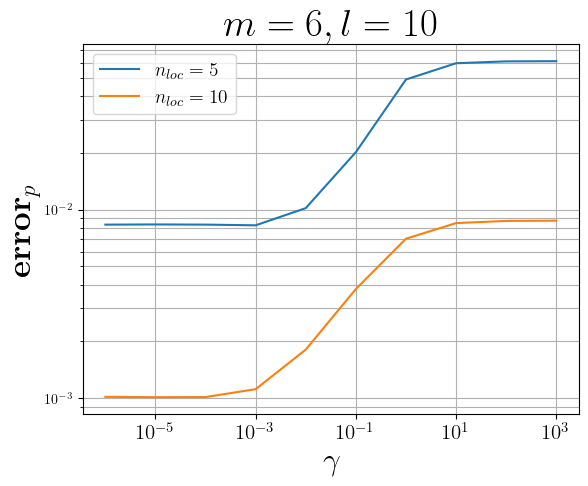} }}%
        \subfloat{{\includegraphics[scale=0.30]{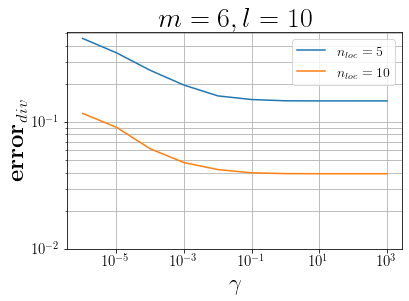} }}%
	\caption{Example 3: $\textbf{error}_v$ (left), $\textbf{error}_p$ (middle), and $\textbf{error}_{div}$ for different augmentation parameters $\gamma$.}
	\label{fig:example3_gamma_nloc}%
\end{figure}

\section{Conclusions}
We have developed a novel multiscale mixed finite element method for solving continuous and discretized elliptic problems with rough coefficients. Two important issues concerning the method -- approximation and stability -- were addressed in this paper. First, we constructed local approximation spaces with exponential accuracy by following the framework of the MS-GFEM. Second, following the general theory of mixed finite elements, we built trial spaces that are inf-sup stable by suitably enriching the global approximation spaces arising in the MS-GFEM. We performed a detailed analysis of our method in the case of general $L^{\infty}$-coefficients that shows it is stable and exponentially convergent. 

There are several aspects that have not been addressed in this paper and need to be investigated
in the future. While numerical results indicate that the method works without modifying the variational formulation, i.e., \ma{for} $\gamma = 0$, our analysis does not cover this case. It remains to prove convergence of our method in this case. Furthermore, from a practical point of view, the local eigenproblems
in our method are quite complicated due to the presence of the harmonic and divergence free constraints. An efficient way of solving these eigenproblems in the three-dimensional case is needed. \ma{Nevertheless, they are small and local}. Another important issue is an adaptive selection of the \ma{number of} local basis functions, which is essential for practical applications.

\appendix

\section{Vector Potentials and Stream Functions}
\label{sec:appendix}

\begin{lemma}
	\label{exVecPotnew}
Let $\Omega\subset \mathbb{R}^d$ ($d=2,3$) be a bounded Lipschitz domain such that $\mathbb{R}^d\setminus\overline{\Omega}$ is connected. Moreover, let $\Gamma:=\partial \Omega= \bigcup_{k=1}^K \Gamma_k$, $K\in\mathbb{N}$, with disjoint, relatively open and simply connected  Lipschitz surface patches $\Gamma_k\subset \partial\Omega$ satisfying $\operatorname{dist}(\Gamma_k,\Gamma_l)>0$ for all $1\leq k\neq l\leq K$. 
	\begin{itemize}
		\item [(i)] Assume that $d=2$. Then, there exists a linear and continuous operator $\mathcal{S}_{\Gamma}: \hdi{\Gamma}{0}{\Omega} \longrightarrow H_{\Gamma}^1(\Omega)$ such that
		\begin{equation*}
			\cuScal \,\mathcal{S}_{\Gamma}\boldsymbol{u} = \boldsymbol{u},\quad \text{ for all } \boldsymbol{u}\in\hdi{\Gamma}{0}{\Omega}.
		\end{equation*}
		
		\item [(ii)] Let $d=3$. Then, there exists a linear and continuous operator 
	\begin{equation*}
		\mathcal{S}_{\Gamma} : \hdi{\Gamma}{0}{\Omega} 
		\longrightarrow (H_{\Gamma}^1(\Omega))^3\cap (H^1(\mathbb{R}^3))^3
	\end{equation*}
	with 
	\begin{equation*}
		\cu \mathcal{S}_{\Gamma}\boldsymbol{u} = \boldsymbol{u},\quad \text{ for all } \boldsymbol{u}\in\hdi{\Gamma}{0}{\Omega}.
	\end{equation*}
	
	\end{itemize}
\end{lemma}
\begin{proof}
	The proof for the case $d=3$ can be found in \cite[Remark 5.4 (ii)]{bauer2016maxwell} and the case $d=2$ works with the same arguments. 
\end{proof}


\begin{lemma}
	\label{lem:discrete_stream_function}
	Let $d=2$ and $\boldsymbol{u}_h \in \boldsymbol{K}_{h,N}(\omega^*)$. Then, there exists ${\phi}_h\in\mathcal{P}_1(\Omega,\mathcal{T}_h)\cap H_N^1(\omega^*)$ such that $\boldsymbol{u}_h = \cuScal\, {\phi}_h.$ Here, $\mathcal{P}_1(\Omega,\mathcal{T}_h)$ denotes the space of piecewise linear function with respect to the mesh $\mathcal{T}_h$. 
\end{lemma}
\begin{proof}
	This is a consequence of \cite[Theorem 3.7]{scheichl2000iterative}.
\end{proof}

\printbibliography 

@Article{chung2015mixed,
  author    = {Chung, Eric T and Efendiev, Yalchin and Lee, Chak Shing},
  journal   = {Multiscale Modeling \& Simulation},
  title     = {Mixed generalized multiscale finite element methods and applications},
  year      = {2015},
  number    = {1},
  pages     = {338--366},
  volume    = {13},
  publisher = {SIAM},
}

@Book{boffi2013mixed,
  author    = {Boffi, Daniele and Brezzi, Franco and Fortin, Michel and others},
  publisher = {Springer},
  title     = {Mixed finite element methods and applications},
  year      = {2013},
  volume    = {44},
}

@Book{monk2003finite,
  author    = {Monk, Peter},
  publisher = {Oxford University Press},
  title     = {Finite element methods for Maxwell's equations},
  year      = {2003},
}

@Article{ma2021novel,
  author  = {Ma, Chupeng and Scheichl, Robert and Dodwell, Tim},
  journal = {arXiv preprint arXiv:2103.09545},
  title   = {Novel design and analysis of generalized FE methods based on locally optimal spectral approximations},
  year    = {2021},
}

@Article{ma2021wavenumber,
  author  = {Ma, Chupeng and Alber, Christian and Scheichl, Robert},
  journal = {arXiv preprint arXiv:2112.10544},
  title   = {Wavenumber explicit convergence of a multiscale GFEM for heterogeneous Helmholtz problems},
  year    = {2021},
}

@Book{girault2012finite,
  author    = {Girault, Vivette and Raviart, Pierre-Arnaud},
  publisher = {Springer Science \& Business Media},
  title     = {Finite element methods for Navier-Stokes equations: theory and algorithms},
  year      = {2012},
  volume    = {5},
}

@Book{pinkus2012n,
  author    = {Pinkus, Allan},
  publisher = {Springer Science \& Business Media},
  title     = {N-widths in Approximation Theory},
  year      = {2012},
  volume    = {7},
}

@Article{chen2003mixed,
  author  = {Chen, Zhiming and Hou, Thomas},
  journal = {Mathematics of Computation},
  title   = {A mixed multiscale finite element method for elliptic problems with oscillating coefficients},
  year    = {2003},
  number  = {242},
  pages   = {541--576},
  volume  = {72},
}

@Article{chen2016least,
  author    = {Chen, Fuchen and Chung, Eric and Jiang, Lijian},
  journal   = {Computer Methods in Applied Mechanics and Engineering},
  title     = {Least-squares mixed generalized multiscale finite element method},
  year      = {2016},
  pages     = {764--787},
  volume    = {311},
  publisher = {Elsevier},
}

@Article{bauer2016maxwell,
  author    = {Bauer, Sebastian and Pauly, Dirk and Schomburg, Michael},
  journal   = {SIAM Journal on Mathematical Analysis},
  title     = {The Maxwell compactness property in bounded weak Lipschitz domains with mixed boundary conditions},
  year      = {2016},
  number    = {4},
  pages     = {2912--2943},
  volume    = {48},
  publisher = {SIAM},
}

@Article{babuvska2020multiscale,
  author    = {Babu{\v{s}}ka, Ivo and Lipton, Robert and Sinz, Paul and Stuebner, Michael},
  journal   = {Computer Methods in Applied Mechanics and Engineering},
  title     = {Multiscale-spectral GFEM and optimal oversampling},
  year      = {2020},
  pages     = {112960},
  volume    = {364},
  publisher = {Elsevier},
}

@Article{babuska2011optimal,
  author    = {Babuska, Ivo and Lipton, Robert},
  journal   = {Multiscale Modeling \& Simulation},
  title     = {Optimal local approximation spaces for generalized finite element methods with application to multiscale problems},
  year      = {2011},
  number    = {1},
  pages     = {373--406},
  volume    = {9},
  publisher = {SIAM},
}

@Article{hou1997multiscale,
  author    = {Hou, Thomas Y and Wu, Xiao-Hui},
  journal   = {Journal of computational physics},
  title     = {A multiscale finite element method for elliptic problems in composite materials and porous media},
  year      = {1997},
  number    = {1},
  pages     = {169--189},
  volume    = {134},
  publisher = {Elsevier},
}

@Article{chung2018constraint,
  author    = {Chung, Eric and Efendiev, Yalchin and Leung, Wing Tat},
  journal   = {Computational Geosciences},
  title     = {Constraint energy minimizing generalized multiscale finite element method in the mixed formulation},
  year      = {2018},
  number    = {3},
  pages     = {677--693},
  volume    = {22},
  publisher = {Springer},
}

@Article{maalqvist2014localization,
  author  = {M{\aa}lqvist, Axel and Peterseim, Daniel},
  journal = {Mathematics of Computation},
  title   = {Localization of elliptic multiscale problems},
  year    = {2014},
  number  = {290},
  pages   = {2583--2603},
  volume  = {83},
}

@Article{arbogast2006subgrid,
  author    = {Arbogast, Todd and Boyd, Kirsten J},
  journal   = {SIAM Journal on Numerical Analysis},
  title     = {Subgrid upscaling and mixed multiscale finite elements},
  year      = {2006},
  number    = {3},
  pages     = {1150--1171},
  volume    = {44},
  publisher = {SIAM},
}

@Article{efendiev2013generalized,
  author    = {Efendiev, Yalchin and Galvis, Juan and Hou, Thomas Y},
  journal   = {Journal of computational physics},
  title     = {Generalized multiscale finite element methods (GMsFEM)},
  year      = {2013},
  pages     = {116--135},
  volume    = {251},
  publisher = {Elsevier},
}

@Article{demlow2011local,
  author  = {Demlow, Alan and Guzm{\'a}n, Johnny and Schatz, Alfred},
  journal = {Mathematics of computation},
  title   = {Local energy estimates for the finite element method on sharply varying grids},
  year    = {2011},
  number  = {273},
  pages   = {1--9},
  volume  = {80},
}

@Book{Ern_2021,
  author    = {Alexandre Ern and Jean-Luc Guermond},
  publisher = {Springer International Publishing},
  title     = {Finite Elements I},
  year      = {2021},
  doi       = {10.1007/978-3-030-56341-7},
}

@Article{ma2021error,
  author  = {Ma, Chupeng and Scheichl, Robert},
  journal = {arXiv preprint arXiv:2107.09988},
  title   = {Error estimates for fully discrete generalized FEMs with locally optimal spectral approximations},
  year    = {2021},
}

@Article{durlofsky1994accuracy,
  author    = {Durlofsky, Louis J},
  journal   = {Water Resources Research},
  title     = {Accuracy of mixed and control volume finite element approximations to Darcy velocity and related quantities},
  year      = {1994},
  number    = {4},
  pages     = {965--973},
  volume    = {30},
  publisher = {Wiley Online Library},
}

@Book{scheichl2000iterative,
  author    = {Scheichl, Robert},
  publisher = {University of Bath (United Kingdom)},
  title     = {Iterative solution of saddle point problems using divergence-free finite elements with applications to groundwater flow},
  year      = {2000},
}

@Article{weinberg2018fast,
  author  = {Weinberg, Theodore},
  journal = {SIAM Undergraduate Research Online},
  title   = {Fast implementation of mixed RT0 finite elements in MATLAB},
  year    = {2018},
  volume  = {12},
}

@Article{schleuss2022optimal,
  author    = {Schleu{\ss}, Julia and Smetana, Kathrin},
  journal   = {Multiscale Modeling \& Simulation},
  title     = {Optimal local approximation spaces for parabolic problems},
  year      = {2022},
  number    = {1},
  pages     = {551--582},
  volume    = {20},
  publisher = {SIAM},
}

@Article{ma2022exponential,
  author  = {Ma, Chupeng and Melenk, Jens Markus},
  journal = {arXiv preprint arXiv:2209.01957},
  title   = {Exponential convergence of a generalized FEM for heterogeneous reaction-diffusion equations},
  year    = {2022},
}

@Article{arbogast2007multiscale,
  author    = {Arbogast, Todd and Pencheva, Gergina and Wheeler, Mary F and Yotov, Ivan},
  journal   = {Multiscale Modeling \& Simulation},
  title     = {A multiscale mortar mixed finite element method},
  year      = {2007},
  number    = {1},
  pages     = {319--346},
  volume    = {6},
  publisher = {SIAM},
}

@Article{aarnes2004use,
  title={On the use of a mixed multiscale finite element method for greaterflexibility and increased speed or improved accuracy in reservoir simulation},
  author={Aarnes, Jorg E},
  journal={Multiscale Modeling \& Simulation},
  volume={2},
  number={3},
  pages={421--439},
  year={2004},
  publisher={SIAM}
}

@Article{aarnes2006hierarchical,
  title={A hierarchical multiscale method for two-phase flow based upon mixed finite elements and nonuniform coarse grids},
  author={Aarnes, Jorg E and Krogstad, Stein and Lie, Knut-Andreas},
  journal={Multiscale Modeling \& Simulation},
  volume={5},
  number={2},
  pages={337--363},
  year={2006},
  publisher={SIAM}
}

@article{larson2009mixed,
  title={A mixed adaptive variational multiscale method with applications in oil reservoir simulation},
  author={Larson, Mats G and M{\aa}lqvist, Axel},
  journal={Mathematical Models and Methods in Applied Sciences},
  volume={19},
  number={07},
  pages={1017--1042},
  year={2009},
  publisher={World Scientific}
}

@article{maalqvist2011multiscale,
  title={Multiscale methods for elliptic problems},
  author={M{\aa}lqvist, Axel},
  journal={Multiscale Modeling \& Simulation},
  volume={9},
  number={3},
  pages={1064--1086},
  year={2011},
  publisher={SIAM}
}

@Article{duran2019multiscale,
  title={A multiscale hybrid method for Darcy’s problems using mixed finite element local solvers},
  author={Duran, Omar and Devloo, Philippe RB and Gomes, Sonia M and Valentin, Fr{\'e}d{\'e}ric},
  journal={Computer methods in applied mechanics and engineering},
  volume={354},
  pages={213--244},
  year={2019},
  publisher={Elsevier}
}

@Article{wang2021comparison,
  title={A comparison of mixed multiscale finite element methods for multiphase transport in highly heterogeneous media},
  author={Wang, Yiran and Chung, Eric and Fu, Shubin and Huang, Zhaoqin},
  journal={Water Resources Research},
  volume={57},
  number={5},
  pages={e2020WR028877},
  year={2021},
  publisher={Wiley Online Library}
}

@article{yang2019multiscale,
  title={Multiscale hybridizable discontinuous Galerkin method for flow simulations in highly heterogeneous media},
  author={Yang, Yanfang and Shi, Ke and Fu, Shubin},
  journal={Journal of Scientific Computing},
  volume={81},
  pages={1712--1731},
  year={2019},
  publisher={Springer}
}

@article{he2021generalized,
  title={Generalized multiscale approximation of a mixed finite element method with velocity elimination for Darcy flow in fractured porous media},
  author={He, Zhengkang and Chen, Huangxin and Chen, Jie and Chen, Zhangxin},
  journal={Computer Methods in Applied Mechanics and Engineering},
  volume={381},
  pages={113846},
  year={2021},
  publisher={Elsevier}
}

@article{guiraldello2018multiscale,
  title={The multiscale Robin coupled method for flows in porous media},
  author={Guiraldello, Rafael T and Ausas, Roberto F and Sousa, Fabricio S and Pereira, Felipe and Buscaglia, Gustavo C},
  journal={Journal of Computational Physics},
  volume={355},
  pages={1--21},
  year={2018},
  publisher={Elsevier}
}

@article{hughes1995multiscale,
  title={Multiscale phenomena: Green's functions, the Dirichlet-to-Neumann formulation, subgrid scale models, bubbles and the origins of stabilized methods},
  author={Hughes, Thomas JR},
  journal={Computer methods in applied mechanics and engineering},
  volume={127},
  number={1-4},
  pages={387--401},
  year={1995},
  publisher={Elsevier}
}

@article{whitaker1986flow,
  title={Flow in porous media I: A theoretical derivation of Darcy's law},
  author={Whitaker, Stephen},
  journal={Transport in porous media},
  volume={1},
  pages={3--25},
  year={1986},
  publisher={Springer}
}

@article{babuvska2014machine,
  title={Machine computation using the exponentially convergent multiscale spectral generalized finite element method},
  author={Babu{\v{s}}ka, Ivo and Huang, Xu and Lipton, Robert},
  journal={ESAIM: Mathematical Modelling and Numerical Analysis},
  volume={48},
  number={2},
  pages={493--515},
  year={2014},
  publisher={EDP Sciences}
}

@article{ma2022error,
  title={Error estimates for discrete generalized FEMs with locally optimal spectral approximations},
  author={Ma, Chupeng and Scheichl, Robert},
  journal={Mathematics of Computation},
  volume={91},
  number={338},
  pages={2539--2569},
  year={2022}
}

@article{ma2022novel,
  title={Novel design and analysis of generalized finite element methods based on locally optimal spectral approximations},
  author={Ma, Chupeng and Scheichl, Robert and Dodwell, Tim},
  journal={SIAM Journal on Numerical Analysis},
  volume={60},
  number={1},
  pages={244--273},
  year={2022},
  publisher={SIAM}
}

@article{ma2023wavenumber,
  title={Wavenumber Explicit Convergence of a Multiscale Generalized Finite Element Method for Heterogeneous Helmholtz Problems},
  author={Ma, Chupeng and Alber, Christian and Scheichl, Robert},
  journal={SIAM Journal on Numerical Analysis},
  volume={61},
  number={3},
  pages={1546--1584},
  year={2023},
  publisher={SIAM}
}

@article{benezech2022scalable,
  title={Scalable multiscale-spectral GFEM for composite aero-structures},
  author={B{\'e}n{\'e}zech, Jean and Seelinger, Linus and Bastian, Peter and Butler, Richard and Dodwell, Timothy and Ma, Chupeng and Scheichl, Robert},
  journal={arXiv preprint arXiv:2211.13893},
  year={2022}
}

@article{ma2023unified,
  title={A unified framework for multiscale spectral generalized FEMs and low-rank approximations to multiscale PDEs},
  author={Ma, Chupeng},
  journal={arXiv:2311.08761},
  year={2023}
}

@article{melenk1996partition,
  title={The partition of unity finite element method: basic theory and applications},
  author={Melenk, Jens Markus and Babu{\v{s}}ka, Ivo},
  journal={Computer methods in applied mechanics and engineering},
  volume={139},
  number={1-4},
  pages={289--314},
  year={1996},
  publisher={Elsevier}
}

@article{chen2003numerical,
  title={Numerical homogenization of well singularities in the flow transport through heterogeneous porous media},
  author={Chen, Zhiming and Yue, Xinye},
  journal={Multiscale Modeling \& Simulation},
  volume={1},
  number={2},
  pages={260--303},
  year={2003},
  publisher={SIAM}
}

@article{hellman2016multiscale,
  title={Multiscale mixed finite elements},
  author={Hellman, Fredrik and Henning, Patrick and M{\aa}lqvist, Axel},
  journal={Discrete and Continuous Dynamical Systems-S},
  volume={9},
  number={5},
  pages={1269--1298},
  year={2016},
  publisher={Discrete and Continuous Dynamical Systems-S}
}

\end{document}